\documentclass{amsart}

\usepackage{amssymb,amsmath,amsthm,amsfonts}
\usepackage{graphicx}
\usepackage{amsthm}
\usepackage{amsmath}
\usepackage{color}
\usepackage{mathabx}
\usepackage[all,cmtip]{xy}
\usepackage{tikz-cd}
\usepackage{stmaryrd}
\usepackage{hyperref}
\usepackage{algorithm}
\usepackage{algpseudocode}
\usepackage{pifont}
\usepackage[margin=1.5in]{geometry}
\usepackage[normalem]{ulem}

\newtheorem{thm}{Theorem}[section]
\newtheorem{prop}[thm]{Proposition}
\newtheorem{lem}[thm]{Lemma}
\newtheorem{cor}[thm]{Corollary}

\newtheorem{example}[thm]{Example}
\newtheorem{remark}[thm]{Remark}

\newcommand{\R}{\mathbb{R}}
\newcommand{\Z}{\mathbb{Z}}

\newcommand{\C}{\mathbb{C}}

\newcommand{\diffI}{\mathrm{Diff}^+(I)}

\begin{document}

\title[Transforms for general elastic metrics]{Simplifying transforms for general elastic metrics on the space of plane curves}

\author{Sebastian Kurtek and Tom Needham}

\maketitle

\begin{abstract}
In the shape analysis approach to computer vision problems, one treats shapes as points in an infinite-dimensional Riemannian manifold, thereby facilitating algorithms for statistical calculations such as geodesic distance between shapes and averaging of a collection of shapes. The performance of these algorithms depends heavily on the choice of the Riemannian metric. In the setting of plane curve shapes, attention has largely been focused on a two-parameter family of first order Sobolev metrics, referred to as elastic metrics. They are particularly useful due to the existence of simplifying coordinate transformations for particular parameter values, such as the well-known square-root velocity transform. In this paper, we extend the transformations appearing in the existing literature to a family of isometries, which take any elastic metric to the flat $L^2$ metric. We also extend the transforms to treat piecewise linear curves and demonstrate the existence of optimal matchings over the diffeomorphism group in this setting. We conclude the paper with multiple examples of shape geodesics for open and closed curves. We also show the benefits of our approach in a simple classification experiment.
\end{abstract}

\section{Introduction}

Shape is a fundamental physical property of objects, and plays an important role in various imaging tasks including identification and tracking. As a result, statistical analysis of shape plays a crucial role in many image-rich application domains including computer vision, medical imaging, biology, bioinformatics, geology, biometrics and others. In statistical shape analysis, shape is viewed as a random object, and one is concerned with developing methods to perform common statistical tasks including registration, comparison, averaging, summarization of variability, hypothesis testing, regression, and other inferential procedures. Any statistical shape analysis approach requires an appropriate shape representation and an associated metric that enables quantification of shape differences. Evidently, the quality of statistical analyses of shape data is heavily dependent on these choices.

There is a rich literature on statistical analysis of shape, with the most prominent shape representation being landmark-based. Landmarks constitute a finite collection of points that are chosen either by the application expert (anatomical landmarks) or according to some mathematical rule such as high absolute curvature (mathematical landmarks). Once the points are selected, the remaining information regarding the object's outline is discarded. Under this representation, D.G. Kendall \cite{Kendall} defined shape as a property of an object that is invariant to its rigid motions and global scaling; this approach is commonly referred to as similarity shape analysis. Since then, there has been continuous development of statistical tools to analyze similarity shapes represented by landmarks; see \cite{DrydenBook,small_boook:96} for a comprehensive set of methods. These approaches combine ideas from differential geometry, algebra and multivariate statistics to establish rigorous estimation and inferential procedures on the landmark shape space. The main benefit of these approaches is that the resulting shape space is finite dimensional, making statistical analysis ``easier." However, the obvious drawback is that the finite collection of landmarks used to represent shapes of interest results in significant loss of information.

Recently, there has been more emphasis on using a function-based representation of shape, i.e., objects are represented via their boundaries as parameterized curves. Thus, in this case, one must account for possible parameterization variability in addition to rigid motion and global scaling. One set of methods removes this variability by normalizing all parameterizations to arclength \cite{Klassen2004,Zahn}. However, such an approach is suboptimal in many real scientific problems due to a lack of appropriate registration. A better approach is to remove such variability in a pairwise manner using an appropriate metric. This is the idea behind elastic shape analysis, where a family of elastic metrics is used for joint registration and comparison. The resulting shape spaces are more complicated than their landmark counterparts, but the benefits of such approaches are clear: (1) there is no need to select landmarks which can be a tedious and expensive process, (2) the curve representation is able to encode all relevant shape information, and (3) the elastic metric quantifies intuitive shape deformations. Elastic shape analysis is the focus of the current paper, and we provide a formal mathematical setup for this approach in the following section.

\subsection{Elastic Shape Analysis}\label{sec:elastic_shape_analysis}

A fundamental ingredient in a theory of shape similarity for plane curves is a distance metric on the space of curves $\mathcal{S}$ which is invariant under rigid transformations of the curves. For a pair of plane curves $C_1$ and $C_2$, we therefore wish to assign a distance $d(C_1,C_2)$ such that $d(\xi_1 \star C_1, \xi_2 \star C_2) = d(C_1,C_2)$ for any elements $\xi_j$ of the Euclidean isometry group $\R^2 \ltimes \mathrm{SO}(2)$ acting in the natural way.

Under the elastic shape analysis paradigm, the distance function described above arises from a Riemannian metric on $\mathcal{S}$. This extra structure has obvious benefits over treating $\mathcal{S}$ only as a metric space---for example, statistical calculations such as Karcher means or principal component analysis can be performed in a  tangent space after pulling back via the logarithm map. The Riemannian structure is defined by representing shapes as elements of the quotient space $\mathcal{S}=\mathrm{Imm}(I,\R^2)/\mathrm{Diff}^+(I)$ ($I$ denotes a subinterval of the real line) of smooth immersions modulo reparameterizations, so that a choice $g$ of $\mathrm{Diff}^+(I)$-invariant metric on the relatively simple space $\mathrm{Imm}(I,\R^2)$ descends to a well-defined metric on $\mathcal{S}$. If the Riemannian metric is also invariant under Euclidean similarities then the geodesic distance with respect to this metric induces our desired distance function $d$, via the formula
$$
d(C_1,C_2) = \inf_{\gamma_1,\gamma_2 \in \mathrm{Diff}^+(I) \times (\R^2 \ltimes  \mathrm{SO}(2))} \mathrm{dist}_g(\gamma_1 \star c_1,\gamma_2 \star c_2).
$$
In this formula, the $c_j$ are arbitrary choices of parameterizations of $C_j$, $\mathrm{dist}_g$ denotes geodesic distance in $\mathrm{Imm}(I,\R^2)$ with respect to $g$, and $\star$ denotes the action of the group $\mathrm{Diff}^+(I) \times (\R^2 \ltimes  \mathrm{SO}(2))$ of shape-preserving transformations on $\mathrm{Imm}(I,\R^2)$.

The simplest choice of Riemannian metric is the reparameterization-invariant $L^2$-metric defined for $c \in \mathrm{Imm}(I,\R^2)$ and $h,k \in T_c \mathrm{Imm}(I,\R^2) \approx C^\infty(I,\R^2)$ by the formula
$$
\hat{g}^{L^2}_c (h,k) = \int_I \left<h,k\right> \; \mathrm{d}s.
$$
The nonlinearity of this metric lies in the \emph{measure with respect to arclength}  $\mathrm{d}s=|c'(t)|\;\mathrm{d}t$, which provides the desired $\mathrm{Diff}^+(I)$-invariance. It is a surprising fact that geodesic distance on the shape space vanishes with respect to $\hat{g}^{L^2}$ \cite{michor2005vanishing}, and one must therefore consider more complicated metrics on the space of immersions. Examples in the literature of such metrics  include almost-local (weighted $L^2$) metrics \cite{bauer2012almost,bauer2012curvature,michor2007overview} and higher order Sobolev-type metrics \cite{bauer2016use,charpiat2007shape,michor2007overview,sundaramoorthi2007sobolev}. A particularly well-studied subfamily of first-order Sobolev metrics are the \emph{elastic metrics} introduced in \cite{mio2007shape}. These are a two-parameter family of metrics $g^{a,b}$ defined as follows. For $c \in \mathrm{Imm}(I,\R^2)$, let $(T,N)$ denote the standard moving frame consisting of the unit tangent and unit normal to $c$, respectively. We use $D_s = \frac{1}{|c'(t)|} \frac{d}{dt}$ for \emph{derivative with respect to arclength}. For $a,b \neq 0$ and $h,k \in T_c \mathrm{Imm}(I,\R^2)$, we define
\begin{equation}\label{eqn:elastic_metric}
g^{a,b}_c(h,k)=\int_I  a^2 \left<D_s h, N\right>\left<D_s k, N\right> + b^2 \left<D_s h,T \right>\left<D_s k, T \right> \; \mathrm{d}s.
\end{equation}
This metric is invariant under reparameterizations and rigid motions, so descends to a well-defined metric on the shape space $\mathcal{S}$.

\subsection{Previous Work}

In order to compute the distance between curves in $\mathcal{S}$, our procedure requires the computation of geodesic distance with respect to the chosen metric. We will focus on elastic metrics $g^{a,b}$. A common technique is to apply the \emph{square-root velocity function} (SRVF) transform,  given by
\begin{align*}
\mathrm{Imm}(I,\R^2) &\rightarrow C^\infty(I,\R^2) \\
c &\mapsto \frac{c'}{|c'|^{1/2}}.
\end{align*}
The theoretical power of the SRVF is the remarkable fact that the pullback of the standard $L^2$ metric on the target space is the elastic metric $g^{1,1/2}$  \cite{joshi2007novel}, whence geodesics with respect to $g^{1,1/2}$ in $\mathrm{Imm}(I,\R^2)$ can be computed explicitly by pushing forward to the flat target space, computing geodesics there, then pulling the result back. Due to this convenient property, the SRVF transform has been studied extensively from a theoretical perspective \cite{bruveris2016optimal, lahiri2015precise, srivastava2011shape, tumpach2017quotient} and has seen a wide variety of applications, including classification of plant leaf shapes  \cite{laga2012riemannian}, statistical analysis of protein structures \cite{srivastava2016efficient} and biomedical imaging of anatomical features in the brain \cite{ayers2015corpus}. A similarly simple transform is introduced in \cite{younes2008metric}, where a plane curve $c$ is taken to the curve $\sqrt{c'}$, with the square-root taken pointwise by considering $c'$ to be a path in the complex plane; in this case, the map pulls back standard $L^2$ to $g^{1/2,1/2}$. A more complicated family of transforms $R_{a,b}$ is defined in \cite{bauer2014constructing} for $2b \geq a > 0$ and it is shown that the pullback by $R_{a,b}$ of the $L^2$ metric on its target is $g^{a,b}$.

In this paper, we define a two-parameter family of transforms $F_{a,b}:\mathrm{Imm}(I,\R^2) \rightarrow C^\infty(I,\R^2)$ which is valid for \emph{all} choices of $a,b > 0$. Our main result (Theorem \ref{thm:pullback}) is that $F_{a,b}$ pulls back the $L^2$ metric to the elastic metric $g^{a,b}$. Moreover, we show that $F_{a,b}$ subsumes the SRVF transform, the complex square-root transform, and the $R_{a,b}$-transforms. 

While preparing this manuscript, we learned of similar work being done concurrently by Bauer, Bruveris, Charon and Moeller-Andersen \cite{bauer2018relaxed}. Their work focuses on efficient numerical computations of geodesics for general Sobolev-type metrics, including elastic metrics. Their methods are disjoint from ours, and the two approaches should be complementary. 

\subsection{Outline of the Paper}

In Section \ref{sec:Fb_transform}, we define the transform $F_{a,b}$ and prove our main result. We also consider the important submanifold of closed plane curves, and more precisely compare our transform to those described in the previous subsection. Section \ref{sec:group_actions} describes how various shape-preserving group actions behave in $F_{a,b}$-coordinates. In Section \ref{sec:computing_geodesics}, we describe the explicit geodesics in the  curve spaces. Numerical implementation is formally treated in Section \ref{sec:PL_shapes}, where the transform is extended to treat piecewise linear (PL) curves. In particular, we show that optimal registrations over the diffeomorphism group are realized by PL reparameterizations in this setting. Finally, we provide numerical examples in Section \ref{sec:applications} and suggest future directions in Section \ref{sec:conclusion}.

\section{The $F_{a,b}$ Transform}\label{sec:Fb_transform}

\subsection{Shape Spaces of Open Curves}\label{sec:shape_spaces}

The space of plane curve shapes is obtained via a quotient construction, starting with the space of immersions
$$
\mathrm{Imm}(I,\R^2)=\{c \in C^\infty(I,\R^2) \mid |c'(t)|\neq 0 \; \forall \; t\in I\},
$$
where, without loss of generality, $I=[0,1]$ is fixed. To simplify calculations, we make the identification $\mathrm{Imm}(I,\R^2) \approx \mathrm{Imm}(I,\C)$. We note here that our choice of $C^\infty$ regularity is primarily a matter of convenience and that our results hold essentially without modification for $C^1$ curves. Decreasing regularity below $C^1$ does cause some theoretical issues, and these are treated in Section \ref{sec:PL_shapes}.

There are various shape-preserving Lie group actions which we will quotient by: $\R^2$ acting by translations, $\R_{>0}$ acting by scaling, $\mathrm{SO}(2)$ acting by rotations and $\mathrm{Diff}^+(I)$ acting by reparameterizations. The easiest action to deal with is translations, as there is an obvious isomorphism with the space of curves based at zero,
$$
\mathrm{Imm}(I,\C)/\R^2 = \mathrm{Imm}(I,\C)/\mbox{Tra} \approx \{c \in C^\infty(I,\C) \mid c'(t) \neq 0 \; \forall \; t\in I, \; c(0)=0\}.
$$
We take this identification as a convention in order to simplify notation. Most of our explicit calculations will take place in this space, which we refer to as the \emph{preshape space of curves}.

Our goal is to understand the following quotient space with the full set of shape similarities modded out:
$$
\mathrm{Imm}(I,\C)/(\R^2 \ltimes \R_{\geq 0} \times \mathrm{SO}(2) \times \mathrm{Diff}^+(I))=\mathrm{Imm}(I,\C)/\{\mbox{Tra, Sca, Rot, Rep}\}.
$$
We refer to this quotient space as the \emph{shape space of curves} and denote it by $\mathcal{S}$. Intermediate spaces such as $\mathrm{Imm}(I,\C)/\{\mathrm{Tra, Rot}\}$ will appear frequently and we will treat them separately as they arise.

\subsection{The $F_{a,b}$ Transform on Preshape Space}\label{sec:F_b_transform}

For any $a,b > 0$, we define the $F_{a,b}$-transform by the formula
 \begin{align*}
F_{a,b}:\mathrm{Imm}(I,\C)/\mbox{Tra} &\rightarrow C^\infty(I,\C^\ast) \\
c &\mapsto 2b |c'|^{1/2} \left(\frac{c'}{|c'|}\right)^{\frac{a}{2b}}.
\end{align*}
We use the notation $\C^\ast = \C \setminus \{0\}$. In the formula, all arithmetic operations are taken pointwise on complex numbers. One should immediately notice that, due to the presence of complex exponentiation, $F_{a,b}$ is not well-defined in general. Indeed, writing $c'$ in polar coordinates $r \exp(i\theta)$, the continuous argument function $\theta$ is only unique up to a global addition of an integer multiple of $2\pi$. The values of $F_{a,b}$ are then given by
$$
2br^{1/2} \exp\left(i( \theta + 2 k \pi) \frac{a}{2b}\right) = 2br^{1/2} \exp\left(i \theta \frac{a}{2b}\right) \cdot \exp\left(i\frac{a}{b}k\pi\right)
$$
for $k \in \Z$. Therefore $F_{a,b}$ is technically defined as a multivalued function with image set
\begin{equation}\label{eqn:image_set}
F_{a,b}(c) = \left\{q \cdot \exp\left(i\frac{a}{b} k \pi\right) \mid k \in \Z \right\},
\end{equation}
where $q$ is some arbitrary choice of image of $c$. In this form, it is easy to see that $F_{a,b}$ is a bijection if and only if $\frac{a}{2b} =1$. If $\frac{a}{2b}$ is an integer not equal to one then $F_{a,b}$ is well-defined but many-to-one. If $\frac{a}{2b}$ is not an integer then $F_{a,b}$ is multivalued, taking finitely many values if and only if $\frac{a}{2b}$ is rational.

For the sake of concreteness, we can locally define $F_{a,b}$ more precisely as follows. Let $c \in \mathrm{Imm}(I,\C)/\mathrm{Tra}$ and choose a polar coordinate representation of its derivative $c'=r\exp(i\theta)$ so that $\theta$ is continuous on $I$. The magnitude function $r$ is unique and such a choice of $\theta$ is unique up to addition of a multiple of $2\pi$. Moreover, any parameterized curve $\widetilde{c}$ which is $C^\infty$-close to $c$ has a polar representation $\widetilde{c}'=\widetilde{r}\exp(i\widetilde{\theta})$ so that $\theta$ and $\widetilde{\theta}$ are $C^\infty$-close.  The $F_{a,b}$-transform is then defined locally by
\begin{equation}\label{eqn:polar_rep}
F_{a,b}(c)=2br^{1/2}\exp\left(i \theta \frac{a}{2b}\right).
\end{equation}
This shows that $F_{a,b}$ can be represented locally as a well-defined smooth map of infinite-dimensional manifolds (in the sense of \cite{hamilton1982inverse}). Statements in this section regarding the differential properties of $F_{a,b}$ should be treated as statements about this local representation.

\begin{remark}
The formula \eqref{eqn:polar_rep} involves a choice of image of $F_{a,b}$. Fortunately, all other choices of image differ from this one by a rotation and we will see in Theorem \ref{thm:isometry} that this implies that $F_{a,b}$ descends to a well-defined map on quotient spaces of curves modulo rotation.
\end{remark}

\begin{remark}\label{rmk:simple_curves}
One issue that can arise in the elastic shape analysis approach to shape matching is that curves which are close in Hausdorff distance can be far apart in geodesic distance (see Figure \ref{fig:bad_example}). This may be undesirable, depending on the application. One can overcome this issue by restricting to simple curves or by using ad hoc methods to account for such differences.
\end{remark}

\begin{figure}[!t]
\includegraphics[scale=0.3]{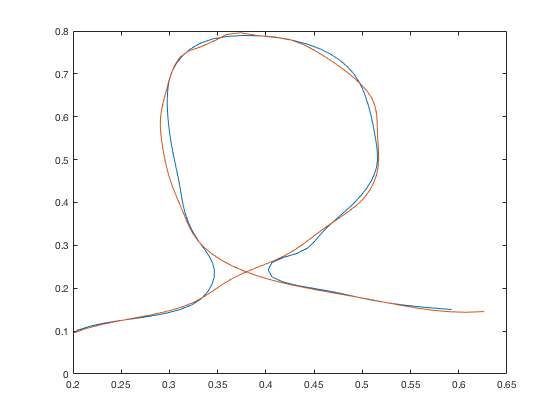} \includegraphics[scale=0.3]{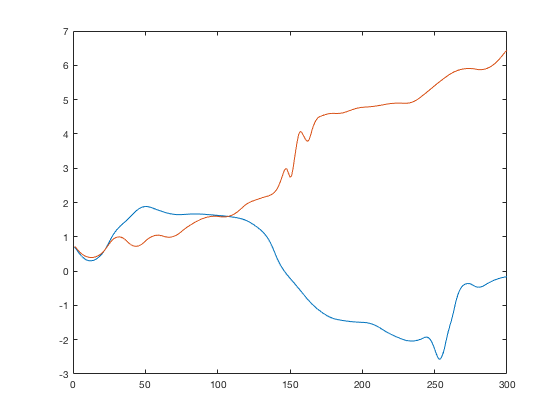}
\caption{Left: A pair of curves which are close in Hausdorff distance. Right: The continuous polar angle functions of the curves are quite different, resulting in images under $F_{a,b}$ which are far apart in Hausdorff distance.}\label{fig:bad_example}
\end{figure}

The transform $F_{a,b}$ is a local diffeomorphism with local inverse given by
\begin{align}\label{eqn:fb_inverse}
F_{a,b}^{-1}(q)(t)&=\frac{1}{4b^2} \int_0^t |q(\tau)|^2 \left(\frac{q(\tau)}{|q(\tau)|}\right)^{\frac{2b}{a}} \; \mathrm{d}\tau,
\end{align}
treated formally for non-integer $\frac{2b}{a}$ either as a multivalued or locally well-defined map.

\subsection{Pullback Metric}

Let $g^{L^2}$ denote the standard $L^2$ metric on the space $C^\infty(I,\C)$, defined at basepoint $q \in C^\infty(I,\C)$ on variations $w,z \in T_q C^\infty(I,\C) \approx C^\infty(I,\C)$ by the formula
$$
g^{L^2}_q(w,z) = \mathrm{Re} \int_I w \overline{z} \; \mathrm{d}t.
$$
This is a flat metric on the vector space $C^\infty(I,\C)$ which restricts to a flat metric on the open submanifold $C^\infty(I,\C^\ast)$. 

\begin{thm}\label{thm:pullback}
The $L^2$ metric $g^{L^2}$ on $C^\infty(I,\C^\ast)$ pulls back to the elastic metric $g^{a,b}$ under the transform $F_{a,b}$.
\end{thm}

\begin{proof}
Let $c \in \mathrm{Imm}(I,\C)$ and let $h$ be a tangent vector to $c$. We first note that, expressing the unit normal $N$ to $c$ as $iD_s c$ and the Euclidean inner product as the real part of $(z,w) \mapsto z \overline{w}$, the elastic metric $g^{a,b}$ can be written as
\begin{align}
g_c^{a,b}(h,k) &= \int_I  a^2 \left<D_s h, N\right>\left<D_s k, N\right> + b^2 \left<D_s h, T\right> \left<D_s k, T\right> \; \mathrm{d}s \nonumber \\
&= \int_I \left(a^2 \mathrm{Re}\left(\frac{1}{|c'|}h' \cdot  -i \frac{\overline{c'}}{|c'|} \right)\mathrm{Re}\left(\frac{1}{|c'|}k' \cdot  -i \frac{\overline{c'}}{|c'|} \right) \right. \nonumber \\
&\hspace{1in} \left.+ b^2 \mathrm{Re} \left(\frac{1}{|c'|}h' \cdot \frac{\overline{c'}}{|c'|}\right) \mathrm{Re} \left(\frac{1}{|c'|}k' \cdot \frac{\overline{c'}}{|c'|}\right)\right) \; |c'|\mathrm{d}t \nonumber \\
&= \int_I \frac{1}{|c'|^3} \left(a^2 \mathrm{Im}\left(c' \overline{h'}\right)  \mathrm{Im}\left(c' \overline{k'}\right) + b^2 \mathrm{Re}\left(c' \overline{h'}\right) \mathrm{Re}\left(c' \overline{k'}\right) \right) \; \mathrm{d}t. \label{eqn:pullback_1}
\end{align}
Using the formula
$$
\left.\frac{d}{d\epsilon}\right|_{\epsilon = 0} |c' + \epsilon h'| = \frac{\mathrm{Re}(c' \overline{h'})}{|c'|},
$$
it is straightforward to show that
\begin{align*}
DF_{a,b}(c)(h) &=\left.\frac{d}{d\epsilon}\right|_{\epsilon=0} 2b |c' + \epsilon h'|^{1/2} \left(\frac{c' + \epsilon h'}{|c' + \epsilon h'|}\right)^{\frac{a}{2b}}  \\
&= \left(\frac{c'}{|c'|}\right)^{\frac{a}{2b}}\left( b |c'|^{-1/2}  \frac{\mathrm{Re}(c' \overline{h'})}{|c'|} + a |c'|^{1/2} \frac{|c'|}{c'} \frac{|c'|h' - c' \mathrm{Re}(c'\overline{h'})/|c'|}{|c'|^2} \right) \\
&= \left(\frac{c'}{|c'|}\right)^{\frac{a}{2b}} |c'|^{-\frac{3}{2}} \left(b \mathrm{Re}\left(c' \overline{h'}\right) - i a \mathrm{Im}\left(c' \overline{h'}\right)\right).
\end{align*}
Then the pullback metric is given by
\begin{align*}
\left(F_{a,b}^\ast g^{L^2}\right)_c(h,k) &= \int_I \mathrm{Re} \; DF_{a,b}(c)(h) \cdot \overline{DF_{a,b}(c)(k)} \; \mathrm{d}t \\
&= \int_I |c'|^{-3}  \mathrm{Re} \; \left(b \mathrm{Re}\left(c' \overline{h'}\right) - i a \mathrm{Im}\left(c'\overline{h'}\right)\right) \cdot \left(b \mathrm{Re}\left(c' \overline{k'}\right) + i a \mathrm{Im}\left(c'\overline{k'}\right)\right) \; \mathrm{d}t,
\end{align*}
which easily simplifies to \eqref{eqn:pullback_1}.
\end{proof}

\noindent We will show in Section \ref{sec:previous_work} that this theorem is a direct generalization of results appearing in \cite{bauer2014constructing,joshi2007novel,younes2008metric}.

\subsection{The Preshape Space of Closed Curves}\label{sec:closed_curves}

We now consider the preshape space of closed loops $\mathrm{Imm}(S^1,\C)/\mathrm{Tra}$ (i.e., the space of ``object outlines"). By identifying $S^1$ with the quotient $[0,1]/(0\sim 1)$, we can consider the preshape space of closed curves to be a submanifold of the preshape space of open curves of infinite codimension. Under this identification, the $F_{a,b}$-transform can be restricted to $\mathrm{Imm}(S^1,\C)/\mathrm{Tra}$ and the image of the restricted map will lie in $C^\infty(I,\C^\ast)$. We wish to characterize the image of the restricted $F_{a,b}$-transform. Using the polar form \eqref{eqn:polar_rep} of $F_{a,b}$, we see that for any closed curve $c$ with $c'=r\exp(i\theta)$,
\begin{align*}
F_{a,b}(c)(1) &= 2b\sqrt{r(1)}\exp\left(i\frac{a}{2b}\theta(1)\right) \\
&= 2b\sqrt{r(0)}\exp\left(i \frac{a}{2b} (\theta(0)+\theta(1)-\theta(0))\right) \\
&= F_{a,b}(c)(0) \cdot \exp\left(i \frac{a}{b} \pi \mathrm{ind}(c)\right),
\end{align*}
where $\mathrm{ind}(c)$ is the Whitney rotation index of the regular curve $c$ \cite{whitney1937regular}. By the same reasoning, a necessary condition for a complex curve $q \in C^\infty(I,\C^\ast)$ to be the image of a closed curve under $F_{a,b}$ is that there exists some integer $\ell$ such that
\begin{equation}\label{eqn:closure_condition_1}
q^{(k)}(1)=q^{(k)}(0)\cdot \exp\left(i \frac{a}{b} \pi \ell \right)
\end{equation}
for all integers $k \in \Z_{\geq 0}$. We denote by $V_{a,b}(\ell)$ the codimension-$\infty$ vector subspace of $C^\infty(I,\C)$ containing curves $q$ with property \eqref{eqn:closure_condition_1}. Let $V_{a,b}^\ast(\ell) = V_{a,b}(\ell) \cap C^\infty(I,\C^\ast)$, $V_{a,b}$ denote the union of all $V_{a,b}(\ell)$, and $V_{a,b}^\ast$ the union of the open submanifolds.  If we restrict to simple closed curves, then we are only interested in the vector space $V_{a,b}(1)$.

The discussion above captures the higher-order $C^k$ closure conditions for $c$, but not the $2$-dimensional $C^0$ closure condition $c(0)=c(1)$. In fact, the image of $F_{a,b}$ is locally a codimension-2 submanifold of $V_{a,b}$. In order to perform calculations for closed curves, it will be useful to characterize the two-dimensional normal space to this submanifold. Consider the function $f_{a,b}:V_{a,b}^\ast \rightarrow \C$ defined by
\begin{equation}\label{eqn:closure_condition}
f_{a,b}(q):=F_{a,b}^{-1}(q)(1)=\frac{1}{4b^2} \int_I |q|^2 \left(\frac{q}{|q|}\right)^{\frac{2b}{a}} \; \mathrm{d}t.
\end{equation}
The image of $\mathrm{Imm}(S^1,\C)/\mathrm{Tra}$ in $V_{a,b}^\ast$ is exactly the set $f_{a,b}^{-1}(0)$. We wish to calculate the gradient to $f_{a,b}(q)$ for $q$ in this submanifold.

The derivative of $f_{a,b}$ at $q$ in the direction of a variation $p$ is given by
\begin{align*}
Df_{a,b}(q)(p)&=\frac{1}{2b^2} \int_I \left(\frac{q}{|q|}\right)^{2b/a} \left(\mathrm{Re}(q\overline{p}) - i \frac{b}{a} \mathrm{Im}(q \overline{p})\right) \; \mathrm{d}t.
\end{align*}
The normal space to the submanifold of closed curves is spanned by the gradients of the real and imaginary parts of $f_{a,b}$. The real component of $Df_{a,b}(q)(p)$ is given by
\begin{align*}
\mathrm{Re} Df_{a,b}(q)(p) &= \frac{1}{2b^2} \int_I \mathrm{Re}\left(\left(\frac{q}{|q|}\right)^{2b/a}\right) \mathrm{Re}(q \overline{p}) +  \frac{b}{a} \mathrm{Im}\left(\left(\frac{q}{|q|}\right)^{2b/a}\right) \mathrm{Im}(q \overline{p}) \; \mathrm{d}t \\
&=\int_I \mathrm{Re} \left[ \frac{1}{2b^2} \left(\mathrm{Re} \left(\left(\frac{q}{|q|}\right)^{2b/a}\right) - i \frac{b}{a} \mathrm{Im} \left(\left(\frac{q}{|q|}\right)^{2b/a}\right) \right) q \overline{p} \right] \; \mathrm{d}t
\end{align*}
and it follows that
$$
\mathrm{grad} \left(\mathrm{Re}(f_{a,b})\right)_q =  \frac{1}{2b^2} \left(\mathrm{Re} \left(\left(\frac{q}{|q|}\right)^{2b/a}\right) - i \frac{b}{a} \mathrm{Im} \left(\left(\frac{q}{|q|}\right)^{2b/a}\right) \right) q.
$$
Similarly,
$$
\mathrm{grad} \left(\mathrm{Im}(f_{a,b})\right)_q =  \frac{1}{2b^2} \left(\mathrm{Im} \left(\left(\frac{q}{|q|}\right)^{2b/a}\right) + i \frac{b}{a} \mathrm{Re} \left(\left(\frac{q}{|q|}\right)^{2b/a}\right) \right) q.
$$

\begin{figure}[!t]
\begin{center}
\begin{tabular}{|c|cccc|}
\hline
$\frac{a}{2b}$&1.25&1.00&0.50&0.17\\
\hline
\includegraphics[height=.4in]{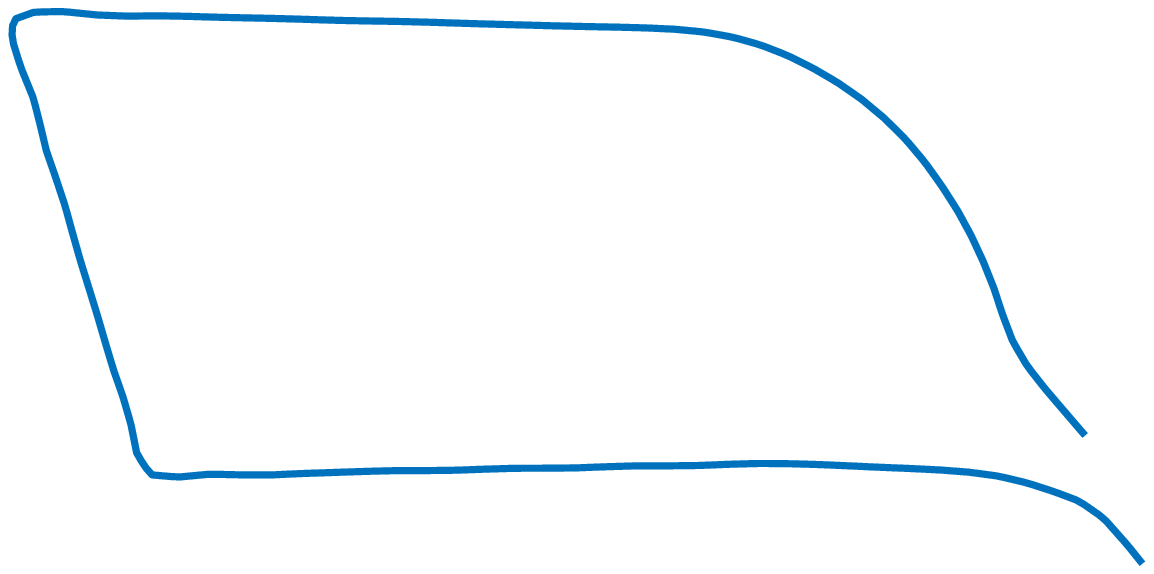}&\includegraphics[height=.4in]{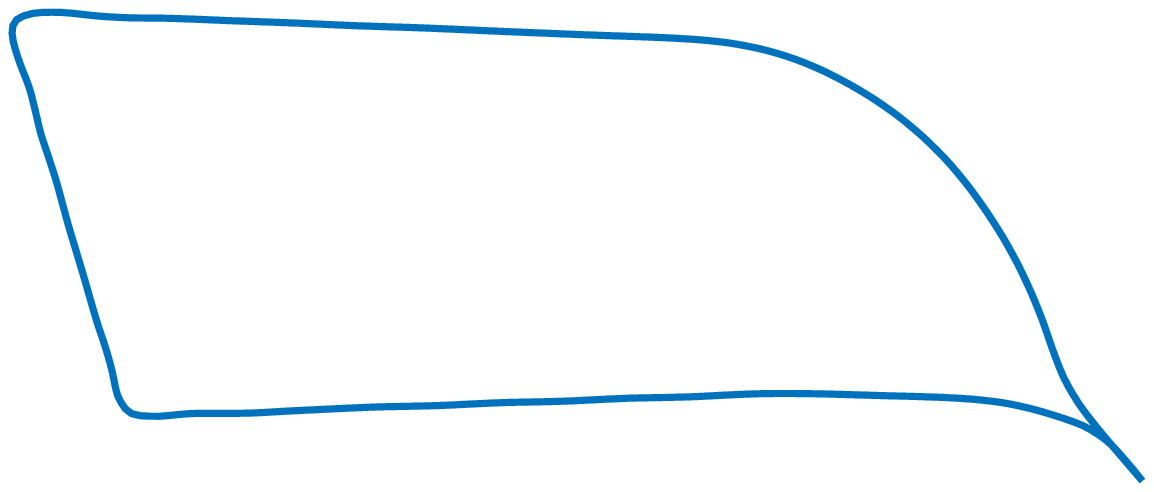}&\includegraphics[height=.4in]{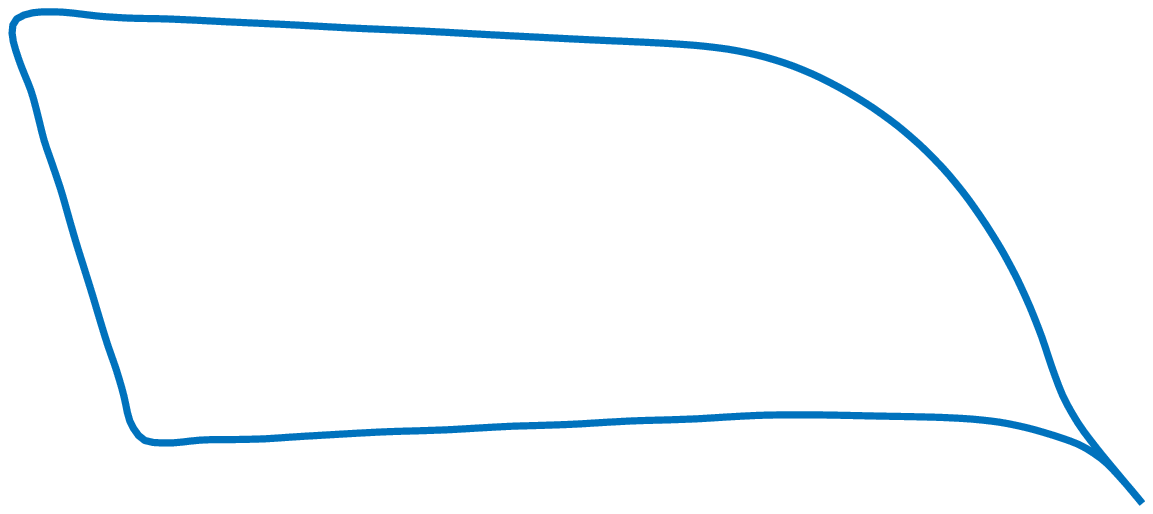}&\includegraphics[height=.4in]{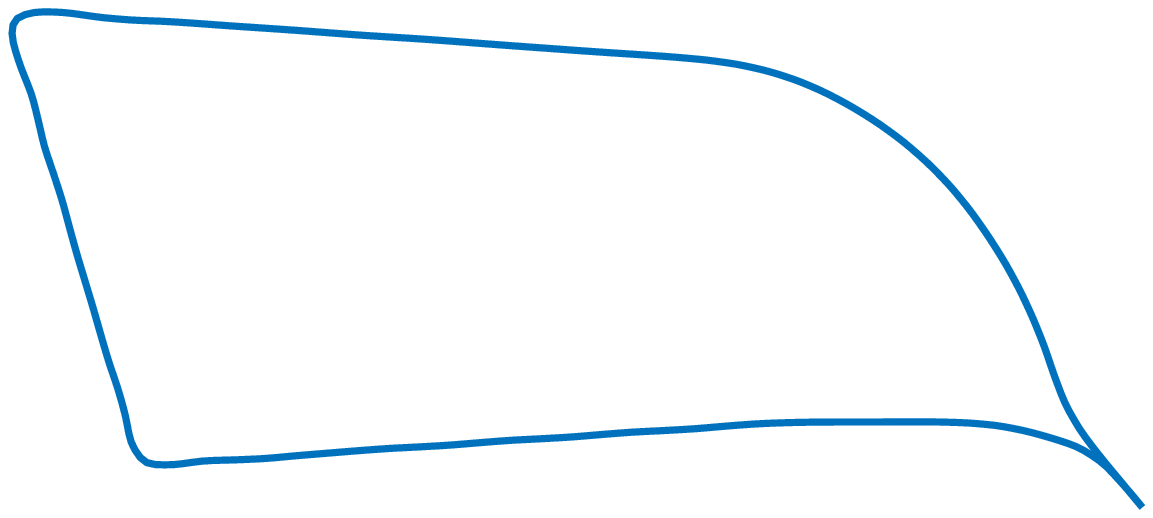}&\includegraphics[height=.4in]{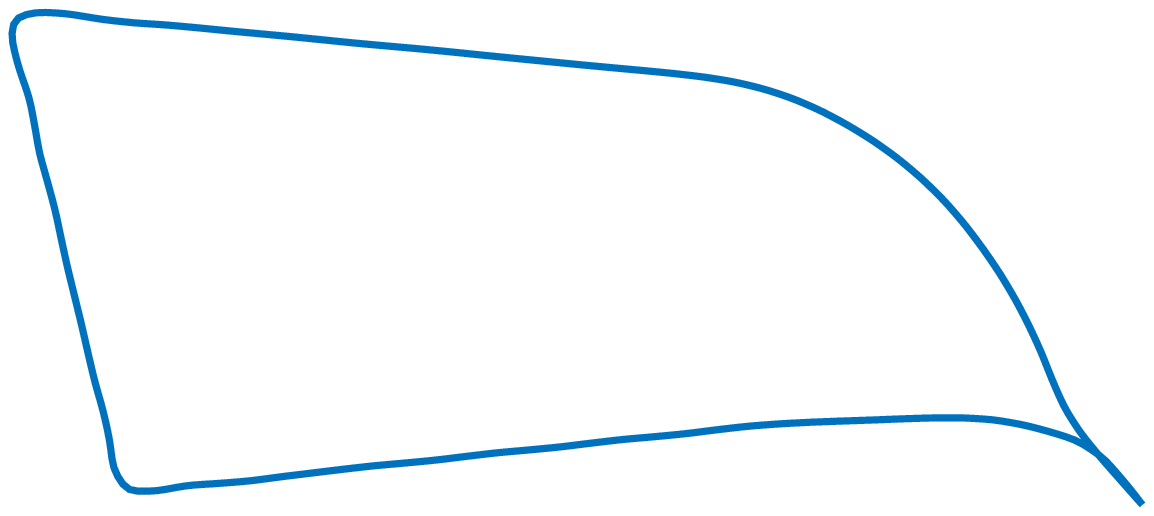}\\
\hline
\includegraphics[height=.79in]{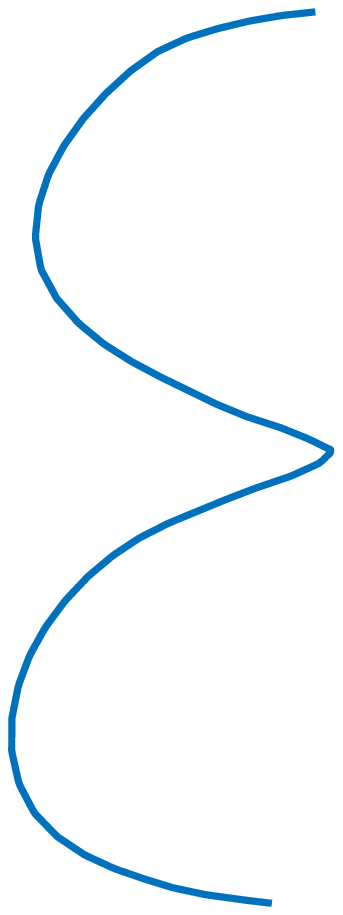}&\includegraphics[height=.79in]{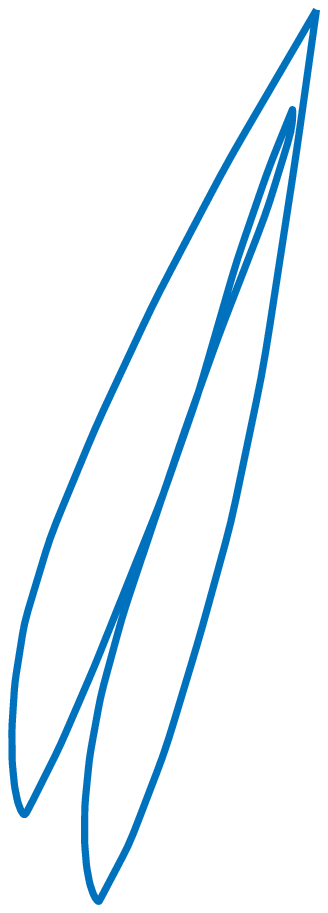}&\includegraphics[height=.4in]{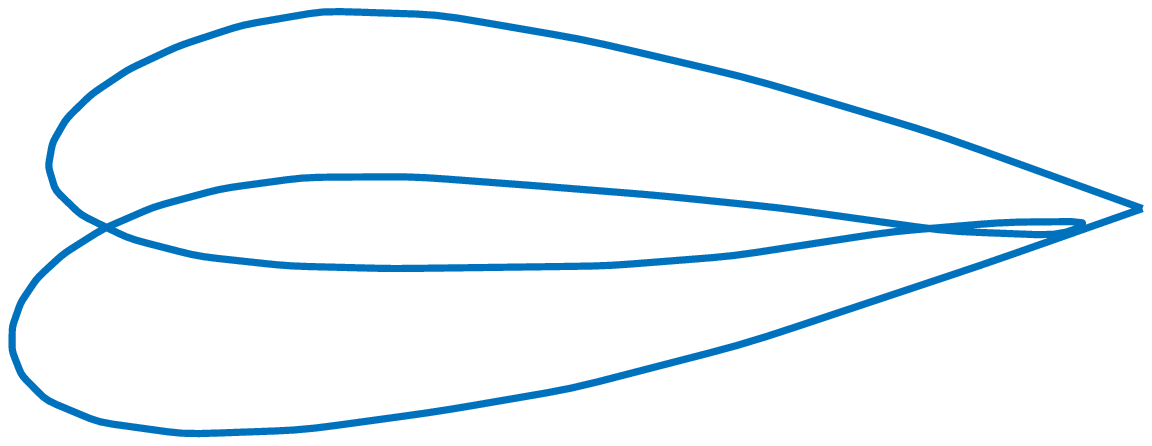}&\includegraphics[height=.4in]{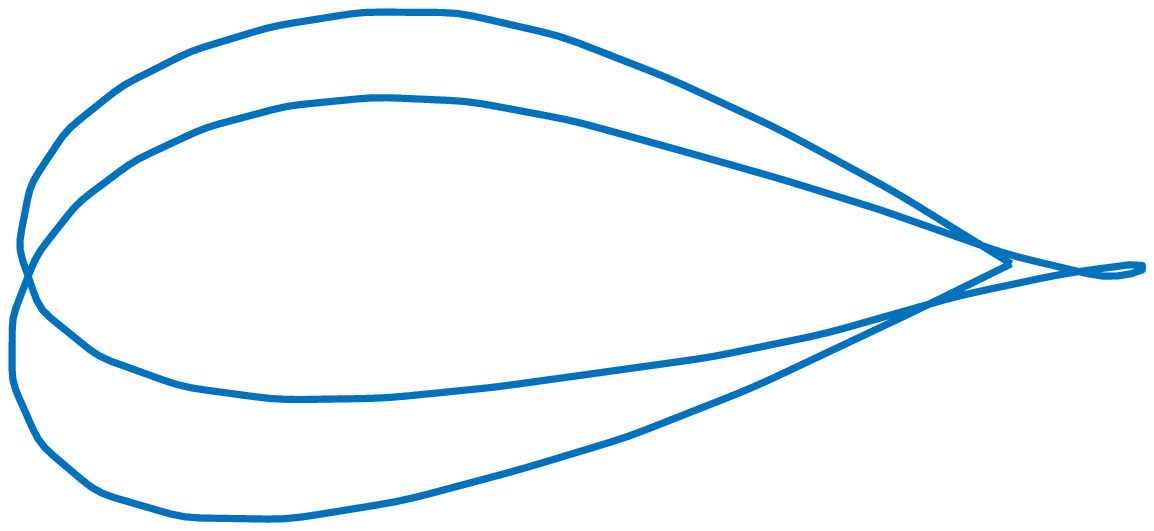}&\includegraphics[height=.4in]{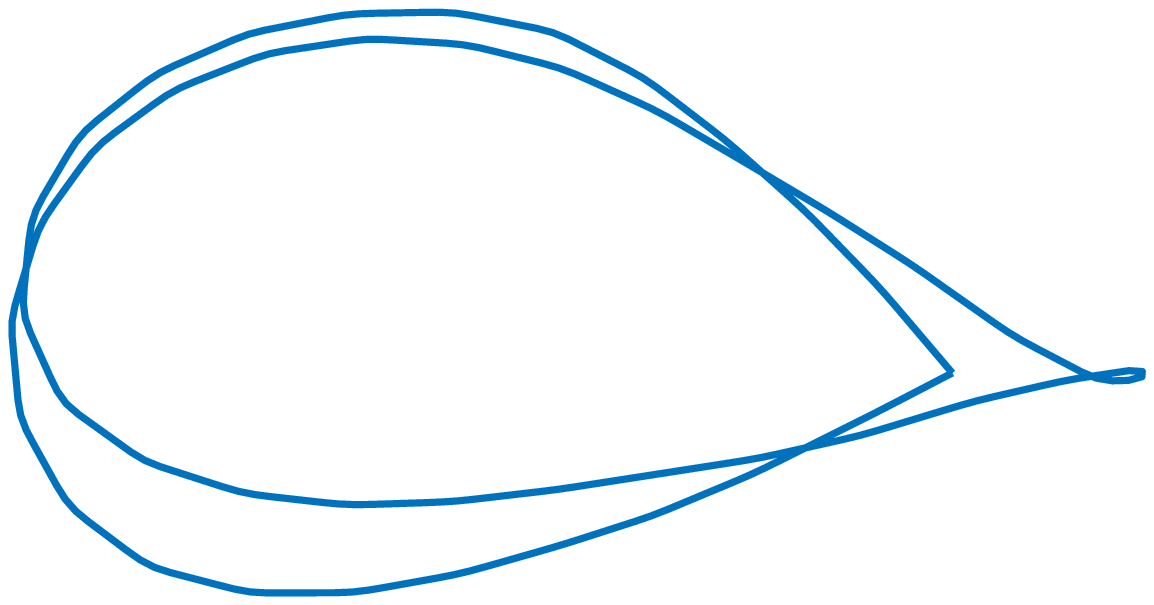}\\
\hline
\includegraphics[height=.65in]{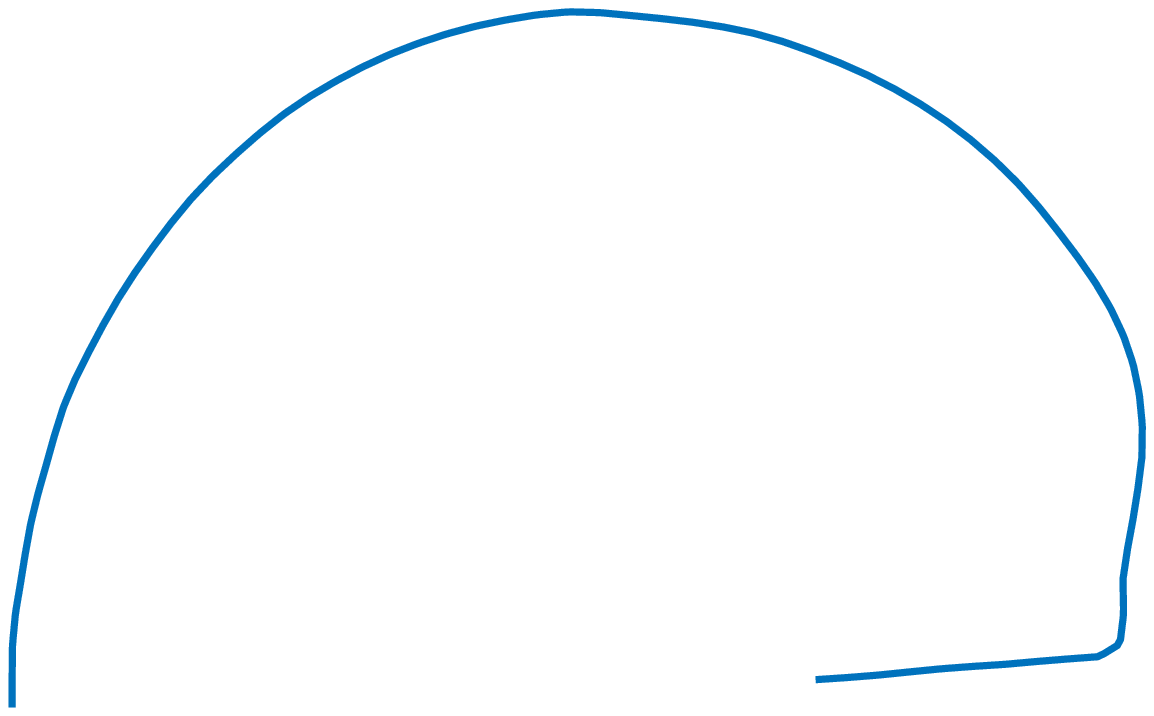}&\includegraphics[height=.79in]{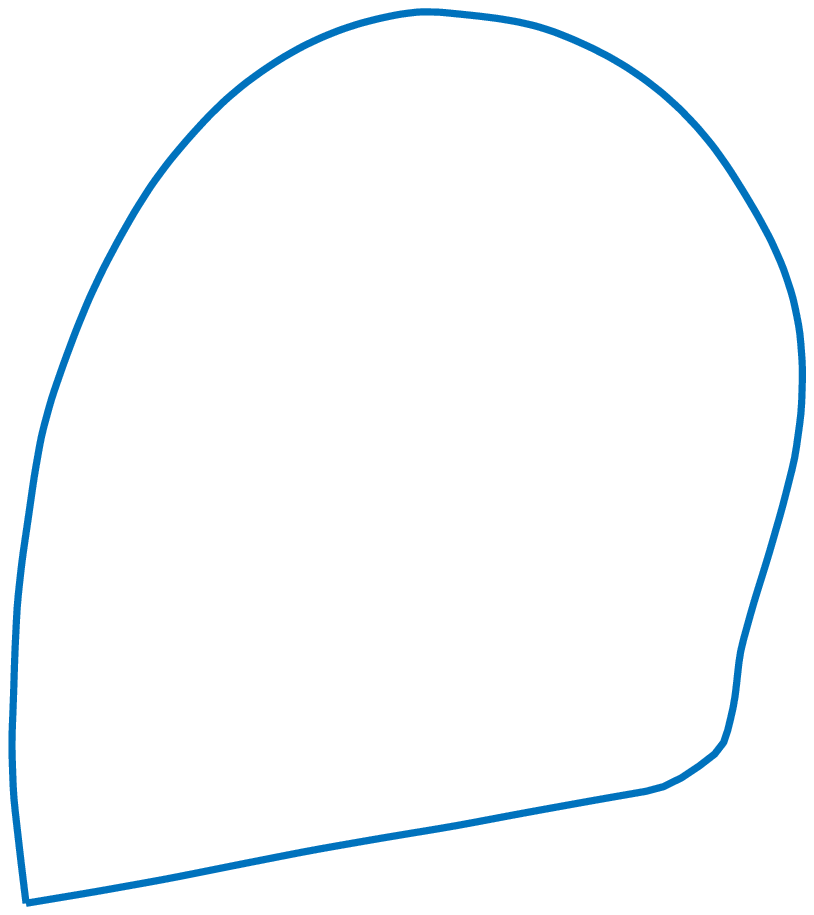}&\includegraphics[height=.79in]{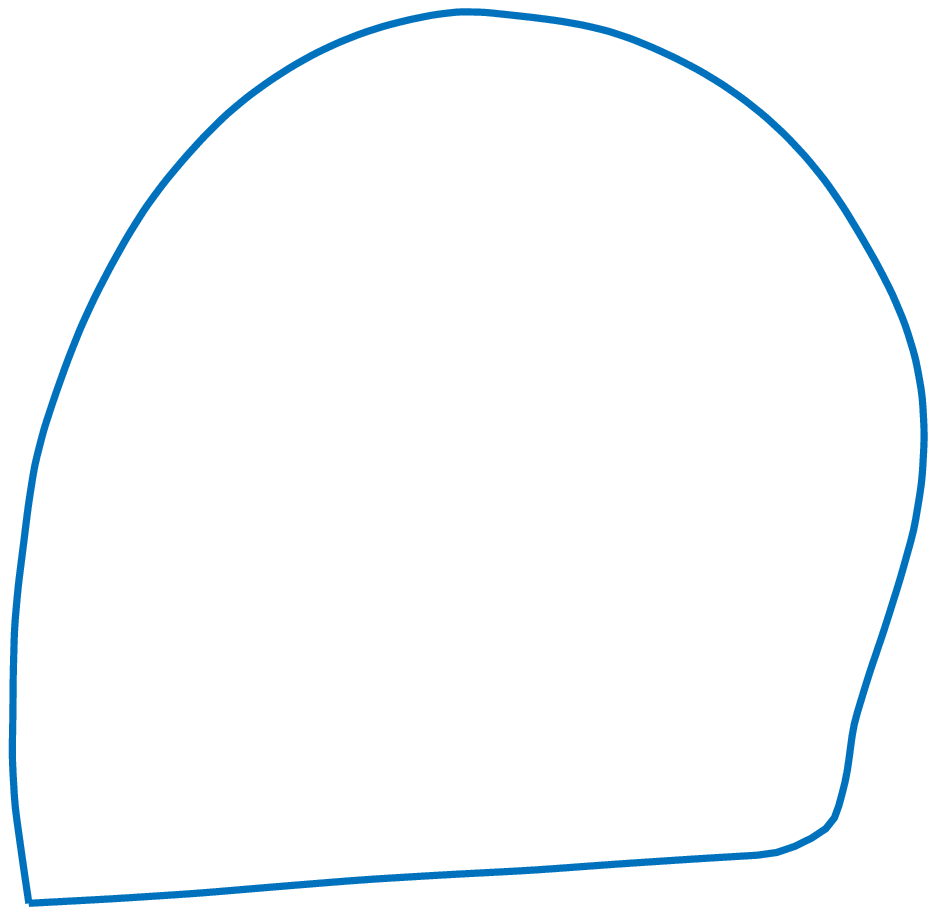}&\includegraphics[height=.79in]{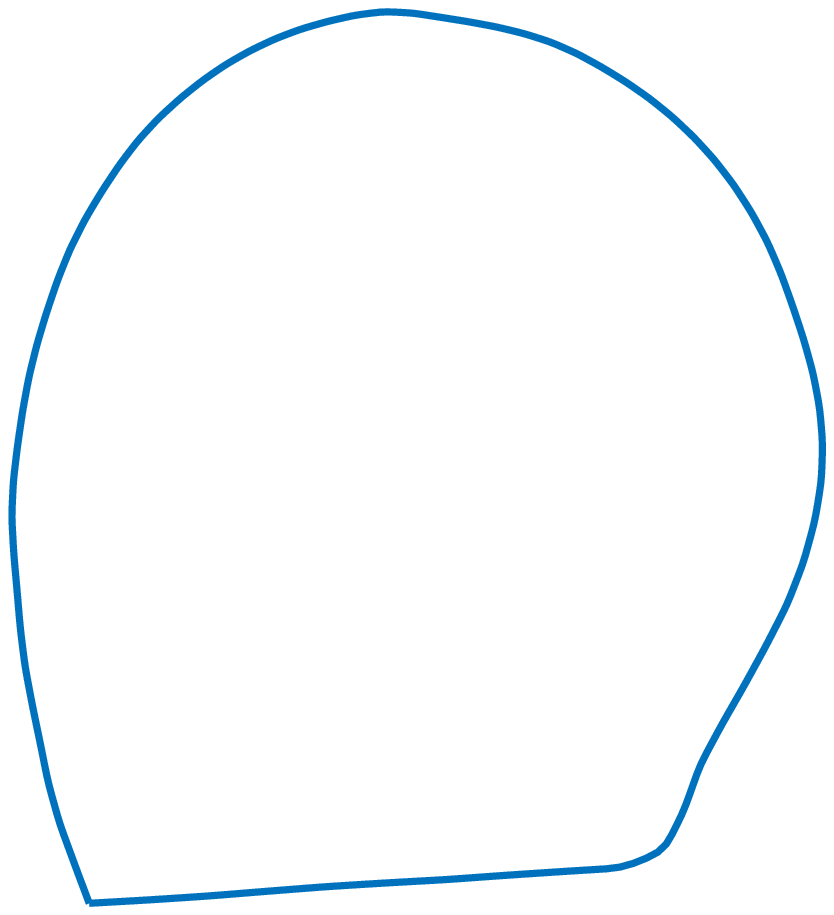}&\includegraphics[height=.79in]{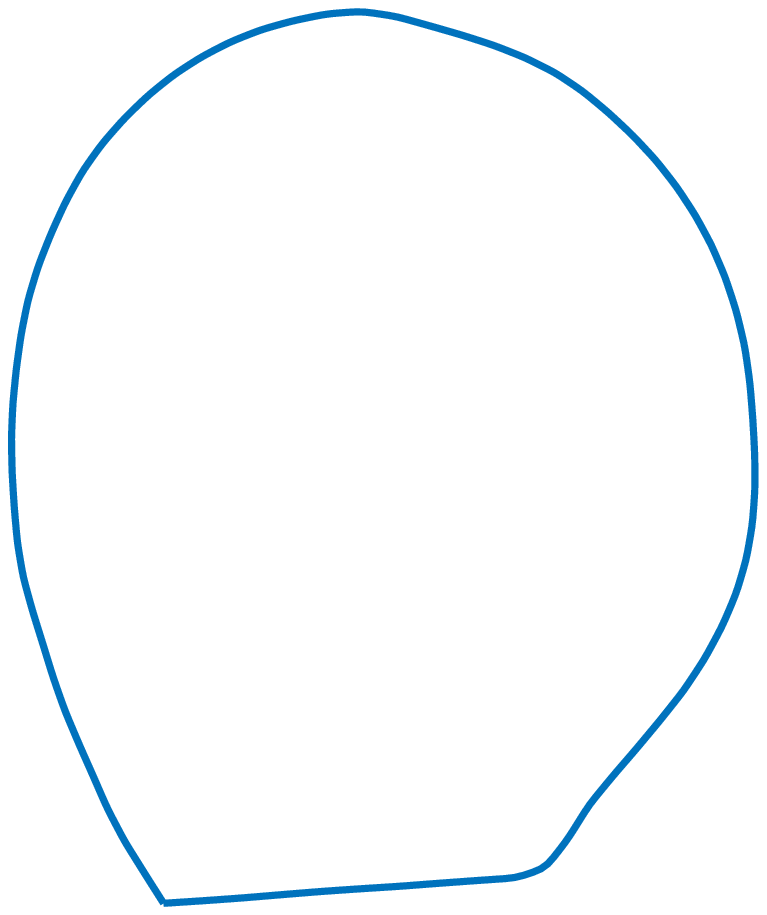}\\
\hline
\end{tabular}
\end{center}
\caption{Projections of an open curve into the preshape space of closed curves under different $F_{a,b}$-transforms.}
\label{fig:closure}
\end{figure}

An important tool for shape analysis of closed curves is the projection from the preshape space of open curves into the preshape space of closed curves. One cannot compute this projection analytically, but the above characterization of the submanifold of closed curves under the $F_{a,b}$-transform allows us to use a gradient descent algorithm for this purpose---see Figure \ref{fig:closure} for a few examples. The algorithm itself is similar in spirit to the one described in \cite{srivastava2011shape} for the SRVF transform.

\subsection{Relation to Previous Work}\label{sec:previous_work}

The family of maps $F_{a,b}$ includes transforms which have already appeared in the literature. Indeed,
$$
F_{1,1/2}(c)=|c'|^{1/2} \frac{c'}{|c'|} = \frac{c'}{|c'|^{1/2}},
$$
so that $F_{1,1/2}$ yields the SRVF transform introduced in \cite{joshi2007novel}. We also have
$$
F_{1/2,1/2}(c)= |c'|^{1/2} \left(\frac{c'}{|c'|}\right)^{1/2} = \sqrt{c'},
$$
and we see that $F_{1/2,1/2}$ is the complex square-root transform studied in \cite{younes2008metric}.

The SRVF was already shown to be a special case of a general family of transforms in \cite{bauer2014constructing}. There, the authors define a two-parameter family of transforms $R_{a,b}$ for $2b \geq a > 0$ by
\begin{align*}
R_{a,b}:\mathrm{Imm}(I,\R^2)/\mathrm{Tra} &\rightarrow C^\infty(I,\R^3) \\
c &\mapsto |c'|^{1/2}\left( \left(\begin{array}{c}
T \\
0 \end{array}\right) + \sqrt{4b^2-a^2}\left(\begin{array}{c}
0 \\
1 \end{array}\right)\right),
\end{align*}
where $T=D_s c$. The image of the $R_{a,b}$-transform is an open subset of a cone given by
$$
C_{a,b}=\{(x,y,z) \in \R^3 \mid (4b^2-a^2)(x^2+y^2)=a^2 z^2, \; z > 0 \}.
$$
The limiting cone $C_{1,1/2}$ can clearly be identified with $\R^2$ and then $R_{1,1/2}(c)$ gives the SRVF transform of $c$. It is shown in \cite{bauer2014constructing} that the $R_{a,b}$ transform pulls back the $L^2$ metric on $C^\infty(I,\R^3)$ to the elastic metric $g^{a,b}$.

We claim that the $R_{a,b}$-transforms correspond to $F_{a,b}$-transforms when $2b \geq a$, so that our main result generalizes \cite{bauer2014constructing} to work for all parameter choices. Indeed, writing $\R^3 \approx \C \times \R$, there is a projection map $\C^\ast \rightarrow C_{a,b}$ defined for each parameter choice with $2b \geq a >0$ in polar coordinates by
\begin{equation}\label{eqn:polar_projection}
(r,\theta) \mapsto \left(\frac{a}{2b} r \cos \left(\frac{2b}{a} \theta\right), \frac{a}{2b} r \sin \left(\frac{2b}{a} \theta \right), \frac{\sqrt{4b^2 - a^2}}{2b} r \right),
\end{equation} 
which extends to a local isometry  $p_{a,b}:C^\infty(I,\C^\ast) \rightarrow C^\infty(I,C_{a,b})$ with respect to the corresponding $L^2$ metrics. Expressing the map in the form
$$
w \mapsto \left(\frac{a |w|}{2b}\left(\frac{w}{|w|}\right)^{2b/a}, \frac{\sqrt{4b^2-a^2}}{2b} |w| \right),
$$
it is easy to see that $R_{a,b} = p_{a,b} \circ F_{a,b}$.

In fact, this idea can be extended to all parameter values. The cones $C_{a,b}$ can be understood in terms of \emph{Regge cones}; these are building blocks of the Regge calculus used to approximate Riemannian manifolds in theoretical physics \cite{regge1961general}. One constructs a two-dimensional Regge cone from polar coordinates $(r,\theta) \in \R_{\geq 0}^2$ with standard metric $\mathrm{d}r^2 + r^2 \mathrm{d}\theta^2$ by identifying points according to the relation $(r,\theta_1) \sim (r,\theta_2) \Leftrightarrow |\theta_1-\theta_2| = 2\pi -\theta_0$ for some choice of \emph{deficit angle} $\theta_0$ (allowed to be positive or negative). The map \eqref{eqn:polar_projection} is an isometric embedding of the Regge cone with deficit angle $\theta_0 = 1-\frac{a}{2b}$ onto a flat cone in Euclidean space. For parameters with $4b^2 < a^2$, replacing the trigonometric functions with their hyperbolic counterparts and the coefficient in the third coodinate with $\frac{\sqrt{a^2-4b^2}}{2b}$ yields an isometric map of the Regge cone with (negative) deficit angle $\theta_0$ onto a flat cone in Lorentz space (see \cite{gronwald1995non}).

\section{Shape Preserving Group Actions}\label{sec:group_actions}

\subsection{Rotation Actions and Fibers of $F_{a,b}$}\label{sec:fibers}

We now treat the fact that $F_{a,b}$ is multivalued for certain parameter choices and noninjective for others. The fibers of $F_{a,b}$ are closely related to the actions of the rotation group $\mathrm{SO}(2)$ on $\mathrm{Imm}(I,\C)/\mathrm{Tra}$ and $C^\infty(I,\C^\ast)$. Using the natural identification of $\mathrm{SO}(2)$ with $S^1$, we can represent the rotation actions as complex multiplication. That is, we express rotations in the respective spaces as $\exp(i \psi) c$ and $\exp(i\psi) q$, where $\exp(i \psi) \in S^1$.  As usual, multiplication in these formulas is performed pointwise as complex numbers. We have the following correspondence between the actions, which follows by an elementary computation.

\begin{lem}\label{lem:rotation_equivariance}
Let $c \in \mathrm{Imm}(I,\C)/\mathrm{Tra}$ and $\exp(i\psi) \in \mathrm{SO}(2)$. Then for all $a,b>0$,
$$
F_{a,b} (\exp(i\psi) c) = \exp\left(i\psi \frac{a}{2b}\right) F_{a,b}(c).
$$
\end{lem}

\begin{thm}\label{thm:isometry}
The $F_{a,b}$ transform induces a well-defined isometry
$$
\mathrm{Imm}(I,\C)/\{\mathrm{Tra,Rot}\} \rightarrow C^\infty(I,\C^\ast)/\mathrm{Rot}
$$
with respect to the metrics induced by $g^{a,b}$ and $g^{L^2}$, respectively.
\end{thm}

We abuse notation and continue to denote the induced isometry by $F_{a,b}$. The induced map is defined by
\begin{equation}\label{eqn:induced_map}
F_{a,b}([c])=\left[F_{a,b}(c)\right],
\end{equation}
where we use brackets to denote the $\mathrm{SO}(2)$-orbit of a parameterized plane curve. The right side denotes the equivalence class of \emph{any} branch of $F_{a,b}(c)$ in the case that the map is multivalued.

\begin{proof}
We first note that all $F_{a,b}$-images of a curve $c$ in the list (\ref{eqn:image_set}) are related by rotations, so that $\left[F_{a,b}(c)\right]$ is a well-defined element of $C^\infty(I,\C^\ast)/\mathrm{Rot}$. Lemma \ref{lem:rotation_equivariance} then implies that the induced map is well-defined. Similar arguments hold for the obvious map induced by the local inverse $F_{a,b}^{-1}$, giving a well-defined inverse map
$$
F_{a,b}^{-1}:  C^\infty(I,\C^\ast)/\mathrm{Rot} \rightarrow \mathrm{Imm}(I,\C)/\{\mathrm{Tra,Rot}\},
$$
so the induced map $F_{a,b}$ is a bijection. Finally, it is easy to see that $g^{a,b}$ and $g^{L^2}$ are invariant under the action of $\mathrm{SO}(2)$. Therefore, the local isometry of Theorem \ref{thm:pullback} descends to a global isometry on the quotient spaces.
\end{proof}

We have the following immediate corollary.

\begin{cor}\label{cor:isometry}
The $F_{a,b}$ transform induces a well-defined isometric embedding
$$
\mathrm{Imm}(S^1,\C)/\{\mathrm{Tra,Rot}\} \hookrightarrow V^\ast_{a,b}/\mathrm{Rot},
$$
where $V^\ast_{a,b}$ is the space of curves defined in Section \ref{sec:closed_curves}.
\end{cor}

\subsection{The Scaling Action}

The group of positive real numbers $\R_{>0}$ acts on a parameterized curve by uniform scaling. An easy calculation shows that the scaling action interacts with the $F_{a,b}$-transform as follows: for $\lambda \in \R_{>0}$ and $c \in \mathrm{Imm}(I,\C)/\mathrm{Tra}$,
$$
F_{a,b}(\lambda c) = \lambda^{1/2} F_{a,b}(c).
$$
It will be convenient to represent the quotient of the preshape space by this scaling action as
\begin{equation}\label{eqn:fixed_length_space}
\mathrm{Imm}(I,\C)/\{\mathrm{Tra,Sca}\} \approx \{c \in \mathrm{Imm}(I,\C) \mid c(0)=0, \; \mathrm{length}(c)=1\}.
\end{equation}

For $\Sigma = I$ or $S^1$ and $a,b > 0$, define the \emph{Hilbert sphere of radius $r$} to be the space
$$
\mathcal{H}_{a,b}^\Sigma(r) =  \left\{\begin{array}{ll}
\left\{q \in C^\infty(I,\C) \mid \int_\Sigma |q|^2 \; \mathrm{d}t = r^2 \right\} & \Sigma = I \\
\left\{q \in V_{a,b} \mid \int_\Sigma |q|^2 \; \mathrm{d}t = r^2 \right\}  & \Sigma = S^1.
\end{array}\right.
$$

\begin{prop}\label{prop:hilbert_sphere}
The $F_{a,b}$ transform sends $\mathrm{Imm}(\Sigma,\C)/\{\mathrm{Tra,Sca}\} $ into the Hilbert sphere $\mathcal{H}^\Sigma_{a,b}(2b)$. It induces an isometry between $\mathrm{Imm}(\Sigma,\C)/\{\mathrm{Tra,Sca,Rot}\}$ and its image in $\mathcal{H}^\Sigma_{a,b}(2b)/\mathrm{Rot}$.
\end{prop}

\begin{proof}
Let $c \in \mathrm{Imm}(\Sigma,\C)/\mathrm{Tra}$ have length one. Then,
$$
\int_\Sigma |F_{a,b}(c) |^2 \; \mathrm{d}t = \int_\Sigma \left(2b |c'|^{1/2} \right)^2 \; \mathrm{d}t = 4b^2 \int_\Sigma |c'| \; \mathrm{d}t = 4b^2.
$$
The second statement follows by noting that $\mathcal{H}_{a,b}^\Sigma(2b)$ is invariant under the rotation action of $\mathrm{SO}(2)$ so that we can restrict the isometries of Theorem \ref{thm:isometry} and Corollary \ref{cor:isometry}.
\end{proof}

\subsection{The Reparameterization Action}

The final shape-preserving group action to consider is the action of $\mathrm{Diff}^+(I)$ on $\mathrm{Imm}(I,\C)$ by orientation-preserving reparameterizations. An element $\gamma \in \mathrm{Diff}^+(I)$ of the diffeomorphism group also acts on $q \in C^\infty(I,\C^\ast)$ by the formula
\begin{equation}\label{eqn:diff_action_transform}
\gamma \ast q := \sqrt{|\gamma'|} (q \circ \gamma).
\end{equation}

\begin{prop}\label{prop:diff_equivariant}
The map $F_{a,b}$ is equivariant with respect to the $\mathrm{Diff}^+(I)$-actions on $\mathrm{Imm}(I,\C)$ and $\C^\infty(I,\C^\ast)$ defined above.
\end{prop}

\begin{proof}
Let $c \in \mathrm{Imm}(\Sigma,\C)$ and $\gamma \in \mathrm{Diff}^+(\Sigma)$. Then
\begin{align*}
F_{a,b}(c \circ \gamma) &= 2b |\gamma' ( c' \circ \gamma)|^{1/2} \left(\frac{\gamma' ( c' \circ \gamma)}{|\gamma' (c' \circ \gamma)|}\right)^\frac{a}{2b} \\
&= |\gamma'|^{1/2} 2b |c' \circ \gamma|^{1/2}\left(\frac{c' \circ \gamma}{|c' \circ \gamma|}\right)^\frac{a}{2b} = |\gamma'|^{1/2} F_{a,b}(c) \circ \gamma.
\end{align*}
\end{proof}

We likewise wish to consider the action of $\mathrm{Diff}^+(S^1)$ on $\mathrm{Imm}(S^1,\C)$ and the corresponding action in transform space. To understand the reparameterization action in transform space for closed curves, it is convenient to identify $V_{a,b}(\ell)$ with the vector space
$$
\widetilde{V}_{a,b}(\ell) = \left\{q \in \C^\infty(\R,\C^\ast) \mid q(t+1)=q(t) \cdot \exp\left(i \frac{a}{b} \pi \ell\right)\right\}.
$$
Under this identification, the $\mathrm{Diff}^+(S^1)$-action on transform space is once again given by \eqref{eqn:diff_action_transform} in the sense that a similar equivariance result holds in this setting.

\section{Geodesics Between Curves}\label{sec:computing_geodesics}

\subsection{Geodesics for Open Curves}\label{sec:geodesics_open_curves}

For open curves, Theorem \ref{thm:isometry} provides an isometry of Riemannian manifolds
$$
F_{a,b}:\left(\mathrm{Imm}(I,\C)/\{\mathrm{Tra,Rot}\},g^{a,b}\right) \rightarrow \left(C^\infty(I,\C^\ast)/\mathrm{Rot},g^{L^2}\right).
$$
We can therefore compute geodesics between curves in the former space by translating the problem to the simpler target space, $C^\infty(I,\C)$, where geodesic paths are simply straight lines.  Passing to the open subset $C^\infty(I,\C^\ast)$, we lose geodesic completeness as the geodesic joining $q_0,q_1 \in C^\infty(I,\C^\ast)$ given by $q_u=(1-u)q_0 + u q_1$, $u \in [0,1]$, may pass through a curve with $q_u(t)=0$ for some $t$. Nonetheless, geodesic distance in the larger space still induces a metric on the restricted space, and curves which are sufficiently close tend to be joined by a nonsingular geodesic in practice. The geodesic distance between a pair of curves is
$$
d^{L^2}(q_0,q_1)=\| q_0 - q_1\|_{L^2} = \left(\int_I |q_0-q_1|^2 \; \mathrm{d}t\right)^{1/2}.
$$
Geodesics in the quotient  $\mathrm{Imm}(I,\C)/\{\mathrm{Tra, Rot}\} \approx C^\infty(I,\C^\ast)/\mathrm{Rot}$ are realized as geodesics between curves in the total space $C^\infty(I,\C)$ after a preprocessing step whereby the curves are aligned over $\mathrm{SO}(2)$ using a standard algorithm called Procrustes analysis (essentially a singular value decomposition problem). Furthermore, we calculate explicit geodesics in $\mathrm{Imm}(I,\C)/\{\mathrm{Tra, Sca, Rot}\}$ by using Proposition \ref{prop:hilbert_sphere} to transfer the problem to the Hilbert sphere. The geodesic joining a pair of curves $q_0,q_1 \in \mathcal{H}^I_{a,b}(2b)$ is given by spherical interpolation
$$
q_u=\frac{1}{\sin(D)} \left(\sin((1-u)D) q_0 + \sin(u D) q_1\right),
$$
where $D=\arccos\left(\left<q_0,q_1\right>_{L^2}/4b^2\right)$ is geodesic distance in the Hilbert sphere. Geodesics in the quotient $\mathcal{H}^I_{a,b}(2b)/\mathrm{Rot}$ are treated by the same optimization procedure as in the flat case.

\subsection{Geodesics for Closed Curves}

Geodesics in $\mathrm{Imm}(S^1,\C)/\{\mathrm{Tra, Rot}\}$ can be treated in a similar manner to the open case; that is, we transfer the problem to the simpler space $V_{a,b}/\{\mathrm{Rot}\}$. Each vector space $V_{a,b}(\ell)$ is flat, so its geodesics are straight lines. However, the fact that the image of the isometry induced by $F_{a,b}$ is codimension-2 in $V_{a,b}$ makes the procedure more complicated.

We first note that the problem has a remarkable simplification in the case that $a=b$, where the closure condition $f_{a,b}(q)=0$ (see \eqref{eqn:closure_condition}) reduces to $L^2$-orthogonality of the coordinate functions of $q$. This was expoited by Younes et.\ al.\ in \cite{younes2008metric} to give explicit geodesics for closed curves by relating the space of closed curves to an infinite-dimensional Stiefel manifold.

Apparently, such a simplification is unique to the case $a=b$, and the space of closed curves is not isometric to a classical manifold otherwise. Fortunately, the low codimension of the space of closed curves in the flat space $V_{a,b}$ allows us to approximate geodesics in the submanifold numerically. There are several algorithms in the literature which are easily adapted to our setting, such as parallel transport-based path-straightening \cite{srivastava2011shape} and other gradient descent-based \cite{bauer2014constructing,you2015riemannian} methods. In the examples provided in Section \ref{sec:applications}, we use a simplistic projection-based algorithm.

\subsection{Optimized Geodesics in the Shape Spaces}\label{sec:optimized_geodesics}

To compute geodesics in the shape space of unparameterized curves with respect to the metric induced by $g^{a,b}$, we pass to the quotient of the parameterized curve space by the action of the diffeomorphism group $\mathrm{Diff}^+(I)$ ($\mathrm{Diff}^+(S^1)$ in the case of closed curves). The geodesic between $\mathrm{Diff}^+(I)$-orbits $[c_1]$ and $[c_2]$ of parameterized curves $c_j$ is realized in practice as the geodesic between $c_1$ and $\widetilde{c}_2$ in the total space $\mathrm{Imm}(I,\C)$, where $\widetilde{c}_2 = c_2 \circ \gamma$ and
\begin{equation}\label{eqn:optimal_reparameterization}
\gamma = \mathrm{arginf} \left\{\mathrm{dist}_{g^{a,b}} (c_1,c_2(\gamma)) \mid \gamma \in \mathrm{Diff}^+(I)\right\}.
\end{equation}
Here $\mathrm{dist}_{g^{a,b}}$ denotes geodesic distance with respect to $g^{a,b}$. In general, the reparameterization realizing this infimum may fail to be smooth (see \cite[Section 4.2]{younes2008metric}), whence the geodesic is actually realized in the larger space of $L^2$ curves. Precise characterizations of the regularity of solutions to the optimization problem for SRVF parameters $g^{1,1/2}$ have been the subject of several recent articles, and it has been shown that for $C^1$ input curves, the optimal reparameterization $\gamma$ is achieved and is differentiable almost everywhere \cite{bruveris2016optimal}. For applications, the realistic setting considers PL curves and it is known that in that case, the optimal reparameterization is realized and is also PL \cite{lahiri2015precise}. We adapt the methods of \cite{lahiri2015precise} to the PL setting for general elastic metrics $g^{a,b}$ in Section \ref{sec:PL_shapes}.

A major benefit of our results is that under the $F_{a,b}$-transform, the optimal reparameterization problem \eqref{eqn:optimal_reparameterization} becomes equivalent to the optimization  problem which appears under the SRVF formalism. We are therefore able to utilize existing, highly efficient numerical approaches to approximate solutions of \eqref{eqn:optimal_reparameterization} (for example, the dynamic programming approach of \cite{mio2007shape}). Once an approximate solution $\gamma$ is obtained, we are able to easily compute geodesics in the shape space using the techniques of Section \ref{sec:geodesics_open_curves} applied to $c_1$ and the reparameterized curve $\widetilde{c}_2$.

The approach described above can be adapted to provide geodesics in the space of curves of fixed length (using the Hilbert sphere geodesics of Section \ref{sec:geodesics_open_curves}, together with the simple relationship between Euclidean and spherical distance) and geodesics in the space of closed curves (by numerically optimizing over $\mathrm{Diff}^+(I) \times S^1$, where the $S^1$ factor corresponds to a search for optimal seed points between the two curves). 
\section{Extending to Piecewise Linear Shapes}\label{sec:PL_shapes}

\subsection{Problem Setup}

A natural question of both theoretical and practical interest is whether the previous results can be extended to spaces of curves of lower regularity. Indeed, the methods above can be used, essentially without modification, to show that the map $F_{a,b}$ can be extended to give an isometry
$$
\left(\mathrm{Imm}^1(I,\C)/\{\mathrm{Tra,Rot}\},g^{a,b}\right) \rightarrow \left(C^0(I,\C^\ast),g^{L^2}\right),
$$
where $\mathrm{Imm}^1$ denotes the space of $C^1$ immersions and $g^{a,b}$ is the appropriately extended elastic metric.

It is common in the literature on the SRVF transform to consider the space $\mathrm{AC}(I,\R^N)$ of absolutely continuous curves in Euclidean space \cite{bruveris2016optimal,lahiri2015precise,srivastava2011shape}; absolutely continuous curves can be characterized as those curves which are continuous everywhere and differentiable almost everywhere \cite{royden1968real}. Recall that the transform $F_{1,1/2}$ recovers the SRVF transform for smooth, immersed plane curves. This map extends to a well-defined map
\begin{align*}
F_{1,1/2}:\mathrm{AC}(I,\C) &\rightarrow L^2(I,\C) \\
c(t) & \mapsto \left\{\begin{array}{cc}
\frac{c'(t)}{\sqrt{|c'(t)|}} & c'(t) \mbox{ exists and is nonzero} \\
0 & \mbox{otherwise}, \end{array}\right.
\end{align*}
which is a homeomorphism that pulls back $g^{L^2}$ to $g^{1,1/2}$ at smooth points.

One would like to similarly extend the remaining $F_{a,b}$ transforms to spaces of curves of low regularity, but this causes immediate issues. For $a/2b \neq 1$, one of $F_{a,b}$ or its inverse is multivalued pointwise, since the map involves complex exponentiation. In previous sections, we relied on the smoothness of our curves (or at least continuity of derivatives) to choose complex roots coherently in order to ensure $F_{a,b}$ was well-defined up to rotations.

\subsection{Extended $F_{a,b}$-Transform}

Inspired by our results for smooth curves, we can extend our work to one of the most important spaces of curves from a practical standpoint: piecewise linear curves. Let $\mathrm{PL}(I,\C)$ denote the space of piecewise linear planar curves; that is, each $c \in \mathrm{PL}(I,\C)$ is a continous curve such that there is a decomposition
$$
I=I_1 \cup I_2 \cup \cdots \cup I_k = [t_0=0,t_1] \cup [t_1,t_2] \cup \cdots \cup [t_{k-1},t_k=1]
$$
with $c'(t) = v_j \in \C$ for all $t \in (t_{j-1},t_j)$, $j=1,\ldots,k$. We call the points $c(t_j)$ \emph{vertices} of $c$ and the $t_j$ are called \emph{vertex parameters}. A PL curve $c$ is called a \emph{piecewise linear immersion} if $|c'(t)| \neq 0$ where $c'(t)$ is defined. Let $\mathrm{PL}(I,\C)/\mathrm{Tra}$ denote the space of piecewise linear curves modulo translations, which we identify with the set of curves based at the origin.

We define the \emph{extended $F_{a,b}$-transform} to be the map
$$
F_{a,b}(c)=2b r^{1/2} \exp \left(i\frac{a \theta}{2b}\right),
$$
where $r$ is the piecewise constant function
\begin{equation}\label{eqn:radius_function}
r(t)=\left\{\begin{array}{cl}
 |v_1| & \mbox{ for } t \in [0,t_1] \\
 |v_j|  & \mbox{ for } t \in (t_{j-1},t_j], \, j=2,\ldots,k, \end{array}\right.
\end{equation}
and $\theta$ is the piecewise constant function defined recursively by
\begin{equation}\label{eqn:recursive_definition}
\theta(t) = \left\{\begin{array}{cl}
\theta_1 = \arctan \frac{\mathrm{Im}(v_1)}{\mathrm{Re}(v_1)} & \mbox{ for } t \in [0,t_1], \\
\theta_{j-1} + s_j \cdot \delta \theta_j & \mbox{ for } t \in (t_{j-1},t_j],\,  j=2,\ldots,k. \end{array}\right.
\end{equation}
The term $\delta\theta_j \in [0,\pi]$ is the \emph{$j$th exterior angle} between the $(j-1)$th and $j$th edges of $c$ and is given by
\begin{equation}\label{eqn:delta_theta}
\delta \theta_j = \arccos \, \mathrm{Re}\left(\frac{v_{j-1} \overline{v_j}}{|v_{j-1}\overline{v_j}|}\right).
\end{equation}
To simplify notation later on, we set $\delta \theta_1 = \theta_1$. The coefficient $s_j\in \{-1,1\}$ describes the \emph{orientation} of the $j$th exterior angle and is given by
\begin{equation}\label{eqn:angle_sign}
s_j=\mathrm{sign}\left(\mathrm{Im}\left(v_{j-1} \overline{v_j}\right)\right),
\end{equation}
where we define the sign function for any real number $a$ according to the convention
$$
\mathrm{sign}(a) := \left\{\begin{array}{cc}
1 & \mbox{ if } a > 0 \\
-1 & \mbox{ if } a \leq 0.\end{array}\right.
$$
We set $s_1=1$.

\begin{lem}\label{lem:convergence_of_angle_function}
Let $c$ be a smooth immersion with a fixed representation of $c'$ in polar coordinates, $c'=re^{i\theta}$, so that $\theta(0)=\arctan \frac{\mathrm{Im}(c'(0))}{\mathrm{Re}(c'(0))}$ and $\theta$ is continuous. Let $\{c^n\}$ be a sequence of PL immersions with vertices  sampled from $c$ (i.e., each $c^n$ is a secant approximation of $c$). Assume that the sequence $\{c^n\}$ converges uniformly to $c$ in $C^1$ and let $(r^n,\theta^n)$ be polar coordinates for $c^n$, obtained by the formulas \eqref{eqn:radius_function} and \eqref{eqn:recursive_definition}. Then $r^n \rightarrow r$ and $\theta^n \rightarrow \theta$ uniformly.
\end{lem}

\begin{proof}
The assumption of uniform convergence $(c^n)'(t) \rightarrow c'(t)$ immediately implies $r^n(t) \rightarrow r(t)$ uniformly, and one only needs to check convergence of the angle functions. Let $v_j^n$ denote the derivative vectors of $c^n$ and let $\widehat{v}_j^n$ denote the derivative vectors of $c$ sampled at vertex parameters of $c^n$. Moreover, let $\widehat{\theta}^n_j$ denote samples of the function $\theta$ at the vertex parameters of $c^n$. Then, our assumptions imply that $|v_j^n - \widehat{v}_j^n| \rightarrow 0$ uniformly in $j$ and it follows that there is some collection of angles $\{\widetilde{\theta}_j^n\}$ such that $v_j^n = r_j^n \exp (i \widetilde{\theta}_j^n)$ for all $j,n$ and such that $|\widetilde{\theta}_j^n - \widehat{\theta}_j^n|$ can be uniformly bounded in $j$ for sufficiently large $n$. We wish to show that $\widetilde{\theta}_j^n = \theta_j^n$ for sufficiently large $n$. By the definition of $\theta^n_j$, the contrary would imply that for some $j$ we have
$$
\pi < \left| \widetilde{\theta}_j^n - \widetilde{\theta}_{j-1}^n \right| \leq \left| \widetilde{\theta}_j^n  - \widehat{\theta}_j^n \right| + \left|\widetilde{\theta}_{j-1}^n - \widehat{\theta}_{j-1}^n\right| + \left|\widehat{\theta}_j^n - \widehat{\theta}_{j-1}^n\right|.
$$
Taking $n$ sufficiently large, the first two terms on the right side can uniformly be made arbitrarily small, by the definition of $\widetilde{\theta}_j^n$. The last term can be made arbitrarily small because the $\widehat{\theta}_j^n$ are sampled from the continuous function $\theta$, and we have arrived at a contradiction.
\end{proof}

This immediately implies the following convergence result.

\begin{prop}
Let $c$ be a smooth immersed plane curve and let $\{c_n\}$ be a sequence of piecewise linear immersions as in Lemma \ref{lem:convergence_of_angle_function}. Then, $F_{a,b}(c_n) \rightarrow F_{a,b}(c)$ in $L^2$.
\end{prop}

\begin{remark}
This shows that if a PL curve approximates a smooth curve, then taking sufficiently many samples produces a transformed curve, which is close to the transformed smooth curve. On the other hand, if the PL curve is truly meant to contain jagged angles, then the discrete $F_{a,b}$ is still well-defined but may not be a faithful representation of the curve in transform space. In this scenario, ad hoc methods are necessary to extend the transformation.
\end{remark}
\subsection{Injectivity}

Theorem \ref{thm:isometry} says that $F_{a,b}$ induces a bijection
$$
\mathrm{Imm}(I,\C)/\{\mathrm{Tra},\mathrm{Rot}\} \leftrightarrow C^\infty(I,\C^\ast)/\mathrm{Rot}.
$$
Unfortunately, the same property is not enjoyed by the extended $F_{a,b}$-transform, as shown by the following example.

\begin{example}
Let $c_1$ and $c_2$ be the PL curves given by $c_1'(t)=1$ and
$$
c_2'(t) = \left\{\begin{array}{cc}
1 & t \in [0,1/2], \\
e^{i\theta} & t \in [1/2,1],\end{array}\right.
$$
for some fixed choice of $\theta \in (0,\pi]$. Then $F_{a,b}(c_1)(t)=2b$ and
$$
F_{a,b}(c_2)(t) = \left\{\begin{array}{cc}
2b & t\in [1/2,1] \\
2b\exp\left(i \theta \frac{a}{2b} \right) & t \in [1/2,1]. \end{array}\right.
$$
Taking parameters $a$ and $b$ satisfying $\frac{a}{2b}=\frac{2\pi}{\theta}$ yields $F_{a,b}(c_1)= F_{a,b}(c_2)$, while it is clear that $c_1$ and $c_2$ do not differ by a rigid rotation.
\end{example}

This potentially causes a major problem when computing distances between PL curves: two PL curves which do not differ by a rotation can receive zero geodesic distance in $F_{a,b}$-transform space. Luckily, the next proposition shows that this situation is highly non-generic and that it does not arise in most applications.

\begin{prop}\label{prop:injectivity}
Let $c_1$ and $c_2$ be PL immersions with $k$ segments whose derivatives are represented in polar coordinates by the functions $(r^1,\theta^1)$ and $(r^2,\theta^2)$, respectively, as defined by formulas \eqref{eqn:radius_function} and \eqref{eqn:recursive_definition}.  The images of the curves under the extended $F_{a,b}$-transform are the same if and only if $r^1(t)=r^2(t)$ for all $t$ and there exist integers $\ell_j$ such that
$$
s_j^1 \delta \theta^1_j = s_j^2 \delta \theta^2_j + \frac{4b}{a} \ell_j \pi
$$
for all $j=1,\ldots,k$.
\end{prop}

\begin{proof}
Assume that $F_{a,b}(c_1)=F_{a,b}(c_2)$. Then $|F_{a,b}(c_1)(t)|=|F_{a,b}(c_2)(t)|$ holds for all $t$ and it follows that $r^1(t)=r^2(t)$ for all $t$. Since $r^j(t)\neq 0$, we have that $\exp\left(\frac{a \theta^1}{2b}\right) = \exp \left(\frac{a \theta^2}{2b}\right)$. This equality holds if and only if
$$
\frac{a \theta^1_j}{2b} = \frac{a \theta^2_j}{2b} + 2 \ell_j \pi
$$
for all $j=1,\ldots,k$ for some integers $\ell_j$. Therefore
\begin{equation}\label{eqn:injectivity_condition}
\theta^1_j = \theta^2_j + \frac{4b}{a} \ell_j \pi.
\end{equation}
Setting $j=1$ and recalling that we are using the convention $\delta \theta_1 = \theta_1$ and $s_1=1$ provides our first condition on the $\theta$-functions. The $j=2$ instance of \eqref{eqn:injectivity_condition} reads
\begin{align*}
\theta^1_2 = \theta^2_2 + \frac{4b}{a}\ell_2 \pi &\Leftrightarrow \theta^1_1 + s^1_2 \delta \theta_2^1 = \theta^2_1 + s^2_2 \delta \theta_2^2  + \frac{4b}{a}\ell_2 \pi  \\
&\Leftrightarrow s^1_2 \delta \theta_2^1 =  s^2_2 \delta \theta_2^2 + \frac{4b}{a}(\ell_2 - \ell_1) \pi.
\end{align*}
The claim then follows in the $j=2$ case after relabelling the integer coefficient of $\frac{4b}{a}\pi$. The general claim follows similarly by induction.
\end{proof}

It follows that $F_{a,b}$ is \emph{generically injective} on $\mathrm{PL}(I,\C)/\{\mathrm{Tra},\mathrm{Rot}\}$ in the sense that if $c_1$ and $c_2$ are PL curves with $F_{a,b}(c_1) = F_{a,b}(c_2)$, then there exists an arbitrarily small perturbation $\widetilde{c}_2$ of $c_2$ such that the $\mathrm{SO}(2)$-orbit of $F_{a,b}\left(\widetilde{c}_2\right)$ is different than that of $F_{a,b}(c_1)$.  We also have the following corollary, which provides injectivity of $F_{a,b}$ on PL curves with bounded exterior angles.

\begin{cor}
The extended map $F_{a,b}$ is injective when restricted to the set of PL immersions with exterior angles uniformly bounded by $\frac{2b}{a}\pi$.
\end{cor}

\noindent In particular, note that $F_{a,b}$ is always injective on PL curves when $\frac{a}{2b} \geq 1$.

\begin{proof}
Let $c_1$ and $c_2$ be PL immersions such that $\delta \theta^1_j, \delta \theta^2_j < \frac{2b}{a}\pi$ for all $j$. If $F_{a,b}(c_1) = F_{a,b}(c_2)$, then Proposition \ref{prop:injectivity} implies that there exist integers $\ell_j$ such that
$$
\left| s_j^1 \delta \theta_j^1 - s_j^2 \delta \theta_j^2 \right| = \left| \frac{4b}{a}\ell_j \pi \right|
$$
for all $j$. We have
$$
\left| s_j^1 \delta \theta_j^1 - s_j^2 \delta \theta_j^2 \right| \leq \delta \theta_j^1 + \delta \theta_j^2 < \frac{4b}{a} \pi,
$$
so that each $\ell_j=0$. Therefore $\theta^1 = \theta^2$ and it follows that $c_1 = c_2$.
\end{proof}

For PL curves obtained by sampling a smooth curve, dense enough sampling ensures that exterior angles can be bounded by an arbitrarily small number. It follows that $F_{a,b}$ can be guaranteed to be injective. To quantify this, consider the simple case of a secant approximation of a smooth curve $c$ by an equilateral PL curve with edgelength $s$. At a vertex of the PL curve, let $\delta \theta$ denote the exterior angle, and let $\kappa$ denote the curvature of $c$ at that point. Then, we have the Taylor approximation (see \cite{anoshkina2002asymtotic})
$$
\delta \theta = s \kappa + O(s^3).
$$
Assuming that the PL curve has $N$ vertices, $c$ is normalized to have length $1$, and that $s$ is therefore roughly equal to $1/N$, it follows that there is an asymptotic bound
$$
|\delta \theta | \leq \frac{\kappa_{max}}{N} + O(1/N^3),
$$
where $\kappa_{max}$ is the maximum curvature of $c$. We therefore obtain the desired bound on the turning angle as soon as
$$
N > \frac{a}{2b \pi} \left(\kappa_{max} + O(1/N^2)\right).
$$

\subsection{Exact Matching}

The algorithm for computing geodesic distances in the shape space $\mathrm{Imm}(I,\C)/\mathrm{Diff}^+(I)$ outlined in Section \ref{sec:optimized_geodesics} calls for a solution of the optimization problem \eqref{eqn:optimal_reparameterization}. We now wish to demonstrate the existence of solutions to the corresponding problem in the PL setting. In order to achieve a solution, we replace the smooth diffeomorphism group $\mathrm{Diff}^+(I)$ with the semigroup
$$
\overline{\Gamma} = \{\gamma \in \mathrm{AC}(I,I) \mid \gamma(0)=0, \, \gamma(1)=1,\, \gamma'(t) \geq 0 \mbox{ when $\gamma'(t)$ exists }\}.
$$
The PL optimization problem seeks the optimal reparameterization $\gamma$ for PL curves $c_1$ and $c_2$ satisfying
\begin{equation}\label{eqn:PL_optimization}
\gamma = \mathrm{arginf} \{\mathrm{dist}_{g^{a,b}}(c_1,c_2(\gamma)) \mid \gamma \in \overline{\Gamma}\}.
\end{equation}
Because the $F_{a,b}$-transform induces the same optimization problem for general elastic metrics $g^{a,b}$ as the one studied in the SRVF setting $g^{1,1/2}$, we are able to directly appeal to recent work of Lahiri, Robinson and Klassen \cite{lahiri2015precise}.

\begin{prop}
The optimal reparameterization between two PL curves $[c_1]$ and $[c_2]$ with respect to $g^{a,b}$ is realized by a PL function in $\overline{\Gamma}$.
\end{prop}

\begin{proof}
The images of the PL curve $c_j$ under $F_{a,b}$ are piecewise constant maps $q_j$ into $\C\approx \R^2$. The reparameterization action of the semigroup $\overline{\Gamma}$ on $c_j$ transforms into an action on the image curves by the formula
$$
\gamma \ast q_j = \sqrt{\gamma'} \cdot q_j \circ \gamma,
$$
defined pointwise almost everywhere. The optimization problem \eqref{eqn:PL_optimization} becomes
$$
\inf_{\gamma \in \overline{\Gamma}} \|q_1 - \gamma \ast q_2 \|_{L^2}.
$$
It follows from \cite[Theorem 5]{lahiri2015precise} that this new optimization problem is solved by a PL element of $\overline{\Gamma}$.
\end{proof}

\begin{figure}[!t]
\begin{center}
\begin{tabular}{|c|c|c|c|}
\hline
$\frac{a}{2b}$=2.00&$\frac{a}{2b}$=1.00&$\frac{a}{2b}$=0.50&$\frac{a}{2b}$=0.17\\
\hline
\includegraphics[width=1in]{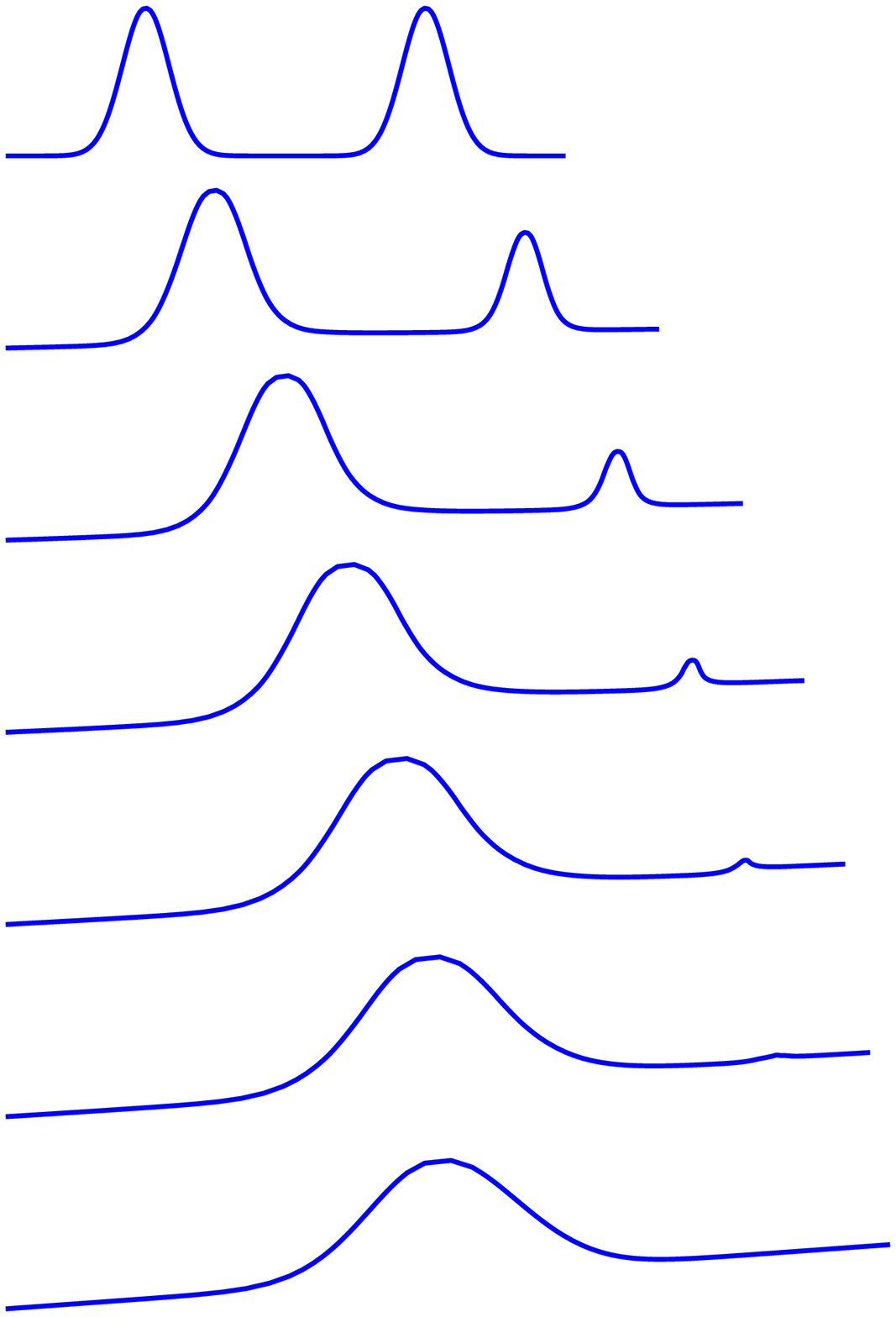}&\includegraphics[width=1in]{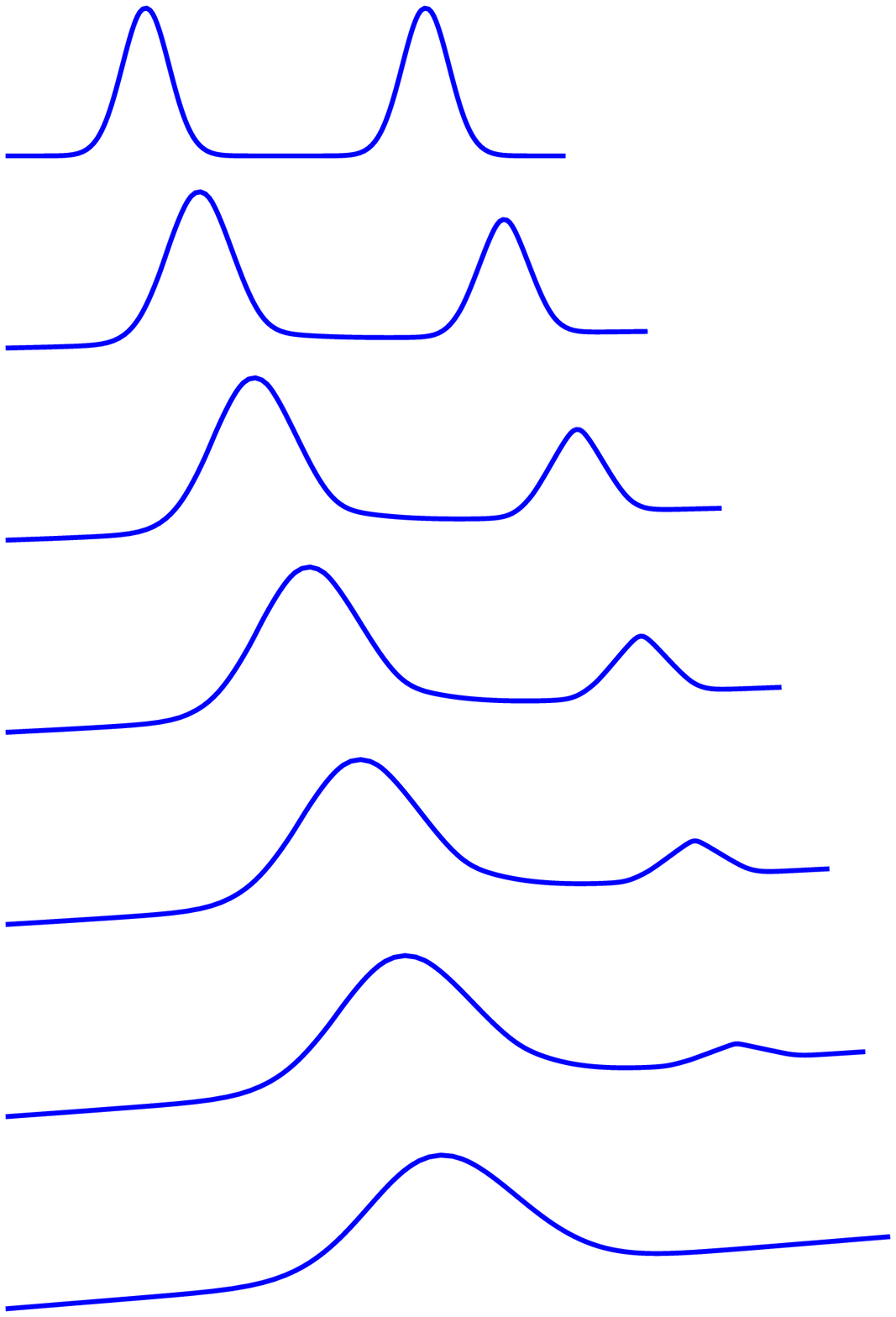}&\includegraphics[width=1in]{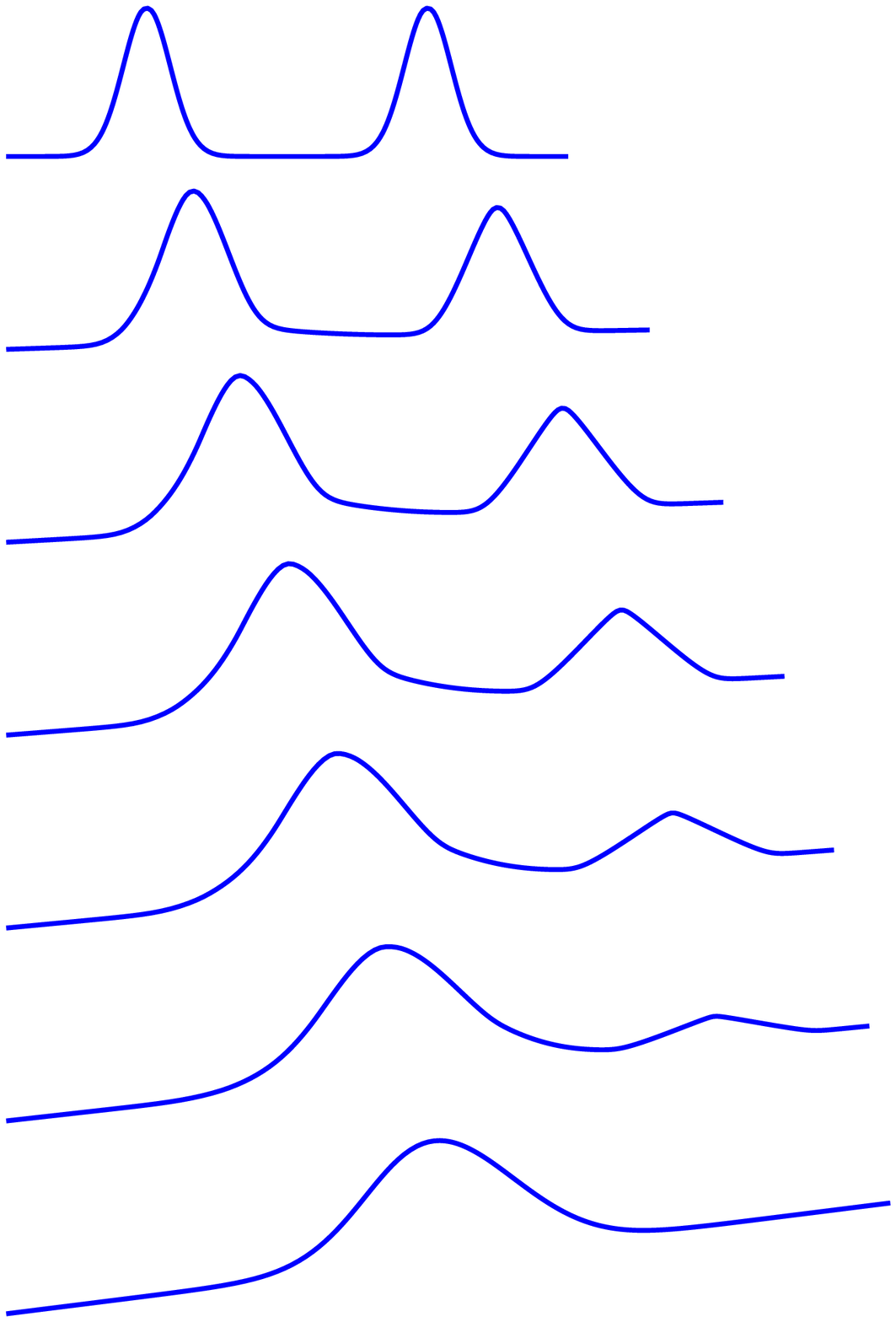}&\includegraphics[width=1in]{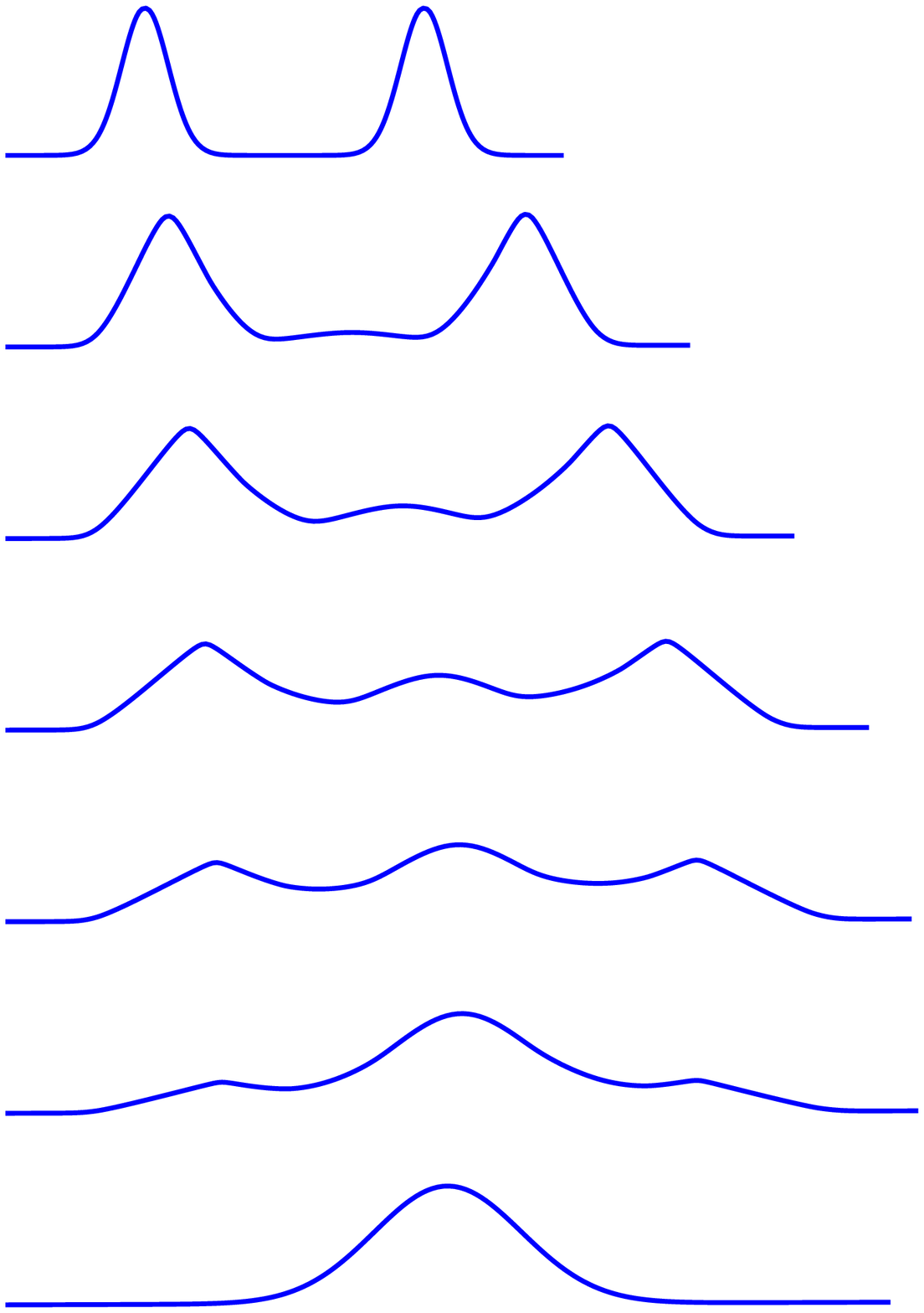}\\
\hline
\includegraphics[width=1in]{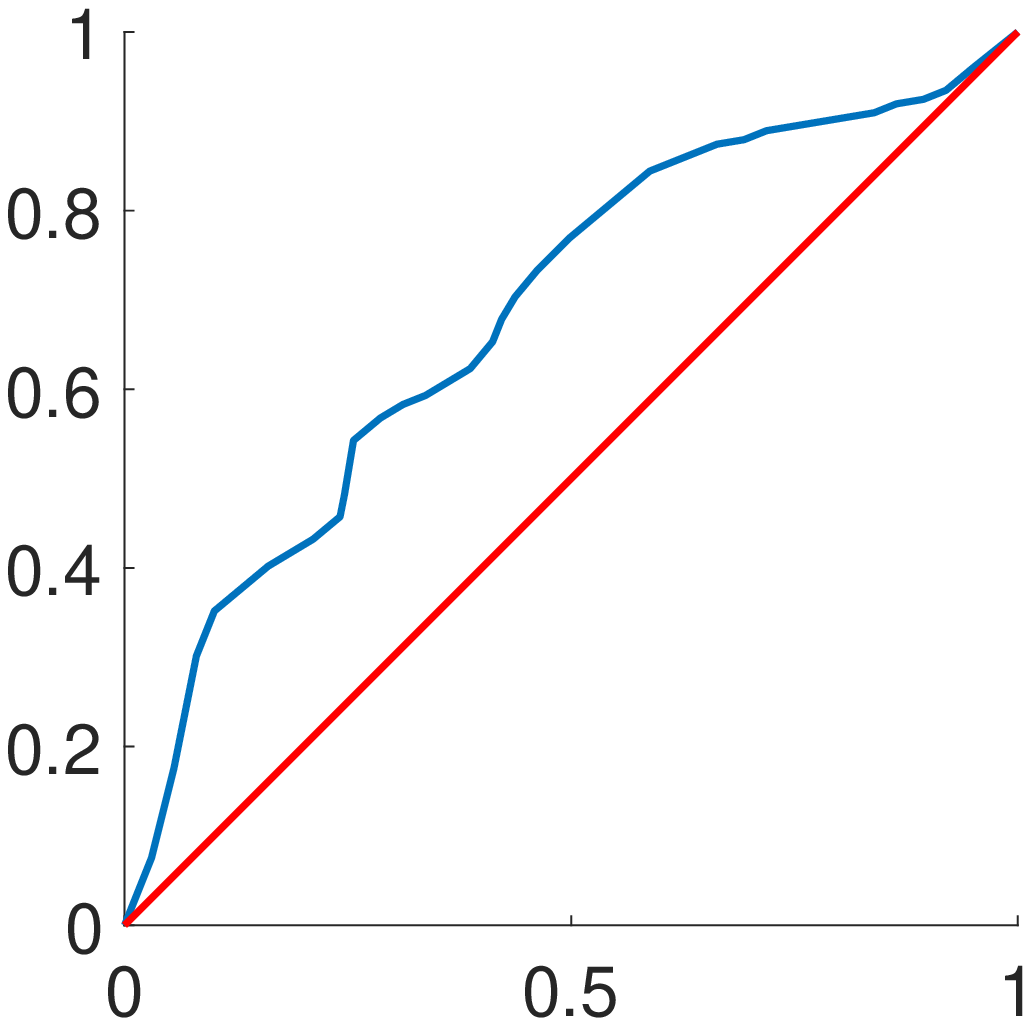}&\includegraphics[width=1in]{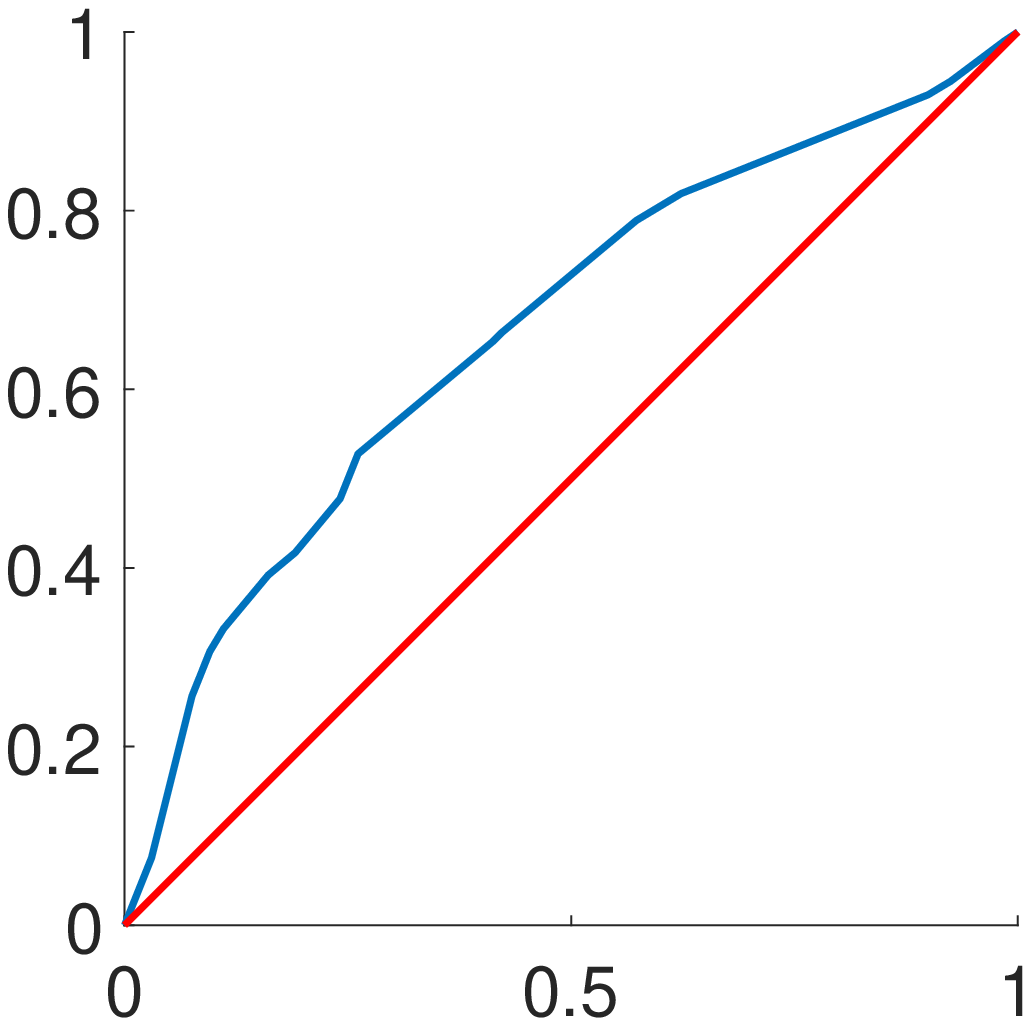}&\includegraphics[width=1in]{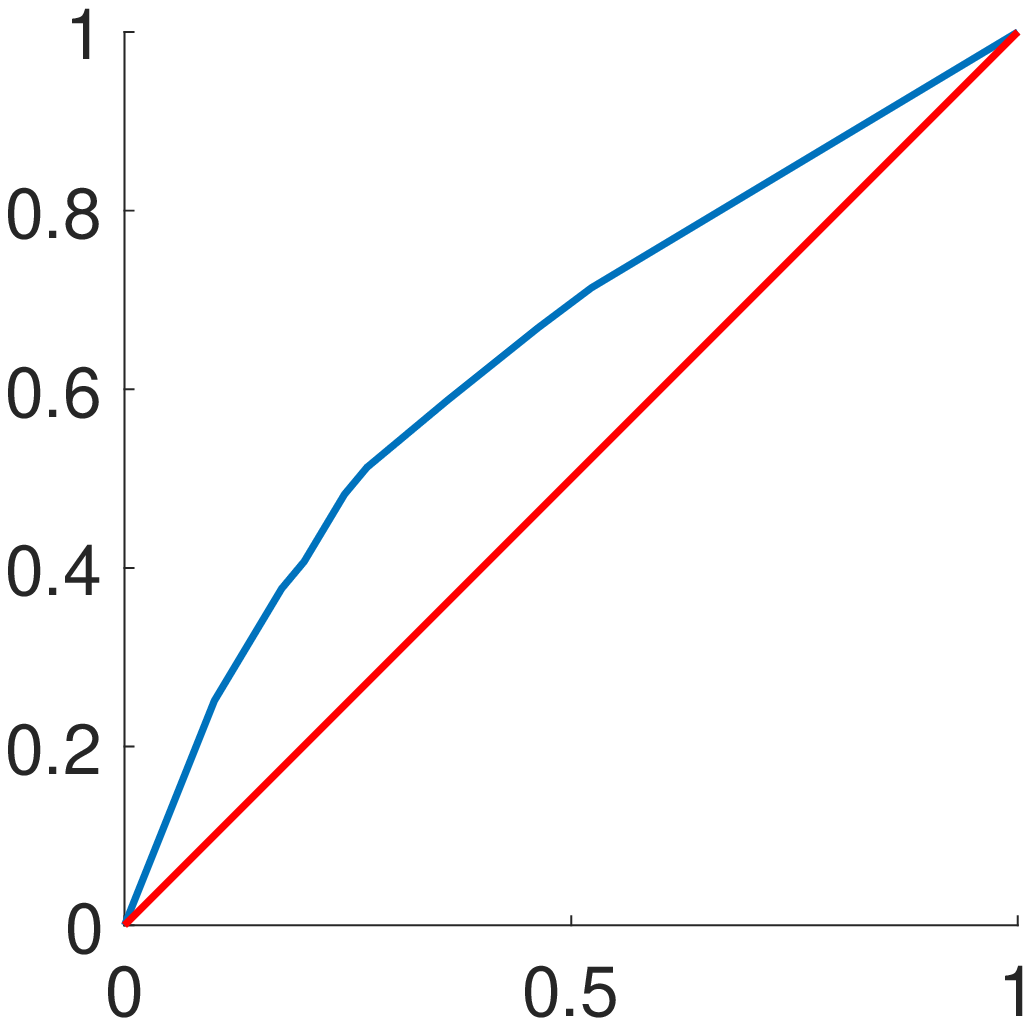}&\includegraphics[width=1in]{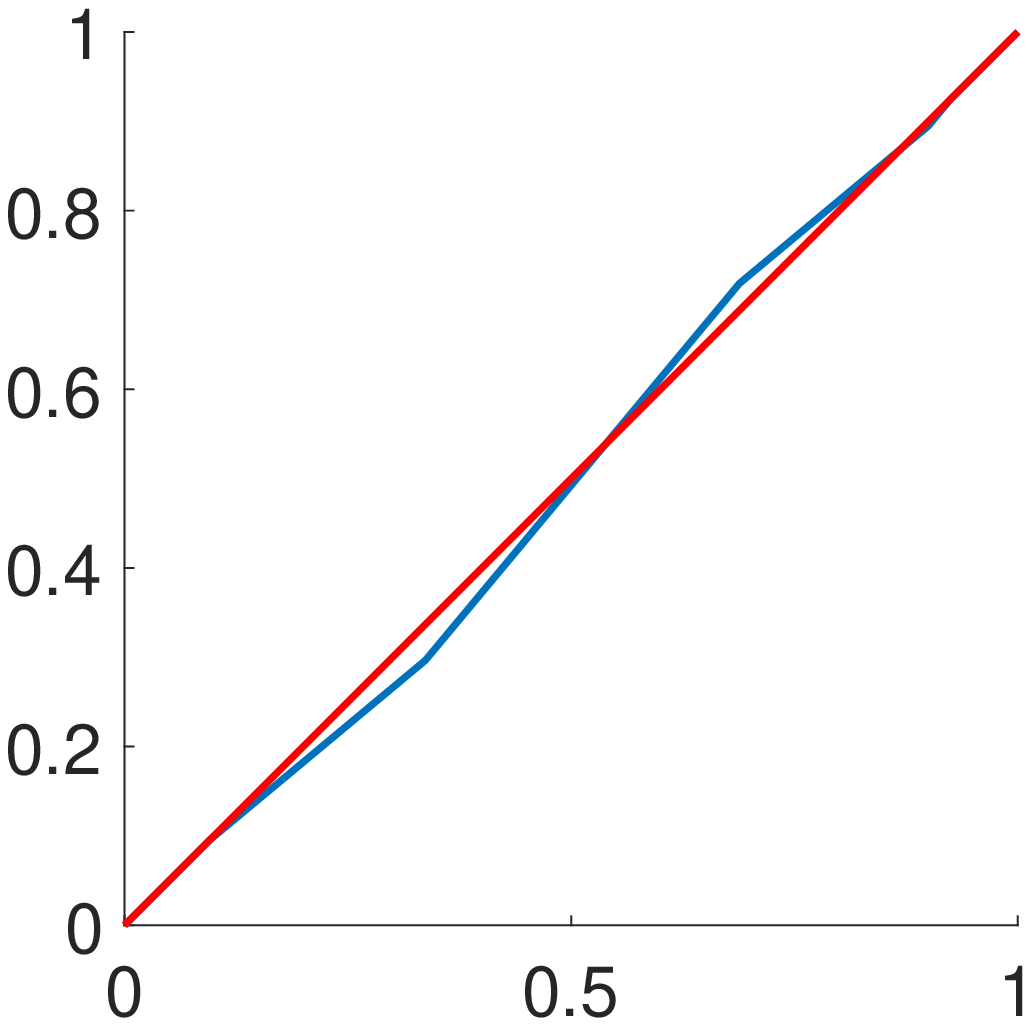}\\
\hline
dist=0.5277&dist=0.7103&dist=0.9060&dist=1.1954\\
\hline
\end{tabular}
\end{center}
\caption{Geodesic path and distance between shapes of two artificial open curves.}\label{fig:open_geodesic1}
\label{fig:openex1}
\end{figure}

\section{Numerical Experiments}\label{sec:applications}

\begin{figure}[!t]
\begin{center}
\begin{tabular}{|c|c|c|c|}
\hline
$\frac{a}{2b}$=2.00&$\frac{a}{2b}$=1.00&$\frac{a}{2b}$=0.50&$\frac{a}{2b}$=0.17\\
\hline
\includegraphics[width=1in]{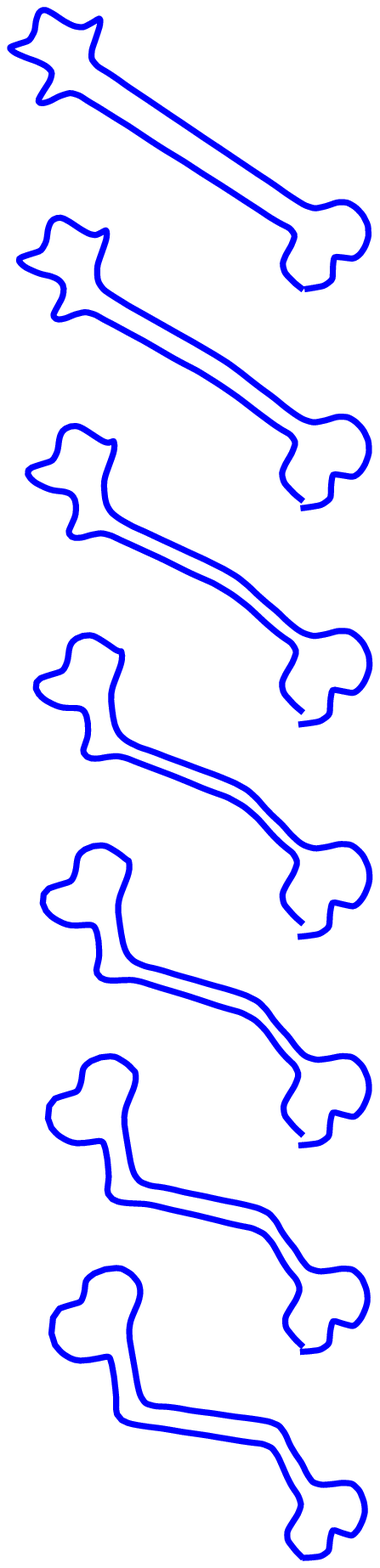}&\includegraphics[width=1in]{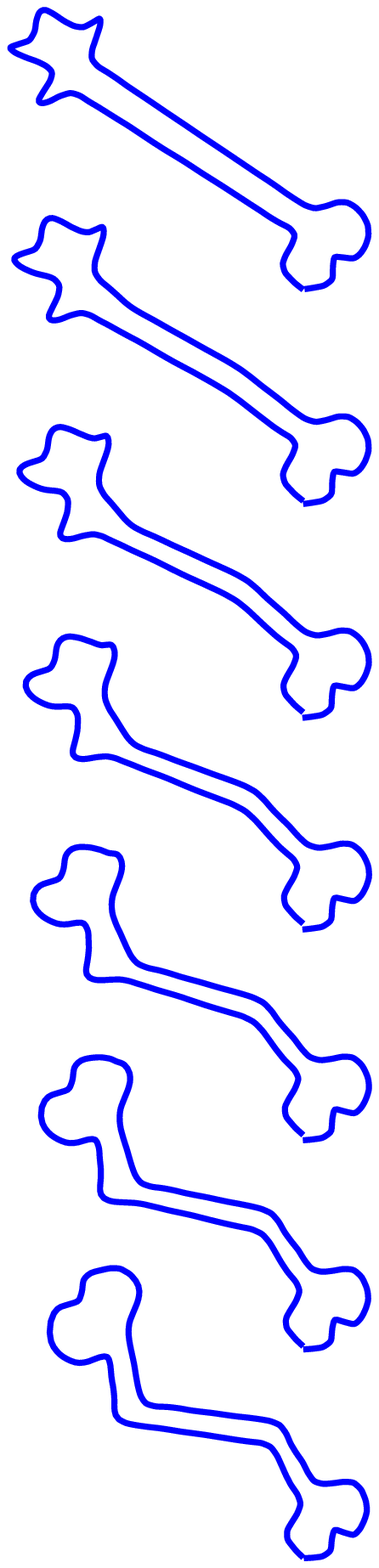}&\includegraphics[width=1in]{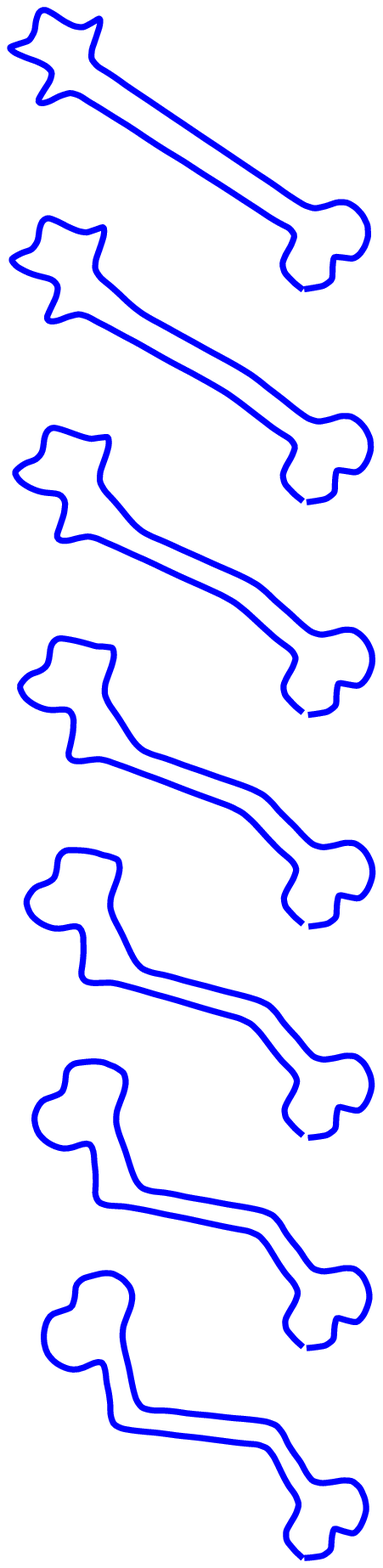}&\includegraphics[width=1in]{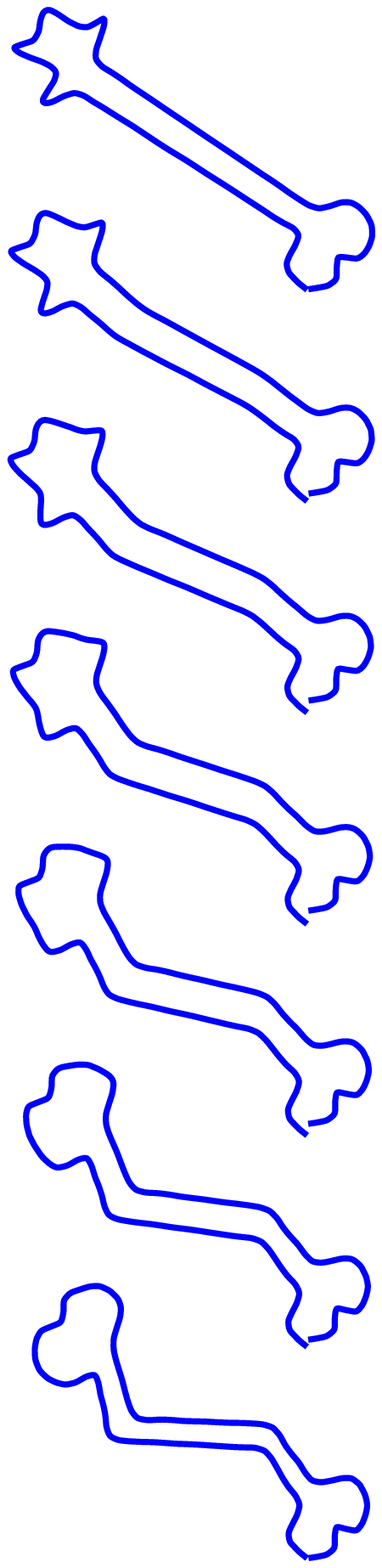}\\
\hline
\includegraphics[width=1in]{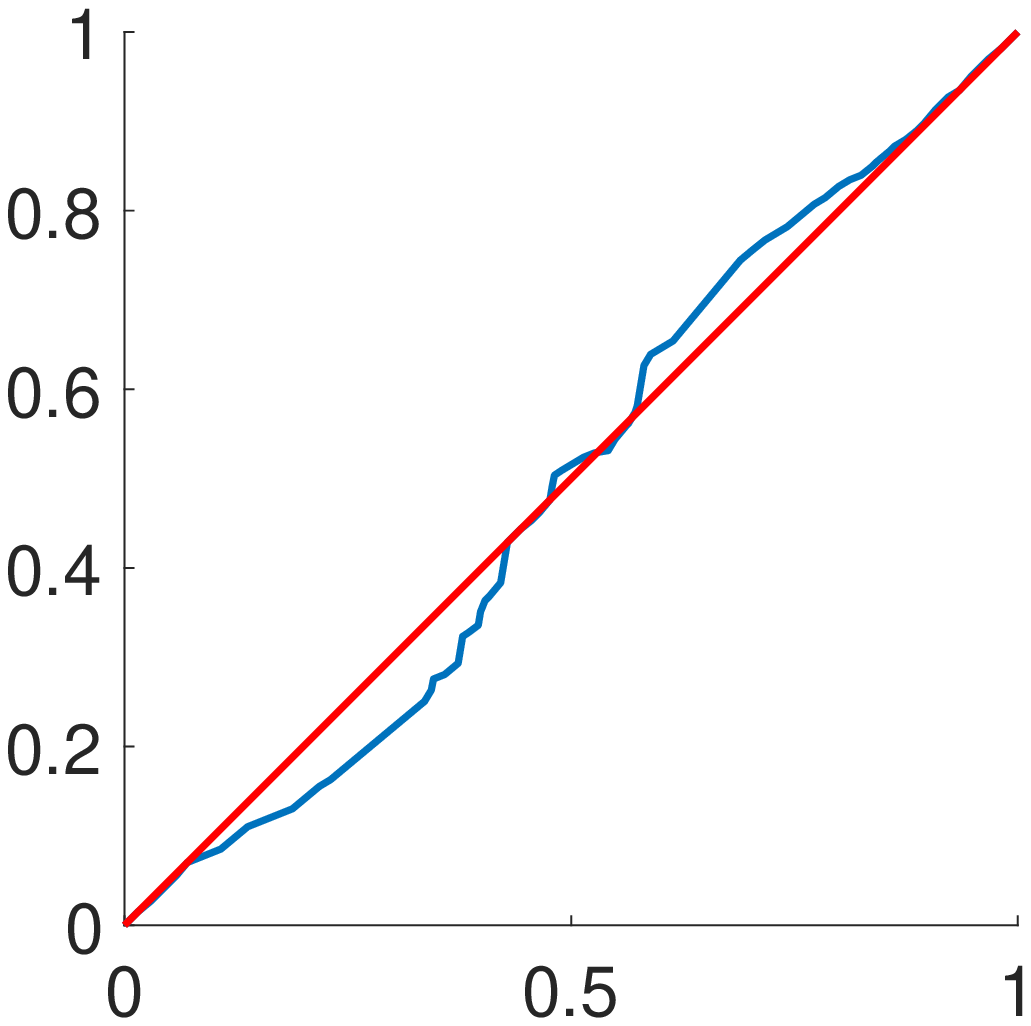}&\includegraphics[width=1in]{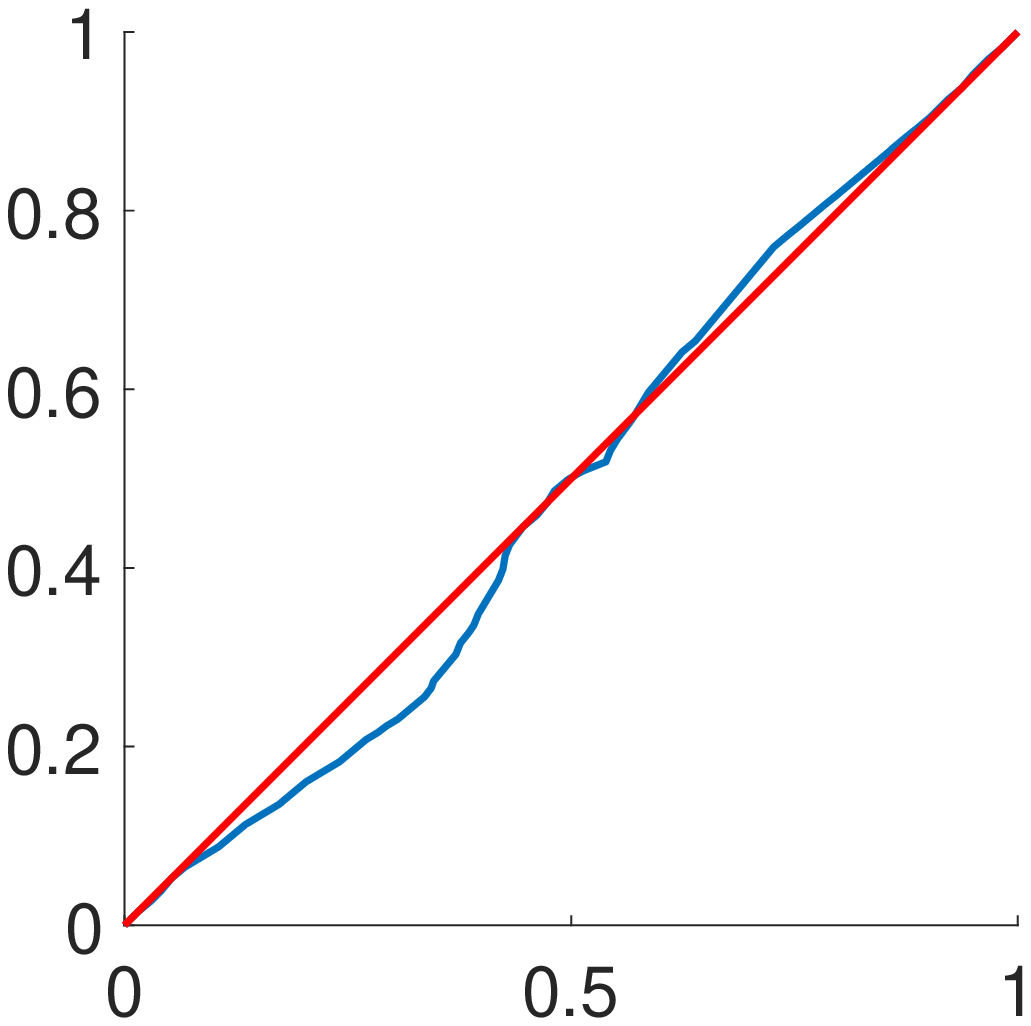}&\includegraphics[width=1in]{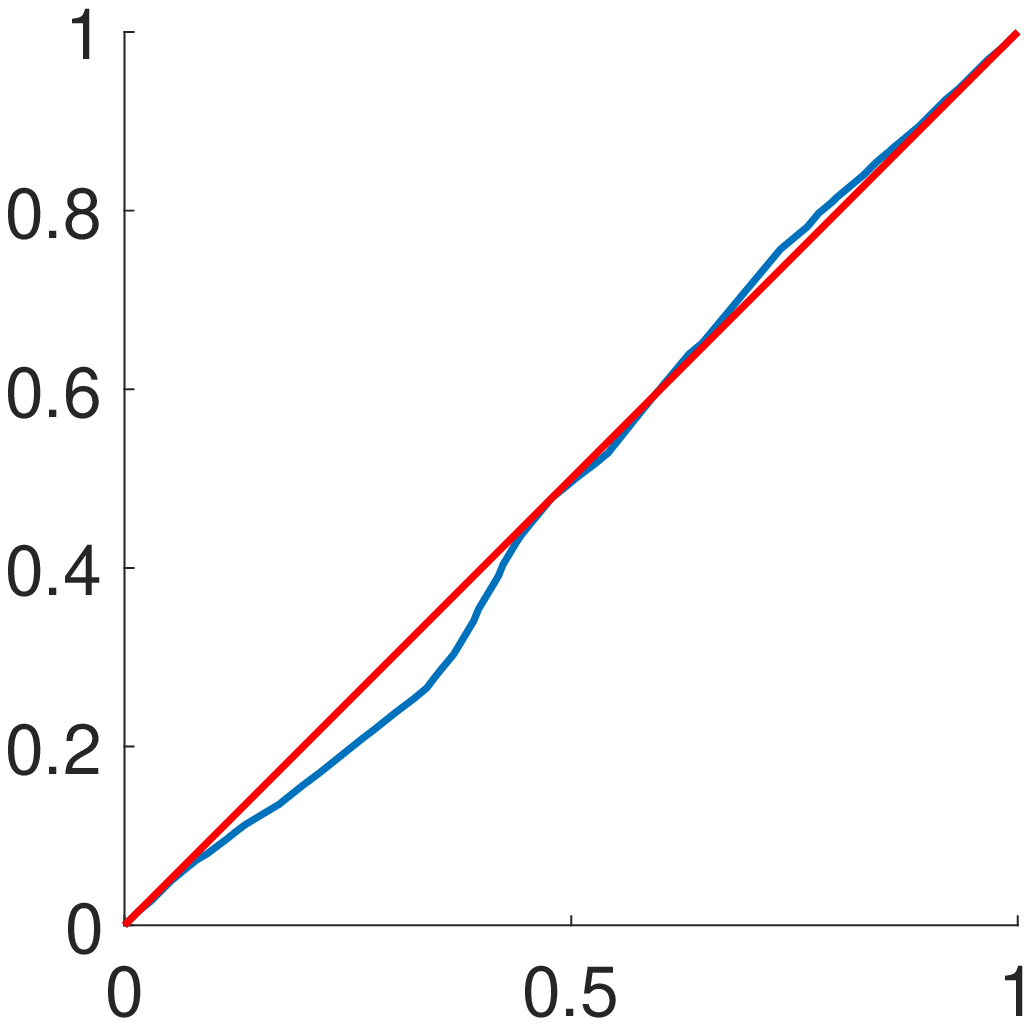}&\includegraphics[width=1in]{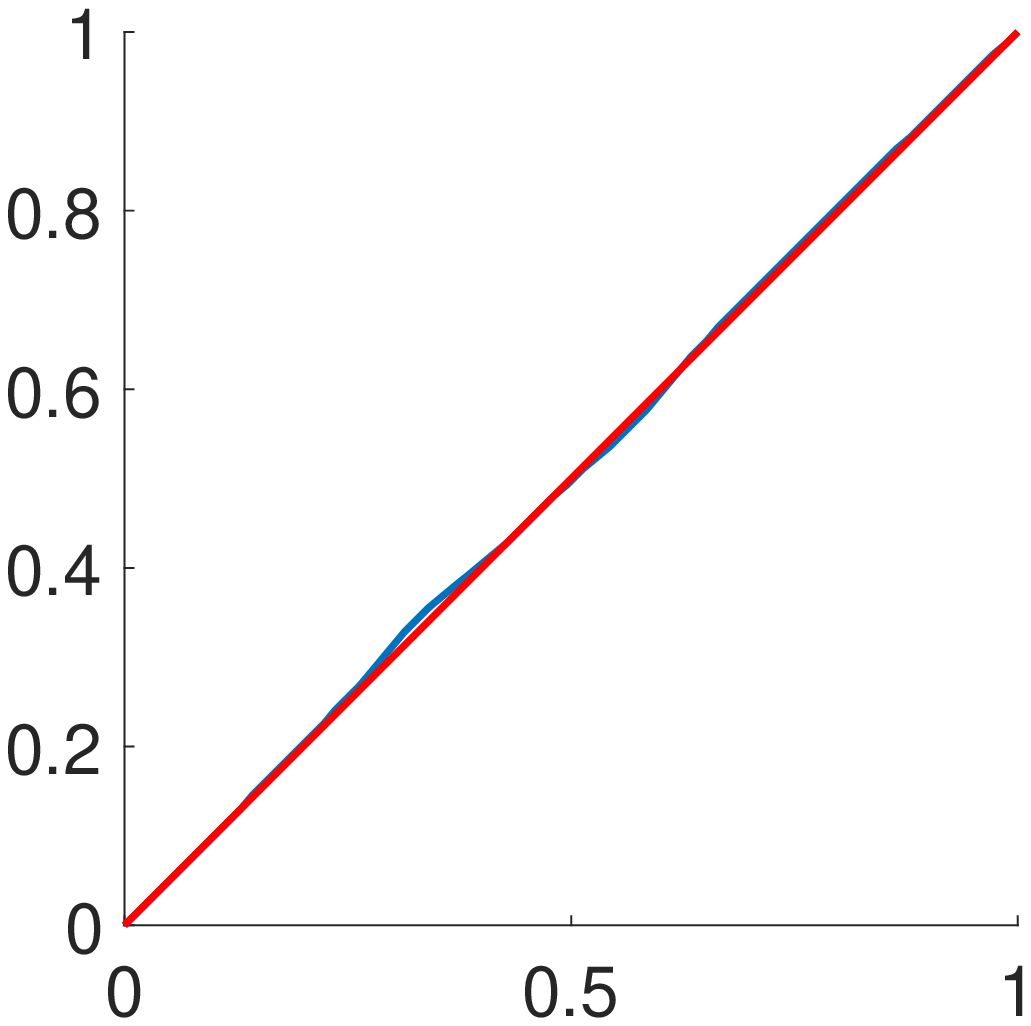}\\
\hline
dist=0.3783&dist=0.4534&dist=0.5342&dist=0.5773\\
\hline
\end{tabular}
\end{center}
\caption{Geodesic path and distance between two structurally different bone shapes.}\label{fig:open_geodesic2}
\label{fig:openex3}
\end{figure}

\subsection{Implementation Issue}\label{sec:geodesic_algorithm}

\begin{figure}[!t]
\begin{center}
\begin{tabular}{|c|c|c|c|}
\hline
$\frac{a}{2b}$=2.00&$\frac{a}{2b}$=1.00&$\frac{a}{2b}$=0.50&$\frac{a}{2b}$=0.17\\
\hline
\includegraphics[width=1in]{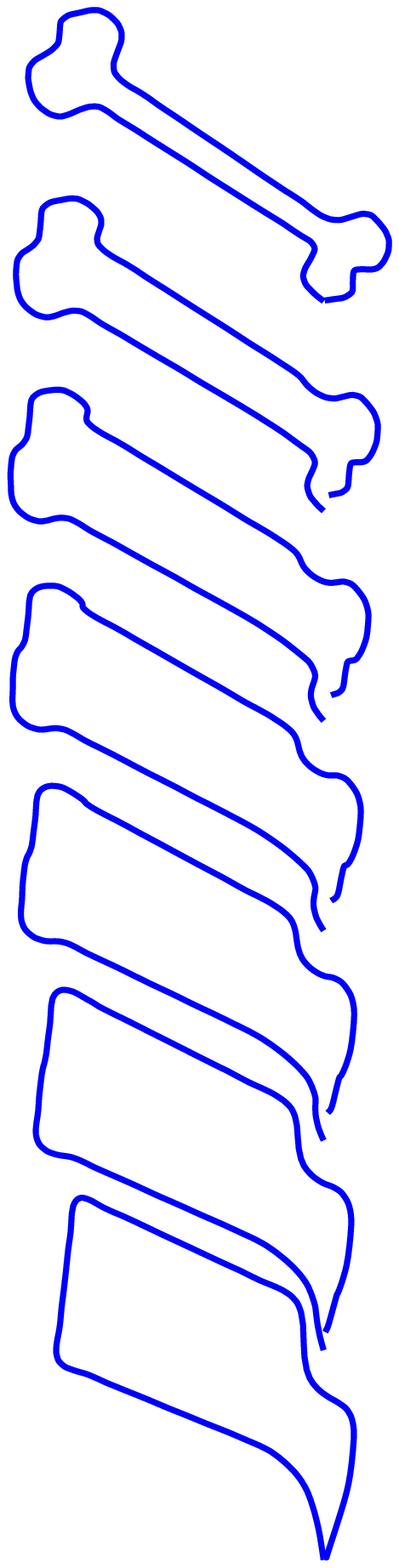}&\includegraphics[width=1in]{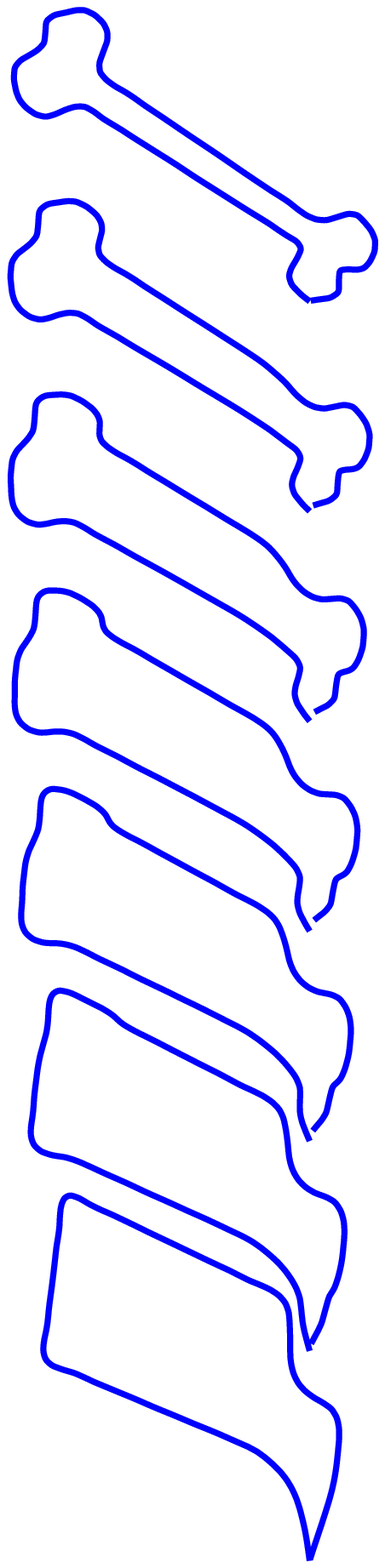}&\includegraphics[width=1in]{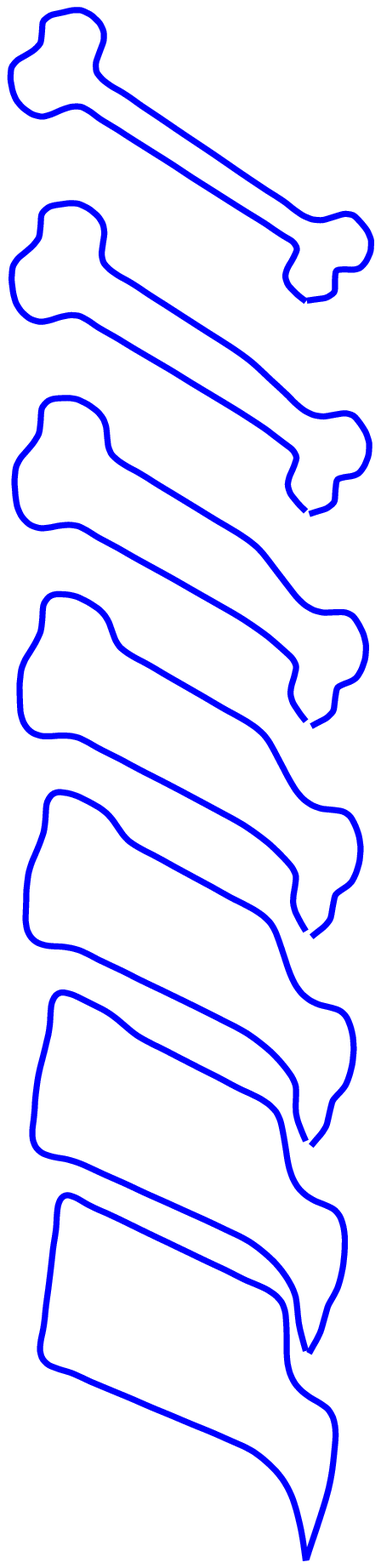}&\includegraphics[width=1in]{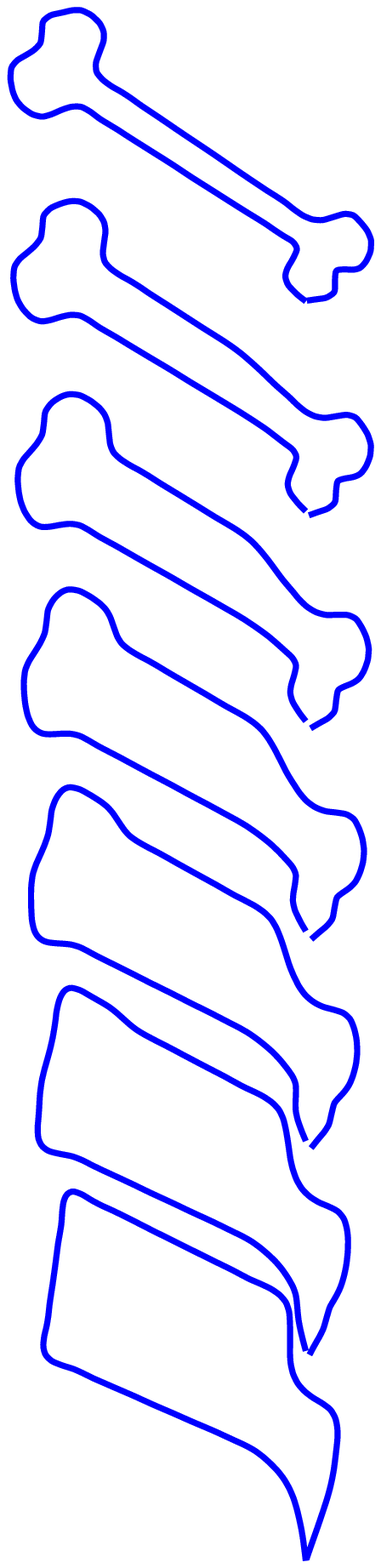}\\
\hline
\includegraphics[width=1in]{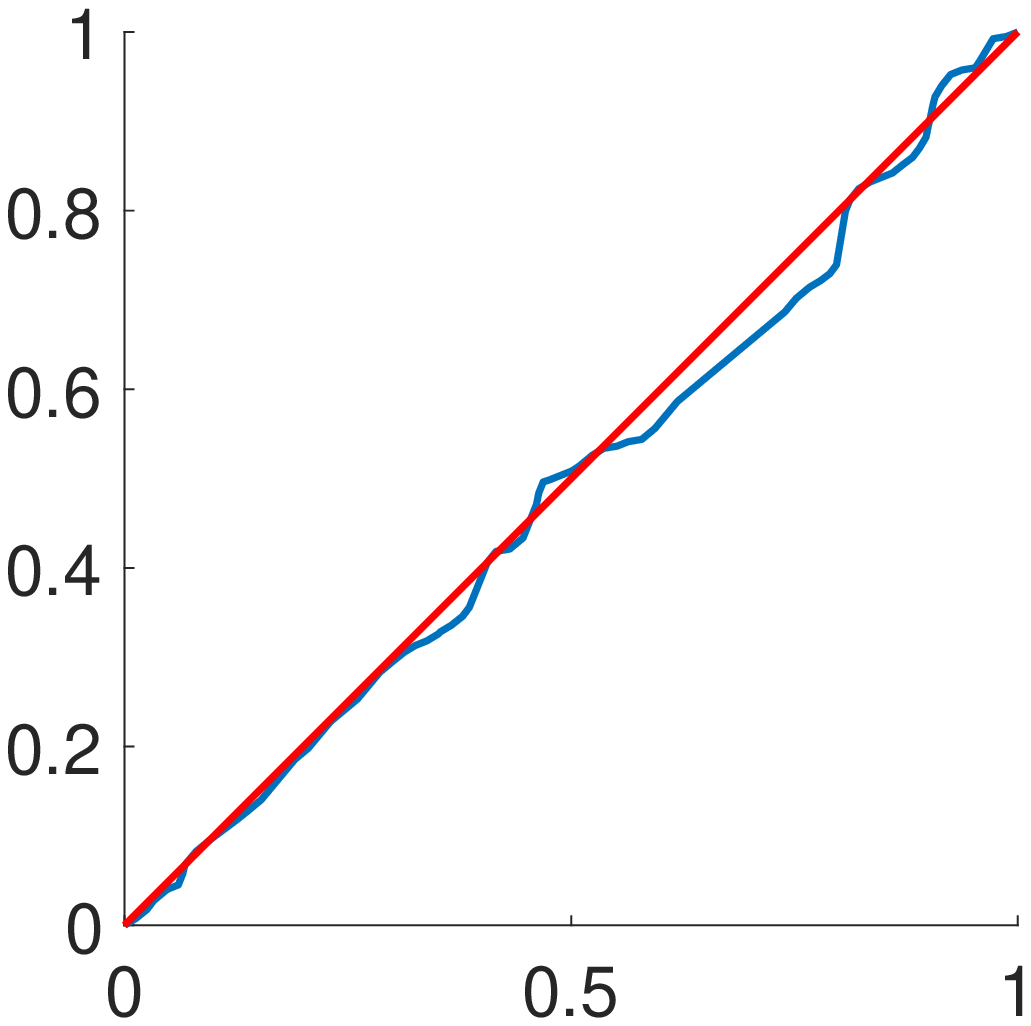}&\includegraphics[width=1in]{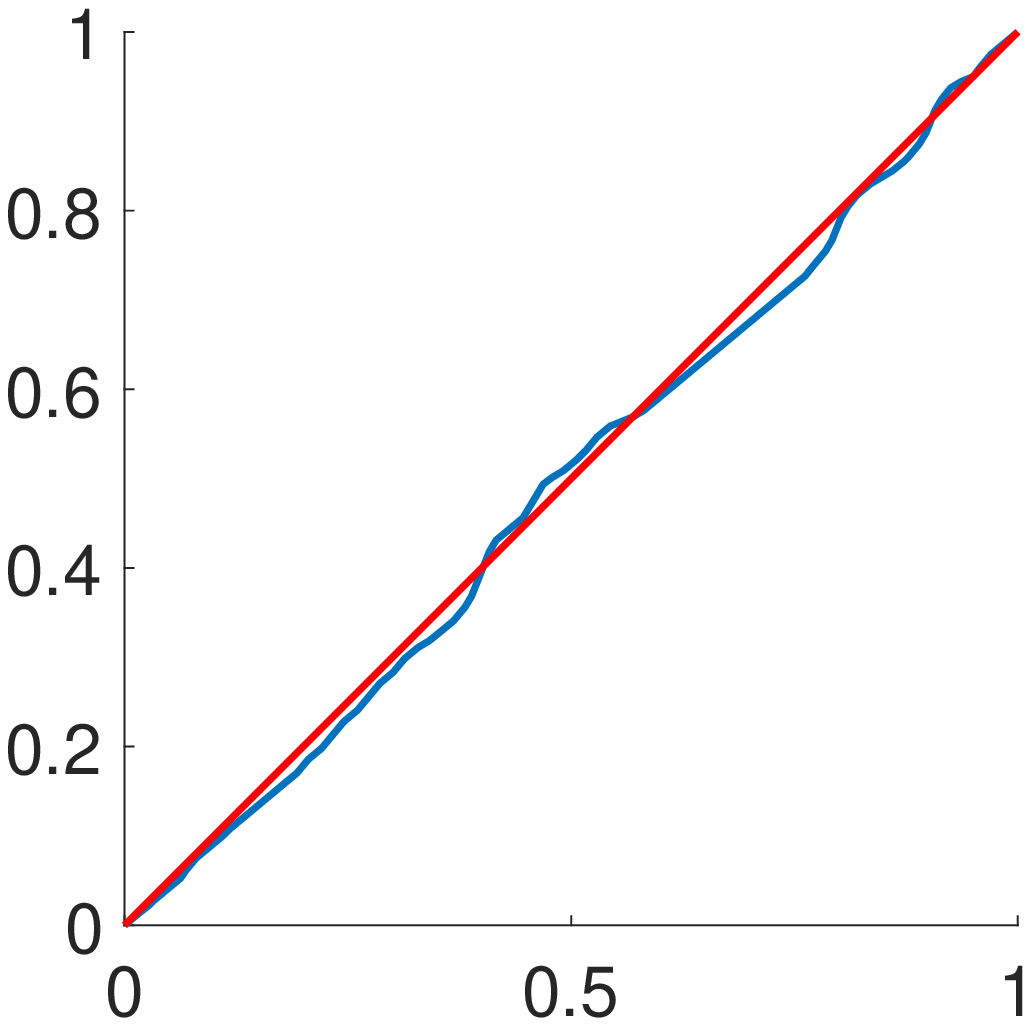}&\includegraphics[width=1in]{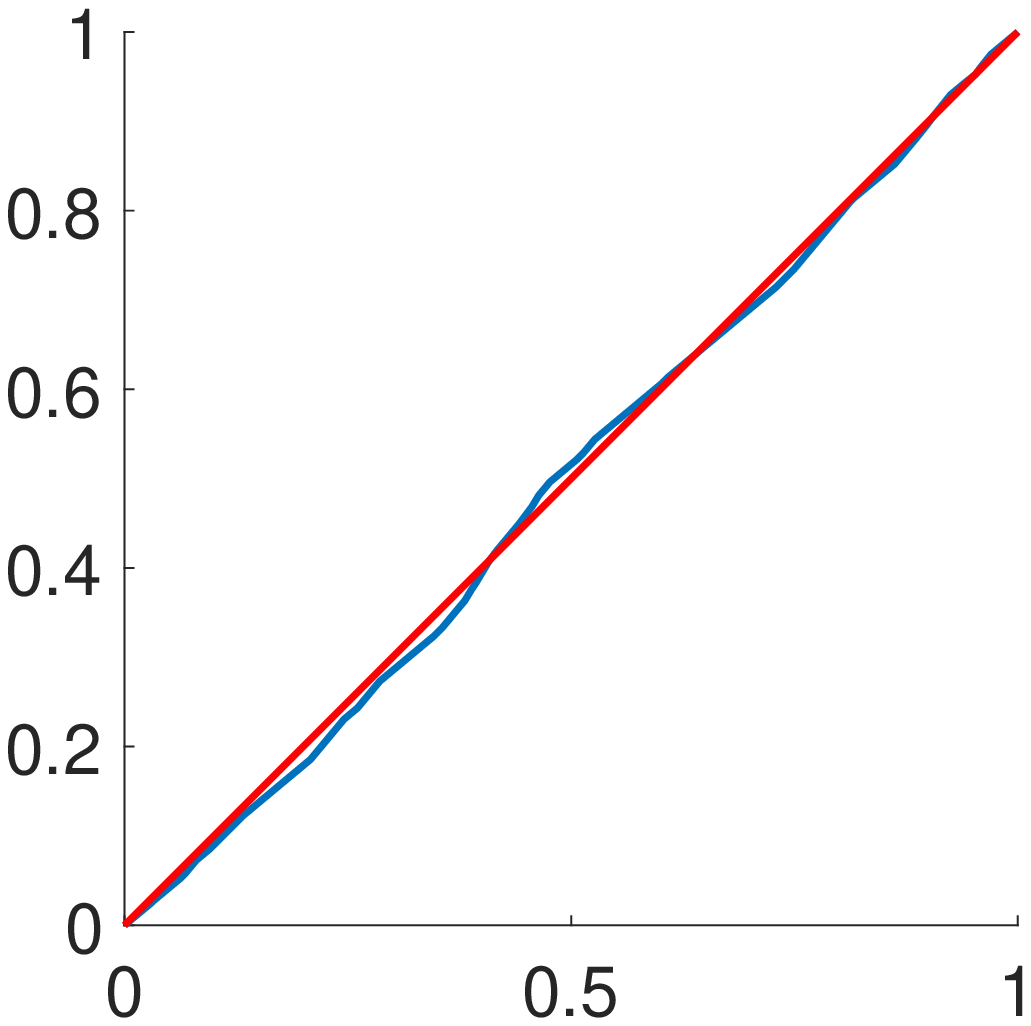}&\includegraphics[width=1in]{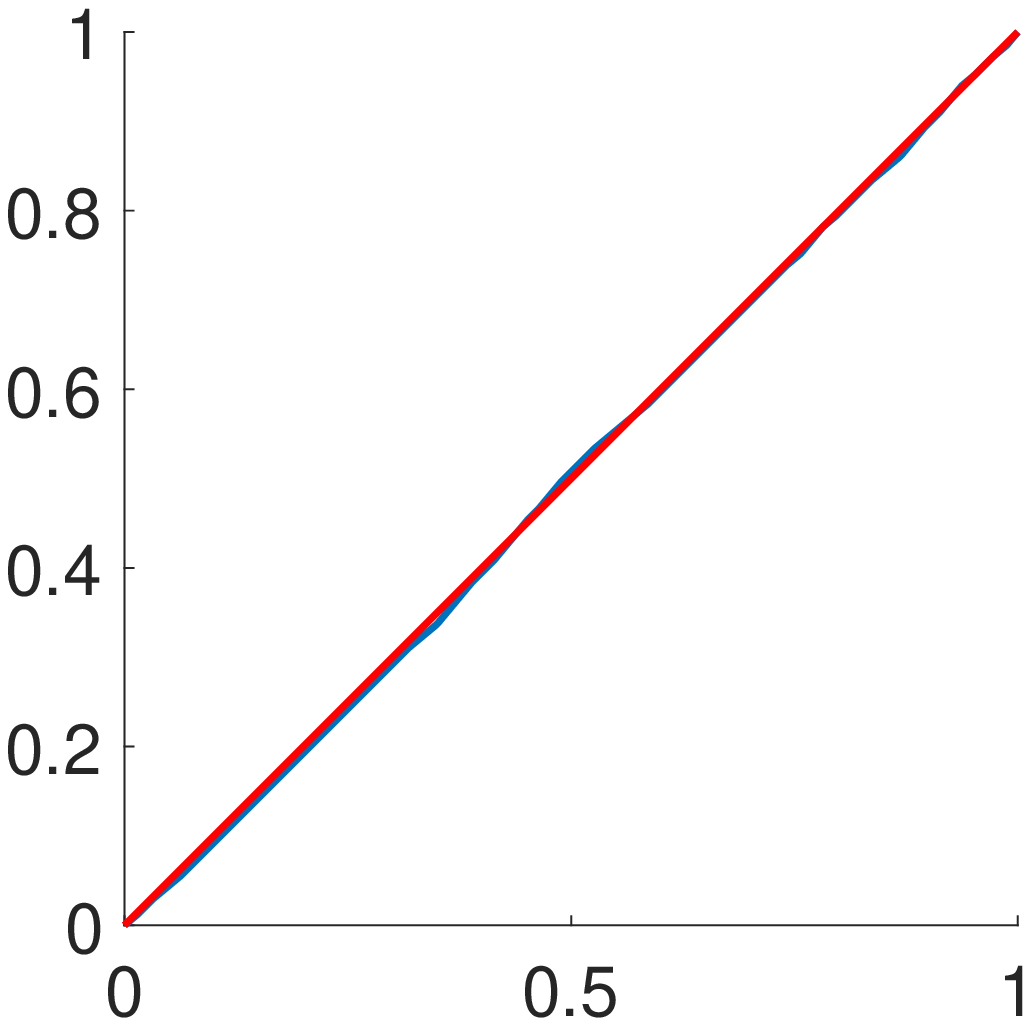}\\
\hline
dist=0.4129&dist=0.5249&dist=0.5855&dist=1.6206\\
\hline
\end{tabular}
\end{center}
\caption{Geodesic path and distance between two very different shapes.}\label{fig:open_geodesic3}
\label{fig:openex4}
\end{figure}

In the case when $\frac{a}{2b}>1$, the inverse mapping from the $F_{a,b}$-transform space to the space of curves is not guaranteed to produce valid angle functions when the curves are expressed in polar coordinates. This is due to large local differences between the angle functions for the shapes being compared. In these cases, we use a path-straightening algorithm to find the appropriate angle functions along the geodesic path. We omit the details of this algorithm for brevity, but note that the $F_{a,b}$ transform still provides a significant numerical simplification in this case as it allows us to search for optimal rotations and reparameterizations in the transform space under the $L^2$-metric. Furthermore, the resulting radius functions are valid, which results in a very simple and computationally efficient path-straightening algorithm.

\begin{figure}[!t]
\begin{center}
\begin{tabular}{|c|c|c|c|}
\hline
$\frac{a}{2b}$=1.25&$\frac{a}{2b}$=1.00&$\frac{a}{2b}$=0.50&$\frac{a}{2b}$=0.17\\
\hline
\includegraphics[width=.75in]{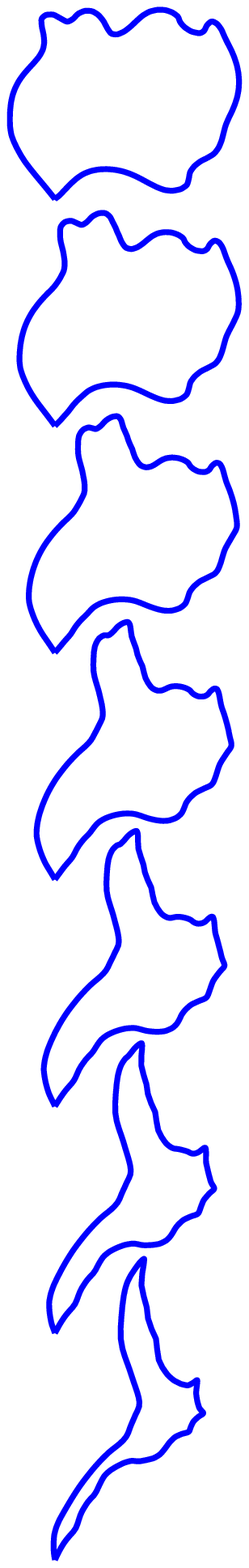}&\includegraphics[width=.75in]{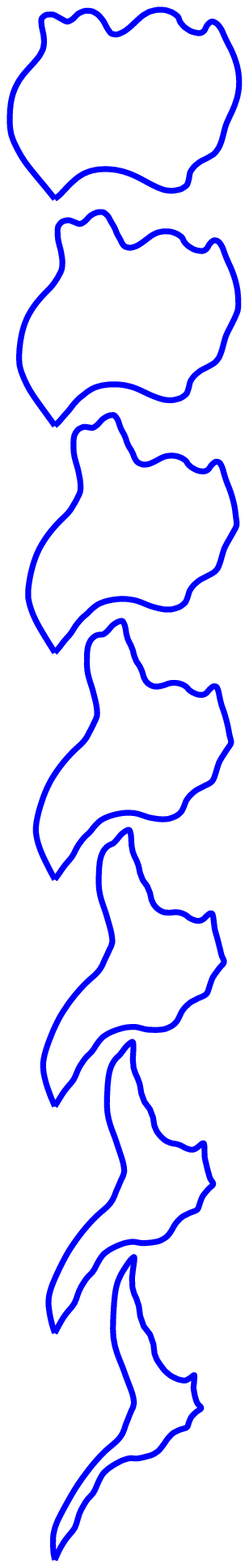}&\includegraphics[width=.75in]{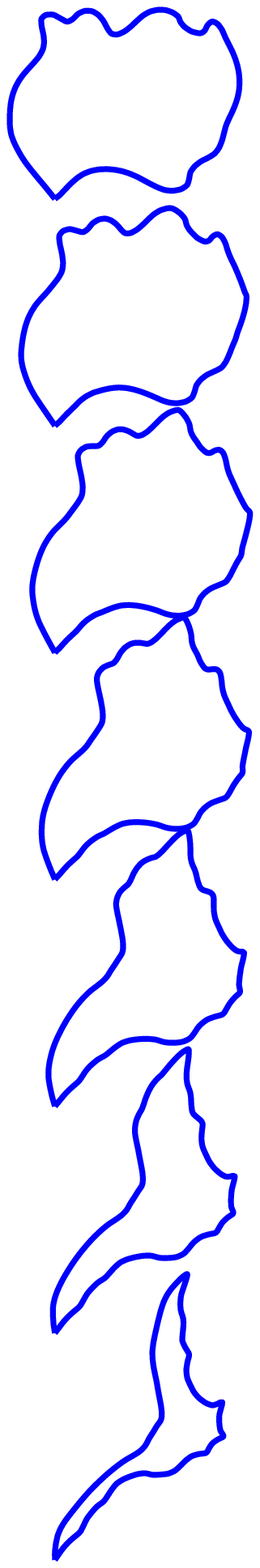}&\includegraphics[width=.75in]{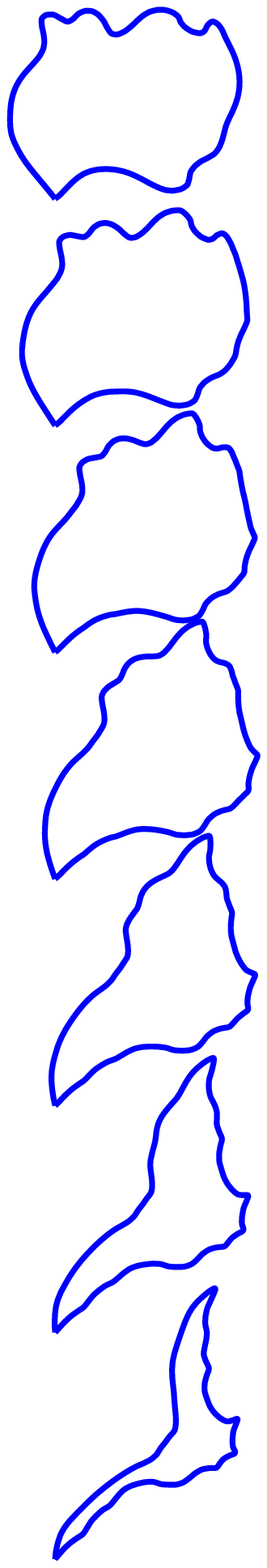}\\
\hline
\includegraphics[width=1in]{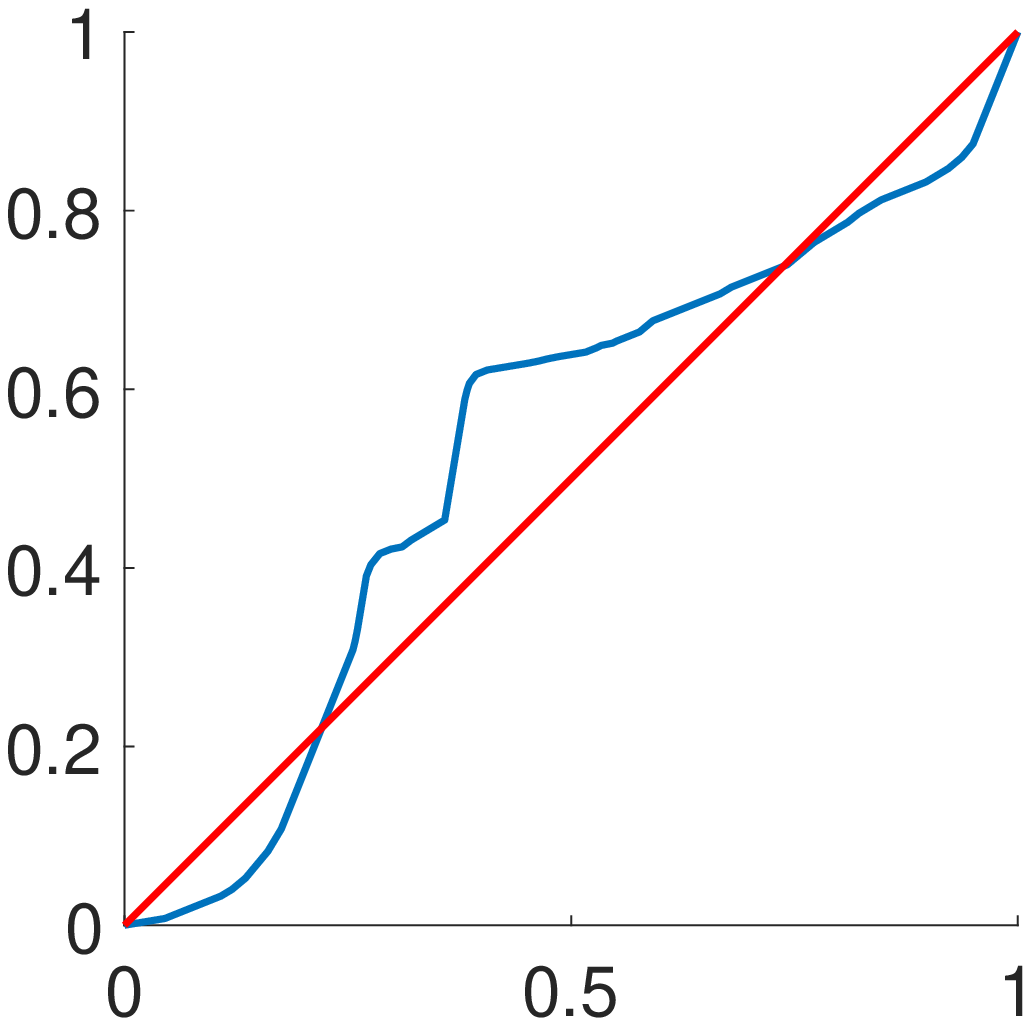}&\includegraphics[width=1in]{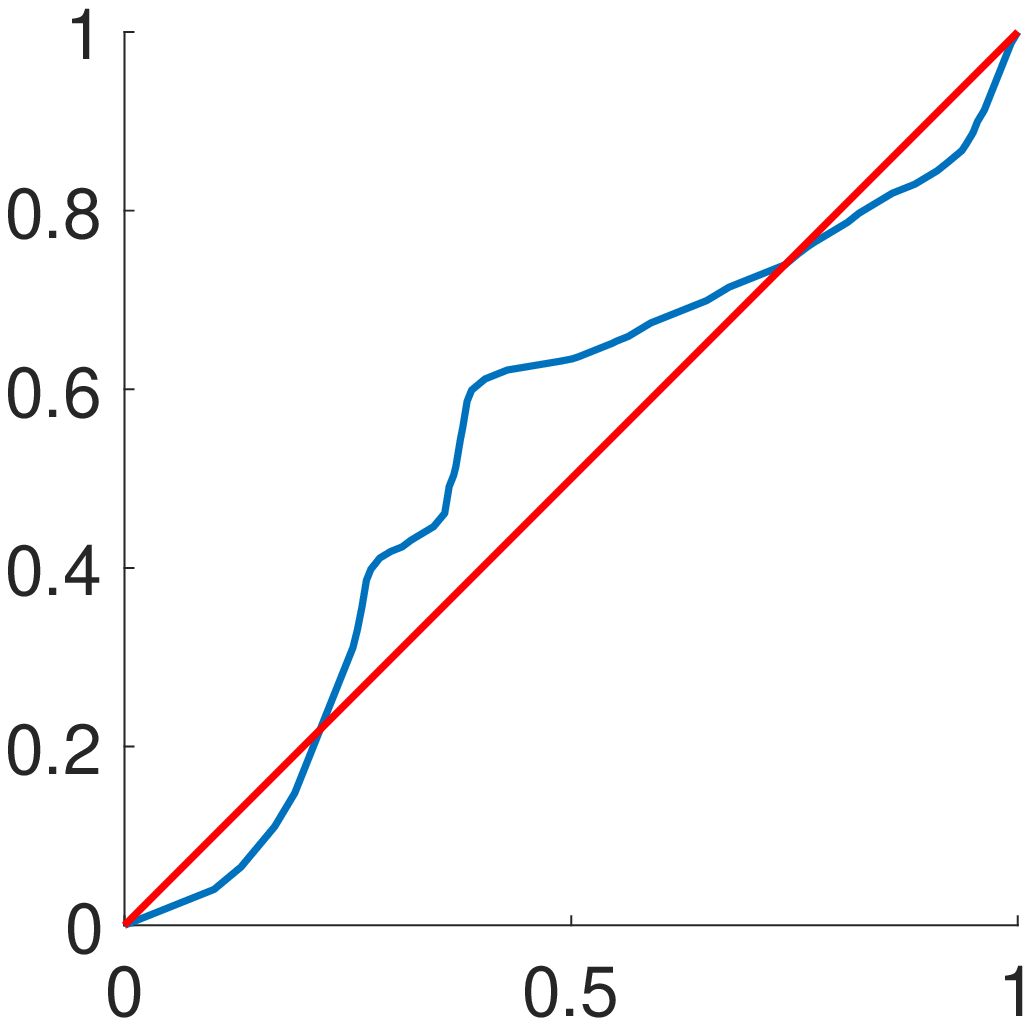}&\includegraphics[width=1in]{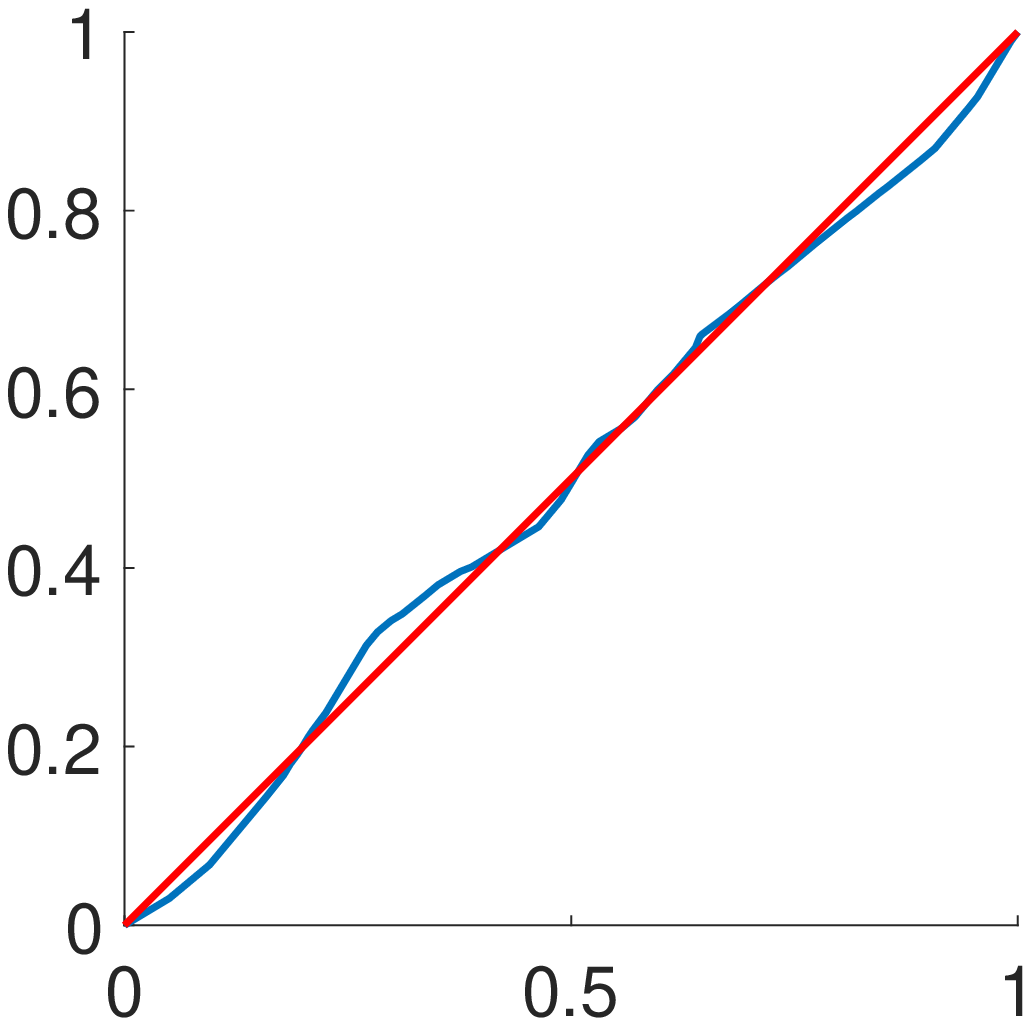}&\includegraphics[width=1in]{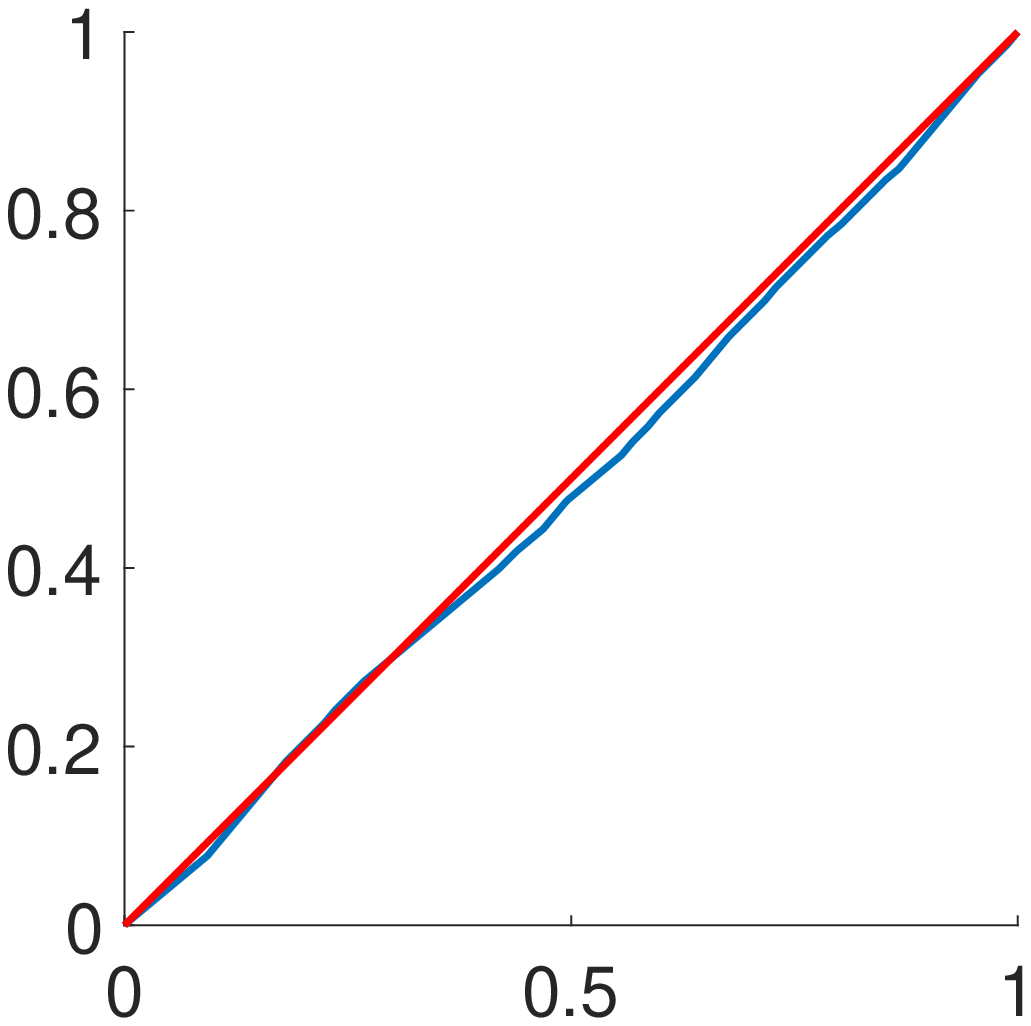}\\
\hline
dist=0.6016&dist=0.6792&dist=0.9109&dist=1.0344\\
\hline
\end{tabular}
\end{center}
\caption{Geodesic path and distance between two very different shapes.}\label{fig:closed_geodesic1}
\label{fig:closedex1}
\end{figure}

\subsection{Examples}\label{sec:examples}

We show several examples of geodesic paths and associated geodesic distances between shapes of open (Figures \ref{fig:open_geodesic1}---\ref{fig:open_geodesic3}) and closed (Figures \ref{fig:closed_geodesic1} and \ref{fig:closed_geodesic2}) curves for different parameter values in the elastic metric, computed via the methods described in Section \ref{sec:computing_geodesics} and \ref{sec:geodesic_algorithm}. Each figure shows the curve evolution and the optimal reparameterization in blue with the identity in red for comparison. While the example in Figure \ref{fig:open_geodesic1} considers two simulated open curves, the examples in Figures \ref{fig:open_geodesic2}-\ref{fig:closed_geodesic2} use curves from the well-known MPEG-7 computer vision database\footnote{http://www.dabi.temple.edu/~shape/MPEG7/dataset.html}.

We also provide a short classification experiment that shows the benefits of using general weights for the stretching and bending terms in the context of differentiating forgeries from genuine signatures. The data used here consists of 40 different signatures, which are a subset of the SVC 2004 dataset \cite{yeung2004}. Each signature class contains 20 genuine signatures and 20 skilled forgeries. We classified each signature as genuine or a forgery using leave-one-out nearest neighbor with respect to the geodesic distance under the $g^{a,b}$-metric for various parameter choices $(a,b)$. As a baseline, signatures were classified by the basic method of $L^2$ distance between arclength parameterized curves. Table \ref{fig:classfication1} reports the overall classification rate (across all 40 signature classes) for each metric, as well as the number of signature classes where classification was perfect. The first group of three results contain parameter values covered by the $R_{a,b}$-transform of \cite{bauer2014constructing}, with $\frac{a}{2b}=1$ corresponding to the SRVF and $\frac{a}{2b}=\frac{1}{2}$ corresponding to the complex square-root transform. The last three results are for parameter values obtained using our new transform. Classification is more successful for the parameter values given by our new results, with $\frac{a}{2b}=2$ matching the performance of the SRVF. Figure \ref{fig:classification2} shows the classification rates for $\frac{a}{2b}=1$ and $\frac{a}{2b}=2$ across the 40 individual signature classes. Although these parameter values have the same overall performance, we see that their performances differ drastically by signature class. This suggests that certain parameter values may be more sensitive to features appearing in particular signature classes, and that it is in general be beneficial to vary the choice of parameters to suit a given task.

\begin{figure}[!t]
\begin{center}
\begin{tabular}{|c|c|c|c|}
\hline
$\frac{a}{2b}$=1.25&$\frac{a}{2b}$=1.00&$\frac{a}{2b}$=0.50&$\frac{a}{2b}$=0.17\\
\hline
\includegraphics[width=.75in]{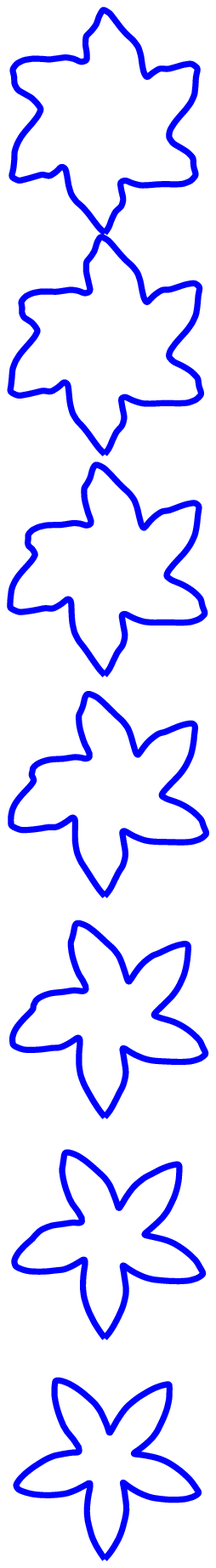}&\includegraphics[width=.75in]{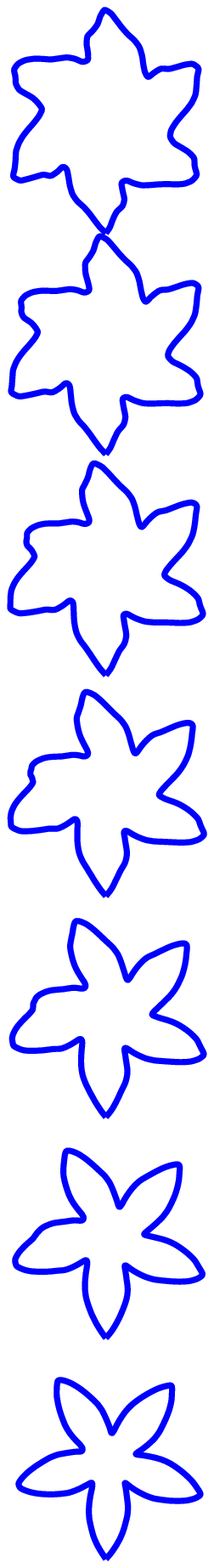}&\includegraphics[width=.75in]{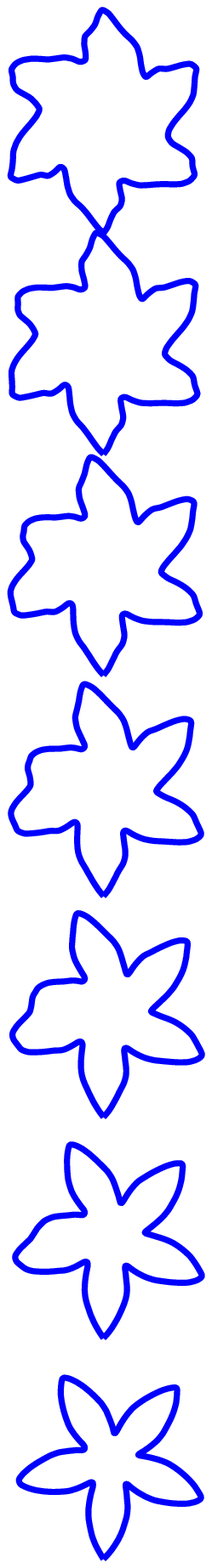}&\includegraphics[width=.75in]{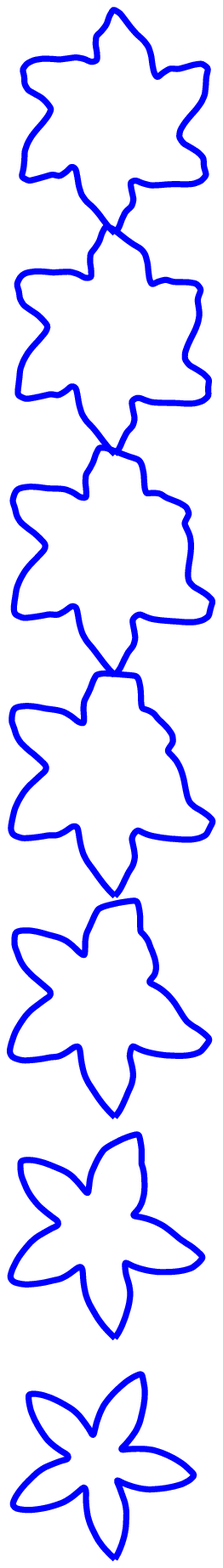}\\
\hline
\includegraphics[width=1in]{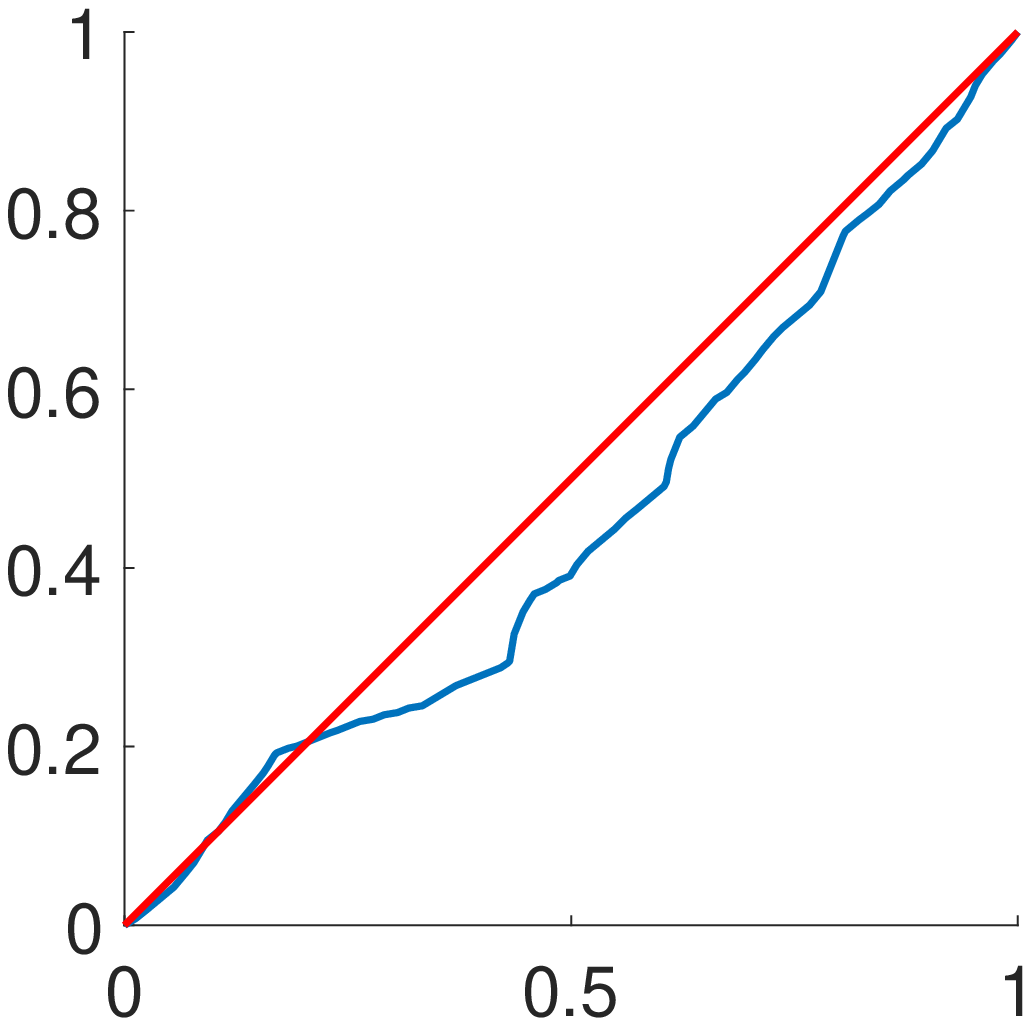}&\includegraphics[width=1in]{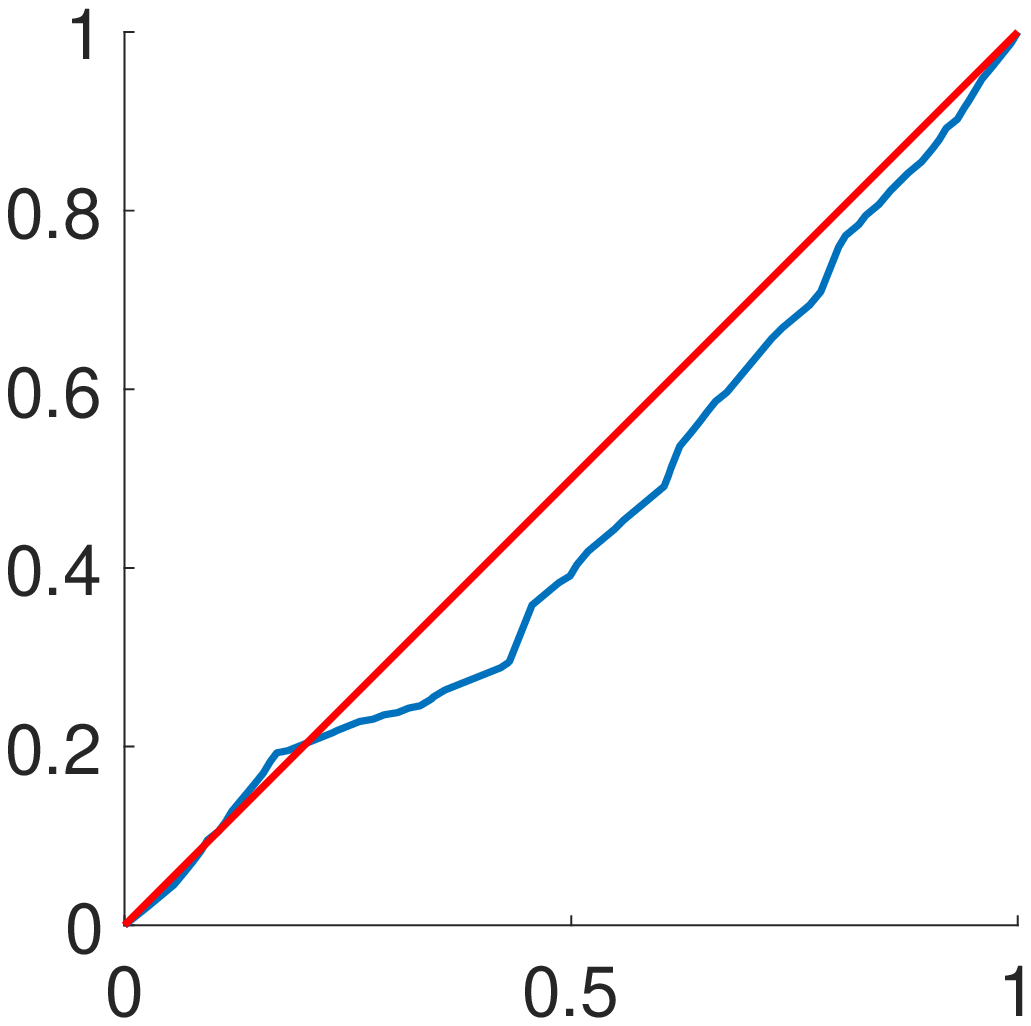}&\includegraphics[width=1in]{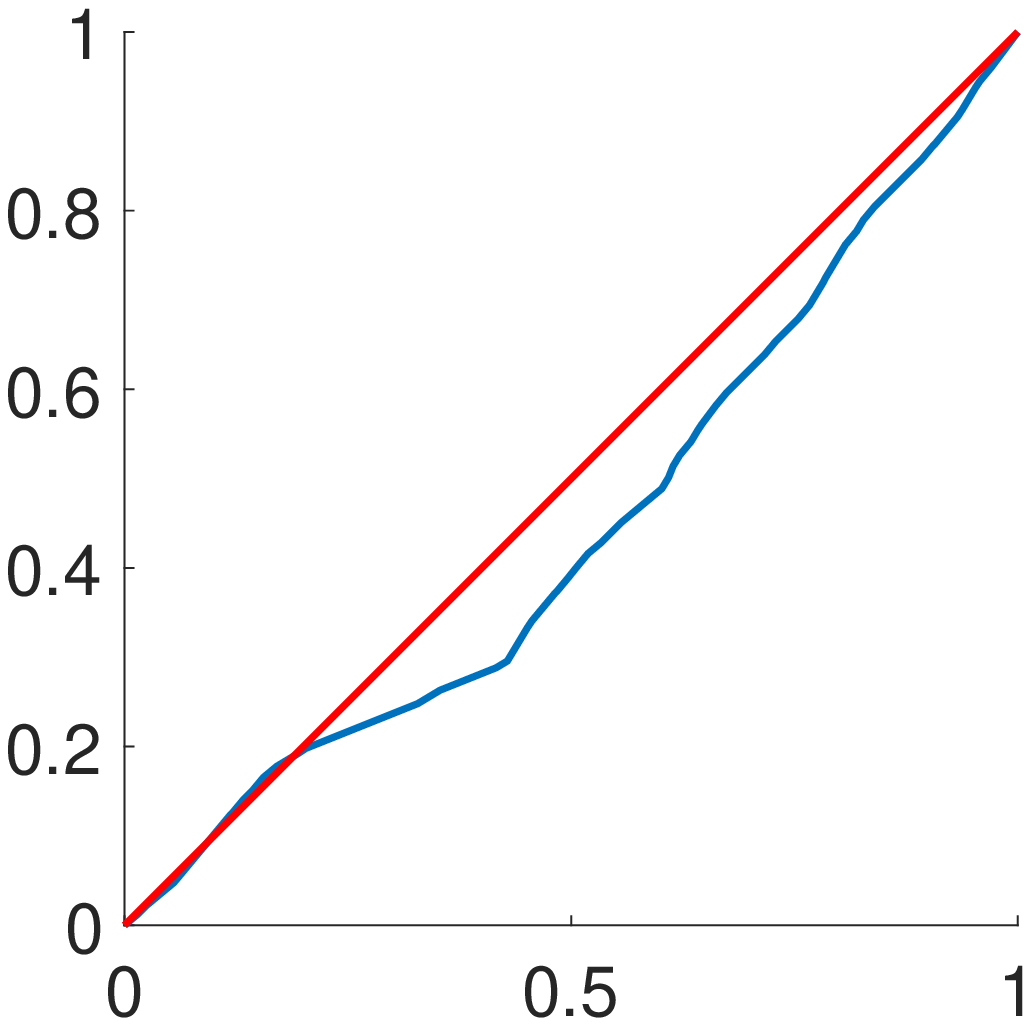}&\includegraphics[width=1in]{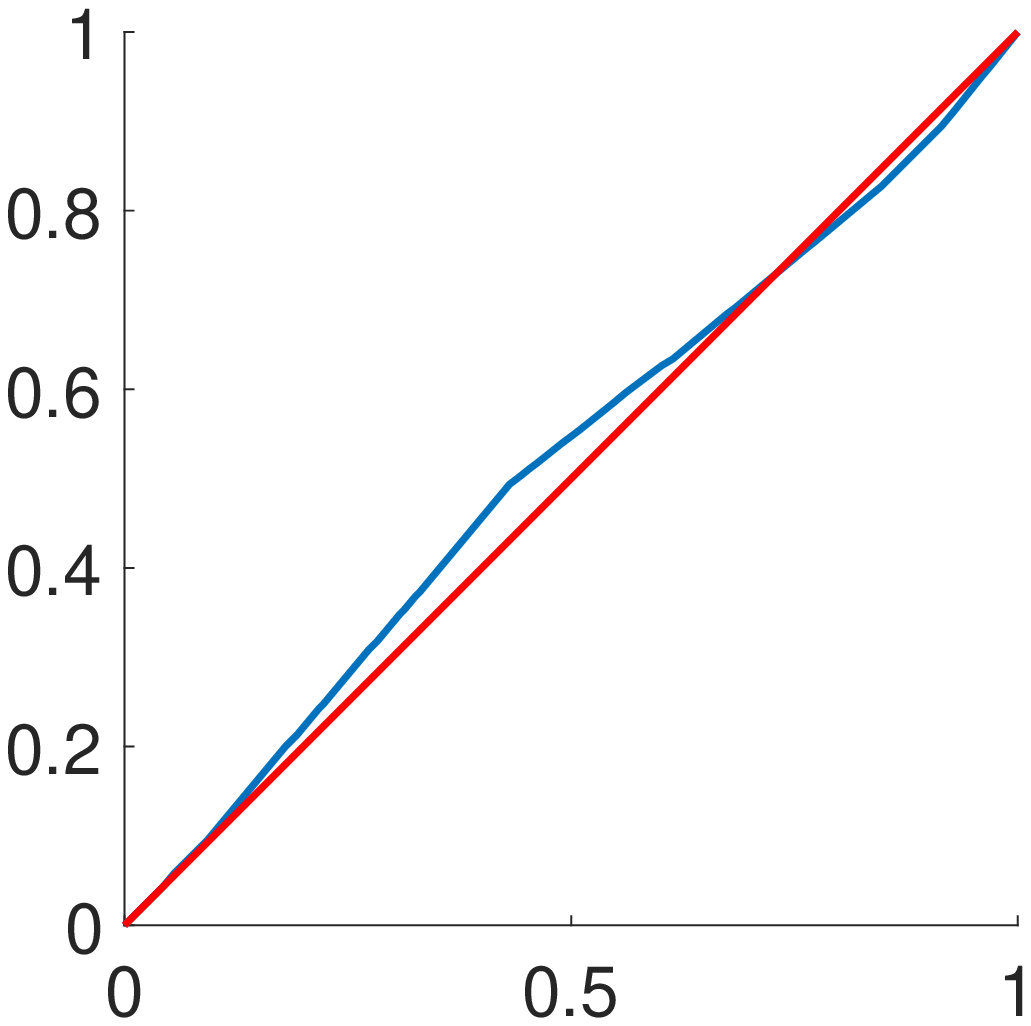}\\
\hline
dist=0.4580&dist=0.5028&dist=0.6732&dist=1.0576\\
\hline
\end{tabular}
\end{center}
\caption{Geodesic path and distance between two shapes of flowers with a different number of petals.}\label{fig:closed_geodesic2}
\label{fig:closedex2}
\end{figure}


\begin{table}[!t]
 \begin{tabular}{|c c c|}
 \hline
Method & Classif. Rate & Perfect Matches  \\ [0.5ex]
 \hline\hline
 Arclength & $86.44$ & 1 \\
 \hline \hline
  $\frac{a}{2b}=\frac{1}{4}$ & $86.50 $ & 0 \\
 \hline
 $\frac{a}{2b}=\frac{1}{2}$ & $93.81 $ & 6 \\
 \hline
   $\frac{a}{2b}=1$ & $97.50 $ & 19  \\
   \hline \hline
     $\frac{a}{2b}=2$ & $97.50 $ & 19  \\
      \hline
   $\frac{a}{2b}=3$ & $96.69 $ & 16  \\
 \hline
   $\frac{a}{2b}=4$ & $97.00 $ & 16  \\
 \hline
\end{tabular}
\caption{Overall classification rates and number of perfect classifications for the signature experiment.}\label{fig:classfication1}
\end{table}

\begin{figure}[!t]
\includegraphics[scale=0.2]{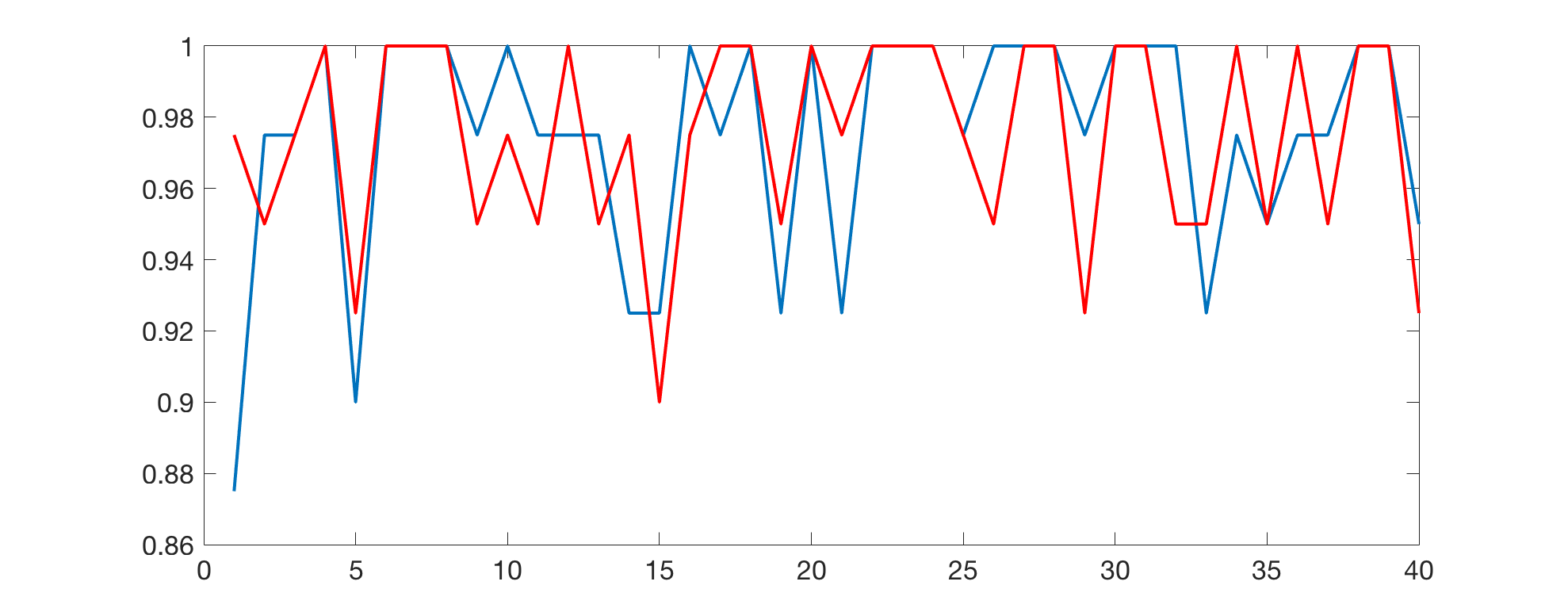}
\caption{Classification rates across 40 individual signature classes using the geodesic distance for parameters $\frac{a}{2b} =1$ (in blue) and $\frac{a}{2b} = 2$ (in red).}\label{fig:classification2}
\end{figure}

\section{Future Directions}\label{sec:conclusion}

Our first direction for future work is to develop various statistical tools under this framework. These will include computation of summary statistics such as the average and covariance of a set of shapes, exploration of variability in various shape classes through PCA, building generative shape models, inference via hypothesis testing and confidence intervals, and finally regression analysis. Given the simplification of the metric under the proposed $F_{a,b}$ transform, and the simple geometry of the preshape space, we will be able to build on existing work in this area based on the SRVF transform (i.e., $\frac{a}{2b}=1$). We have seen in the presented examples that different choices of $a$ and $b$ produce different geodesic paths, and thus, will result in different statistical analyses.

Second, we will build statistical models that allow the data to choose appropriate values of $a$ and $b$ for the given application and task. For example, in the context of classification, one can learn optimal weights on training data and then apply the proposed framework on a held-out set. Furthermore, one can build hierarchical Bayesian models for registration, comparison and averaging of shapes of planar curves that include priors on the values of $a$ and $b$. Such models can be developed in a similar manner to the functional data approaches in \cite{SK, YL}. In those works, the authors simply work with fixed values of $a$ and $b$. We propose to extend those methods by additionally including the weights for stretching and bending in the posterior ditribution.

In previous work, one of the authors extended the results of Younes et al. \cite{younes2008metric} to give a metric with explicit geodesics on the space of closed loops in $\R^3$ \cite{needham2017kahler}. This is accomplished by replacing the complex squaring map with the Hopf map $S^3 \rightarrow S^2$. Using quaternionic arithmetic, we expect that the results of this paper can be extended to provide transforms which simplify metrics on space curves as well.

\bibliographystyle{plain}

\bibliography{needham_bibliography}

\end{document}